\documentclass[12pt,reqno]{amsart}
\usepackage{verbatim}
\usepackage[utf8]{inputenc}
\usepackage{hyperref, amsmath, graphicx, enumitem, amsthm, amsfonts, ytableau, latexsym, amssymb}
\usepackage{tikz, xifthen, placeins, afterpage}

\numberwithin{equation}{section}

\catcode`\@=11
\@namedef{subjclassname@2020}{
  \textup{2020} Mathematics Subject Classification}
\catcode`\@=12

\theoremstyle{plain}
\newtheorem{thm}{Theorem}[section]
\newtheorem{lem}[thm]{Lemma}
\newtheorem{cor}[thm]{Corollary}
\newtheorem{prop}[thm]{Proposition}

\theoremstyle{definition}
\newtheorem{defn}[thm]{Definition}

\newtheorem{rem}[thm]{Remark}

\title{Domino tilings of generalized Aztec triangles}
\date{\today}
\author[Sylvie Corteel]{Sylvie Corteel$^\dagger$}

\address{CNRS, Laboratoire IRIF,
Universit\'e Paris
Cit\'e,
F-75205 Paris cedex 13, France.
WWW: {\tt https://www.irif.fr/\lower0.5ex\hbox{\~{}}corteel/}.}

\author[Frederick Huang]{Frederick Huang$^\dagger$}
\address{Department of Mathematics, University of California Berkeley, 
Evans Hall, Berkeley, CA 94720, USA}

\author{Christian Krattenthaler}

\address{Fakult\"at f\"ur Mathematik, Universit\"at Wien,
Oskar-Morgenstern-Platz~1, A-1090 Vienna, Austria.
WWW: {\tt http://www.mat.univie.ac.at/\lower0.5ex\hbox{\~{}}kratt}.}

\thanks{$^\dagger$SC and FH are partially funded by NSF grant DMS-2054482. SC is partially funded  by ANR grant ANR-19-CE48-0011.}

\begin{document}
    \ytableausetup{centertableaux}

\begin{abstract}
Di~Francesco [{\it Electron. J. Combin.} {\bf 28}.4 (2021), Paper No.~4.38]
introduced Aztec triangles as combinatorial objects for which
their domino tilings are equinumerous with certain sets of configurations of the
twenty-vertex model that are the main focus of his article.
We generalize Di~Francesco's construction of Aztec triangles.  
While we do not know whether there is again a correspondence with configurations
in the twenty-vertex model, we prove closed-form product formulas for the number
of domino tilings of our generalized Aztec triangles. As a special case, we
obtain a proof of Di~Francesco's conjectured formula for the number of domino
tilings of his Aztec triangles, and thus for the number of the corresponding
configurations in the twenty-vertex model.
\end{abstract}



    \maketitle
    
    \section{Introduction}
    
    Recently Di~Francesco and Guitter \cite{difrancesco1} uncovered interesting combinatorics regarding the twenty vertex (20V) model, which previously was of interest to physicists
(cf.\ \cite{20V1} or \cite{20V2}). The 20V model considers all orientations on the edges on some finite portion of a triangular lattice, shown in Figure~\ref{20V model}, such that at each vertex the number of ingoing edges is equal to the number of outgoing edges. Since each such vertex is adjacent to six edges in the triangular lattice, of which three must be ingoing and three must be outgoing, we see that there are exactly twenty possible configurations of edges at a vertex --- hence the name twenty vertex model. The 20V model is analogous to the six vertex (6V) 
model (cf.\ \cite[Ch.~8]{20V1} or \cite{6V}), which uses a square lattice (with six possible configurations of edges at a vertex). This 6V model is famously used
in a proof of the alternating sign matrix (ASM) theorem~\cite{ASM}, i.e., that the number of $n \times n$ ASMs is given by $\prod_{i=0}^{n-1} \frac{(3i+1)!}{(n+i)!}$. ASMs have long been of interest to combinatorialists, with the first proof of this theorem regarding their enumeration in \cite{ASM2}.

In \cite{difrancesco1}, Di~Francesco and Guitter show that the number of configurations of the 20V model on a square domain (with particular boundary conditions) can be counted by (a weighted enumeration of) configurations of the 6V model on an analogous domain (in fact, the ones from the ASM theorem). Later on, by considering 20V models on another domain, named an Aztec triangle, an example of which is in Figure~\ref{df domain}, Di~Francesco \cite{difrancesco2} also conjectures that 20V models on these domains are counted by
\begin{equation}
2^{n(n-1)/2}\prod_{i=0}^{n-1} \frac{(4i+2)!}{(n+2i+1)!},
\label{eq:DFconj}
\end{equation}
a formula which strikingly resembles the one from the ASM theorem.
The formula has been recently proved by Koutschan~\cite{KoutAB} using
Zeilberger's holonomic Ansatz~\cite{Zeilholo}
for proving determinant evaluations
(more or less) automatically with the help of a computer.
    
In the same paper \cite{difrancesco2}, Di~Francesco establishes that configurations of the 20V model on this expanded domain are equal in number to domino tilings of a particular domain, one such example shown in Figure~\ref{df domain}.
There is no combinatorial bijection between these two sets of equal size, though. On the other hand,
it is well known that domino tilings
are in bijection with families of non-intersecting paths and,
by using the Lindstr\"om--Gessel--Viennot
lemma, Di~Francesco proves that the number of such tilings can be given
in terms of determinants.
    \begin{thm}[{\cite[Theorems~3.4 and~8.2]{difrancesco2}}]
The number of domino tilings of the Aztec triangle of size $n$ equals
    \begin{align*}
Z_n&=\det_{0\le i,j\le n-1}\left([u^iv^j]\frac{1+u}{1 - v - 4uv - u^2v + u^2v^2}\right)\\
&=\det_{0\le i,j\le n-1}\left(2^i\binom {i+2j-1}{2j+1}-\binom {i-1}{2j+1}\right).
    \end{align*}
    \end{thm}
Indeed, for his proof of \eqref{eq:DFconj}, Koutschan evaluated the latter
determinant.
    
The present paper interprets
these domino tilings as sequences of $2n$ partitions subject to certain conditions where the final partition is $(n, n-1, \ldots, 1)$, like those seen in \cite{steep}.
More generally, we
construct domains, which we call generalized Aztec triangles, for sequences where the final partition is any partition $\mu$ and the number of partitions in the sequence is $\ell+1$ for $\ell\ge 0$. From
there we provide a number of ways to enumerate the domino tilings for our domains and we show that there is a bijection between these domino tilings and a family of tableaux that are both super and symplectic. 
Furthermore we provide a generalization
of Di~Francesco's formula when the final partition is $\mu=(k, k-1, \ldots, 1, 0^{n-k})$.
    \begin{thm} \label{thm:1}
Let $\ell$ and $k$  be positive integers and $\mu=(k, k-1, \ldots, 1, 0^{n-k})$
with $n=\lfloor \ell/2\rfloor$.
Then the number of domino tilings of the generalized Aztec triangle with final partition $\mu$ and parameter $\ell$ equals
    \begin{equation}
    \left. \prod_{j\ge 0} \left( \prod_{i=-2k+4j+1}^{-k+2j}(\ell+i)\prod_{i=k-2j}^{2k-4j-2}(\ell+i) \right) \middle/ \prod_{j=1}^{k-1} (2j+1)^{k-j} \right.
.
    \label{cidessus}
    \end{equation}
    \end{thm}

See Theorem~\ref{formula} for a precise statement.
Di~Francesco's conjecture corresponds to the case $\mu=(n,n-1,\ldots ,2,1)$ and $\ell=2n$.
Consequently, we obtain the following corollary.
    \begin{cor}
    The number of 20V configurations on an Aztec triangle of size $n$ is
    \[
    2^{n(n-1)/2}\prod_{i=0}^{n-1} \frac{(4i+2)!}{(n+2i+1)!}. 
    \]
    \label{corDIF}
\end{cor}
There is in fact another family, namely for $\mu=(n-1,n-2,\ldots ,1,0)$ and $\ell=2n$,
which is counted by the same formula. 

\begin{thm}
Domino tilings of the generalized Aztec triangle with final partition
$\mu=(n,n-1,\ldots ,1)$ and parameter $\ell=2n$ are equinumerous with the domino
tilings of the generalized Aztec triangle with final partition
$\mu=(n-1,n-2,\ldots ,1,0)$ and parameter $\ell=2n+1$.
\label{thmbij}
\end{thm}
See Theorem~\ref{thm:DD} for an equivalent statement.
We leave it as an open problem to find a bijection between
these two families of domino tilings.
    
    \begin{figure}[ht]
        \centering
        \begin{tikzpicture}[scale=.7]
            \draw (0,0) grid (5,5);
            \draw[color=white,ultra thick] (0,0) rectangle (5,5);
            \foreach\i in {2,...,5} {
                \draw (0,\i) -- (\i,0);
                \draw (5-\i,5) -- (5,5-\i);
            }
        \end{tikzpicture} \qquad
        \begin{tikzpicture}[scale=.7]
            \draw (0,0) grid (5,5);
            \draw[color=white,ultra thick] (0,0) rectangle (5,5);
            \foreach\i in {2,...,5} {
                \draw (0,\i) -- (5,\i-5);
                \draw (\i-1,2-\i) -- (5,2-\i);
                \draw (\i-1,0) -- (\i-1,2-\i);
                \draw (5-\i,5) -- (5,5-\i);
            }
        \end{tikzpicture}
        \caption{On the left
a domain of the 20V model considered in \cite{difrancesco1} is shown, in which configurations of various boundary conditions are counted. On the right is an example of another domain of interest, which is considered in \cite{difrancesco2}.}
        \label{20V model}
    \end{figure}
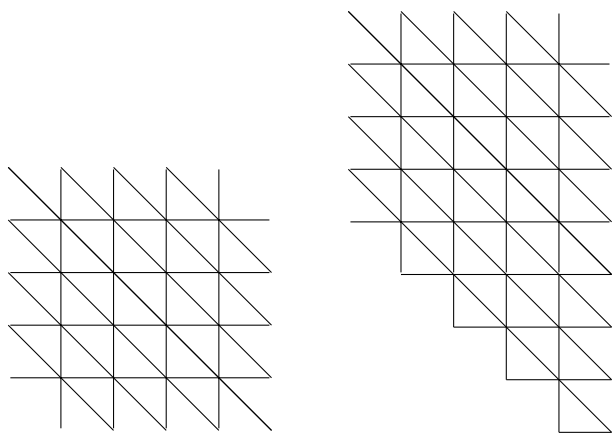

    \begin{figure}[ht]
        \centering
        \begin{tikzpicture}[scale=.5]
            \draw[ultra thick] (0,7) rectangle (1,8);
\draw[ultra thick] (1,8) rectangle (2,9);
\draw[ultra thick] (2,9) rectangle (3,10);
\draw[ultra thick] (3,10) rectangle (4,11);
\draw[ultra thick] (0,6) rectangle (1,7);
\draw[ultra thick] (1,7) rectangle (2,8);
\draw[ultra thick] (2,8) rectangle (3,9);
\draw[ultra thick] (3,9) rectangle (4,10);
\draw[ultra thick] (4,10) rectangle (5,11);
\draw[ultra thick] (0,5) rectangle (1,6);
\draw[ultra thick] (1,6) rectangle (2,7);
\draw[ultra thick] (2,7) rectangle (3,8);
\draw[ultra thick] (3,8) rectangle (4,9);
\draw[ultra thick] (4,9) rectangle (5,10);
\draw[ultra thick] (0,4) rectangle (1,5);
\draw[ultra thick] (1,5) rectangle (2,6);
\draw[ultra thick] (2,6) rectangle (3,7);
\draw[ultra thick] (3,7) rectangle (4,8);
\draw[ultra thick] (4,8) rectangle (5,9);
\draw[ultra thick] (5,9) rectangle (6,10);
\draw[ultra thick] (0,3) rectangle (1,4);
\draw[ultra thick] (1,4) rectangle (2,5);
\draw[ultra thick] (2,5) rectangle (3,6);
\draw[ultra thick] (3,6) rectangle (4,7);
\draw[ultra thick] (4,7) rectangle (5,8);
\draw[ultra thick] (5,8) rectangle (6,9);
\draw[ultra thick] (0,2) rectangle (1,3);
\draw[ultra thick] (1,3) rectangle (2,4);
\draw[ultra thick] (2,4) rectangle (3,5);
\draw[ultra thick] (3,5) rectangle (4,6);
\draw[ultra thick] (4,6) rectangle (5,7);
\draw[ultra thick] (5,7) rectangle (6,8);
\draw[ultra thick] (6,8) rectangle (7,9);
\draw[ultra thick] (0,1) rectangle (1,2);
\draw[ultra thick] (1,2) rectangle (2,3);
\draw[ultra thick] (2,3) rectangle (3,4);
\draw[ultra thick] (3,4) rectangle (4,5);
\draw[ultra thick] (4,5) rectangle (5,6);
\draw[ultra thick] (5,6) rectangle (6,7);
\draw[ultra thick] (6,7) rectangle (7,8);
\draw[ultra thick] (0,0) rectangle (1,1);
\draw[ultra thick] (2,2) rectangle (3,3);
\draw[ultra thick] (4,4) rectangle (5,5);
\draw[ultra thick] (6,6) rectangle (7,7);
\draw[fill=white,draw=gray] (0,7) rectangle (1,8);
\draw[fill=white,draw=gray] (1,8) rectangle (2,9);
\draw[fill=white,draw=gray] (2,9) rectangle (3,10);
\draw[fill=white,draw=gray] (3,10) rectangle (4,11);
\draw[fill=white,draw=gray] (0,6) rectangle (1,7);
\draw[fill=white,draw=gray] (1,7) rectangle (2,8);
\draw[fill=white,draw=gray] (2,8) rectangle (3,9);
\draw[fill=white,draw=gray] (3,9) rectangle (4,10);
\draw[fill=white,draw=gray] (4,10) rectangle (5,11);
\draw[fill=white,draw=gray] (0,5) rectangle (1,6);
\draw[fill=white,draw=gray] (1,6) rectangle (2,7);
\draw[fill=white,draw=gray] (2,7) rectangle (3,8);
\draw[fill=white,draw=gray] (3,8) rectangle (4,9);
\draw[fill=white,draw=gray] (4,9) rectangle (5,10);
\draw[fill=white,draw=gray] (0,4) rectangle (1,5);
\draw[fill=white,draw=gray] (1,5) rectangle (2,6);
\draw[fill=white,draw=gray] (2,6) rectangle (3,7);
\draw[fill=white,draw=gray] (3,7) rectangle (4,8);
\draw[fill=white,draw=gray] (4,8) rectangle (5,9);
\draw[fill=white,draw=gray] (5,9) rectangle (6,10);
\draw[fill=white,draw=gray] (0,3) rectangle (1,4);
\draw[fill=white,draw=gray] (1,4) rectangle (2,5);
\draw[fill=white,draw=gray] (2,5) rectangle (3,6);
\draw[fill=white,draw=gray] (3,6) rectangle (4,7);
\draw[fill=white,draw=gray] (4,7) rectangle (5,8);
\draw[fill=white,draw=gray] (5,8) rectangle (6,9);
\draw[fill=white,draw=gray] (0,2) rectangle (1,3);
\draw[fill=white,draw=gray] (1,3) rectangle (2,4);
\draw[fill=white,draw=gray] (2,4) rectangle (3,5);
\draw[fill=white,draw=gray] (3,5) rectangle (4,6);
\draw[fill=white,draw=gray] (4,6) rectangle (5,7);
\draw[fill=white,draw=gray] (5,7) rectangle (6,8);
\draw[fill=white,draw=gray] (6,8) rectangle (7,9);
\draw[fill=white,draw=gray] (0,1) rectangle (1,2);
\draw[fill=white,draw=gray] (1,2) rectangle (2,3);
\draw[fill=white,draw=gray] (2,3) rectangle (3,4);
\draw[fill=white,draw=gray] (3,4) rectangle (4,5);
\draw[fill=white,draw=gray] (4,5) rectangle (5,6);
\draw[fill=white,draw=gray] (5,6) rectangle (6,7);
\draw[fill=white,draw=gray] (6,7) rectangle (7,8);
\draw[fill=white,draw=gray] (0,0) rectangle (1,1);
\draw[fill=white,draw=gray] (2,2) rectangle (3,3);
\draw[fill=white,draw=gray] (4,4) rectangle (5,5);
\draw[fill=white,draw=gray] (6,6) rectangle (7,7);
            \draw[fill=white] (0.5,0.5) circle[radius=.1];
            \draw[fill=white] (2.5,2.5) circle[radius=.1];
            \draw[fill=white] (4.5,4.5) circle[radius=.1];
            \draw[fill=white] (6.5,6.5) circle[radius=.1];
            \draw[fill=black] (1.5,1.5) circle[radius=.1];
            \draw[fill=black] (3.5,3.5) circle[radius=.1];
            \draw[fill=black] (5.5,5.5) circle[radius=.1];
            \draw[fill=black] (7.5,7.5) circle[radius=.1];
        \end{tikzpicture} \quad \begin{tikzpicture}[scale=.5]
            \draw[ultra thick] (0,7) rectangle (1,8);
\draw[ultra thick] (1,8) rectangle (2,9);
\draw[ultra thick] (2,9) rectangle (3,10);
\draw[ultra thick] (0,6) rectangle (1,7);
\draw[ultra thick] (1,7) rectangle (2,8);
\draw[ultra thick] (2,8) rectangle (3,9);
\draw[ultra thick] (3,9) rectangle (4,10);
\draw[ultra thick] (0,5) rectangle (1,6);
\draw[ultra thick] (1,6) rectangle (2,7);
\draw[ultra thick] (2,7) rectangle (3,8);
\draw[ultra thick] (3,8) rectangle (4,9);
\draw[ultra thick] (0,4) rectangle (1,5);
\draw[ultra thick] (1,5) rectangle (2,6);
\draw[ultra thick] (2,6) rectangle (3,7);
\draw[ultra thick] (3,7) rectangle (4,8);
\draw[ultra thick] (4,8) rectangle (5,9);
\draw[ultra thick] (0,3) rectangle (1,4);
\draw[ultra thick] (1,4) rectangle (2,5);
\draw[ultra thick] (2,5) rectangle (3,6);
\draw[ultra thick] (3,6) rectangle (4,7);
\draw[ultra thick] (4,7) rectangle (5,8);
\draw[ultra thick] (0,2) rectangle (1,3);
\draw[ultra thick] (1,3) rectangle (2,4);
\draw[ultra thick] (2,4) rectangle (3,5);
\draw[ultra thick] (3,5) rectangle (4,6);
\draw[ultra thick] (4,6) rectangle (5,7);
\draw[ultra thick] (5,7) rectangle (6,8);
\draw[ultra thick] (0,1) rectangle (1,2);
\draw[ultra thick] (1,2) rectangle (2,3);
\draw[ultra thick] (2,3) rectangle (3,4);
\draw[ultra thick] (3,4) rectangle (4,5);
\draw[ultra thick] (4,5) rectangle (5,6);
\draw[ultra thick] (5,6) rectangle (6,7);
\draw[ultra thick] (0,0) rectangle (1,1);
\draw[ultra thick] (1,1) rectangle (2,2);
\draw[ultra thick] (2,2) rectangle (3,3);
\draw[ultra thick] (3,3) rectangle (4,4);
\draw[ultra thick] (4,4) rectangle (5,5);
\draw[ultra thick] (5,5) rectangle (6,6);
\draw[ultra thick] (6,6) rectangle (7,7);
\draw[ultra thick] (0,-1) rectangle (1,0);
\draw[ultra thick] (2,1) rectangle (3,2);
\draw[ultra thick] (4,3) rectangle (5,4);
\draw[ultra thick] (6,5) rectangle (7,6);
\draw[fill=white,draw=gray] (0,7) rectangle (1,8);
\draw[fill=white,draw=gray] (1,8) rectangle (2,9);
\draw[fill=white,draw=gray] (2,9) rectangle (3,10);
\draw[fill=white,draw=gray] (0,6) rectangle (1,7);
\draw[fill=white,draw=gray] (1,7) rectangle (2,8);
\draw[fill=white,draw=gray] (2,8) rectangle (3,9);
\draw[fill=white,draw=gray] (3,9) rectangle (4,10);
\draw[fill=white,draw=gray] (0,5) rectangle (1,6);
\draw[fill=white,draw=gray] (1,6) rectangle (2,7);
\draw[fill=white,draw=gray] (2,7) rectangle (3,8);
\draw[fill=white,draw=gray] (3,8) rectangle (4,9);
\draw[fill=white,draw=gray] (0,4) rectangle (1,5);
\draw[fill=white,draw=gray] (1,5) rectangle (2,6);
\draw[fill=white,draw=gray] (2,6) rectangle (3,7);
\draw[fill=white,draw=gray] (3,7) rectangle (4,8);
\draw[fill=white,draw=gray] (4,8) rectangle (5,9);
\draw[fill=white,draw=gray] (0,3) rectangle (1,4);
\draw[fill=white,draw=gray] (1,4) rectangle (2,5);
\draw[fill=white,draw=gray] (2,5) rectangle (3,6);
\draw[fill=white,draw=gray] (3,6) rectangle (4,7);
\draw[fill=white,draw=gray] (4,7) rectangle (5,8);
\draw[fill=white,draw=gray] (0,2) rectangle (1,3);
\draw[fill=white,draw=gray] (1,3) rectangle (2,4);
\draw[fill=white,draw=gray] (2,4) rectangle (3,5);
\draw[fill=white,draw=gray] (3,5) rectangle (4,6);
\draw[fill=white,draw=gray] (4,6) rectangle (5,7);
\draw[fill=white,draw=gray] (5,7) rectangle (6,8);
\draw[fill=white,draw=gray] (0,1) rectangle (1,2);
\draw[fill=white,draw=gray] (1,2) rectangle (2,3);
\draw[fill=white,draw=gray] (2,3) rectangle (3,4);
\draw[fill=white,draw=gray] (3,4) rectangle (4,5);
\draw[fill=white,draw=gray] (4,5) rectangle (5,6);
\draw[fill=white,draw=gray] (5,6) rectangle (6,7);
\draw[fill=white,draw=gray] (0,0) rectangle (1,1);
\draw[fill=white,draw=gray] (1,1) rectangle (2,2);
\draw[fill=white,draw=gray] (2,2) rectangle (3,3);
\draw[fill=white,draw=gray] (3,3) rectangle (4,4);
\draw[fill=white,draw=gray] (4,4) rectangle (5,5);
\draw[fill=white,draw=gray] (5,5) rectangle (6,6);
\draw[fill=white,draw=gray] (6,6) rectangle (7,7);
\draw[fill=white,draw=gray] (0,-1) rectangle (1,0);
\draw[fill=white,draw=gray] (2,1) rectangle (3,2);
\draw[fill=white,draw=gray] (4,3) rectangle (5,4);
\draw[fill=white,draw=gray] (6,5) rectangle (7,6);
            \draw[fill=white] (1.5,0.5) circle[radius=.1];
\draw[fill=white] (3.5,2.5) circle[radius=.1];
\draw[fill=white] (5.5,4.5) circle[radius=.1];
\draw[fill=black] (0.5,-0.5) circle[radius=.1];
\draw[fill=black] (2.5,1.5) circle[radius=.1];
\draw[fill=black] (4.5,3.5) circle[radius=.1];
\draw[fill=black] (6.5,5.5) circle[radius=.1];            
        \end{tikzpicture}
        \caption{On the left, an example of a domain considered by Di~Francesco in \cite{difrancesco2} whose domino tilings are equal in number to 20V configurations on the domain on the right in Figure~\ref{20V model}. On the right, a domain whose domino tilings are
equinumerous with domino tilings of the domain on the left, as in Theorem~\ref{thmbij} for $n=4$.}
        \label{df domain}
    \end{figure}
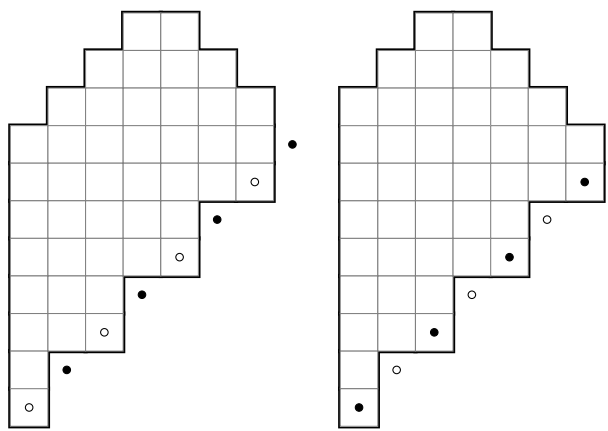

The paper is organized as follows. In Section~\ref{prelim}, we define all the combinatorial models and 
notions. In Section~\ref{objects} we define the sequences of partitions, the tableaux, and the path systems in bijection with tilings.
Subsequently, in Section~\ref{conj}, we state our main (enumeration) results,
namely product formulas for the number of domino tilings of our generalized
Aztec triangles; see Theorem~\ref{formula}
(which is equivalent with Theorem~\ref{thm:1}) and Theorem~\ref{thm:DD}
(which is equivalent with Theorem~\ref{thmbij}). While we provide a proof of
Theorem~\ref{thm:DD} in the same section
(under the condition that Theorem~\ref{formula} holds),
the proof of the other theorem requires considerably more work.
In preparation of this
proof, which we give in Section~\ref{mainproof}, we rewrite the product formulas in a form
as needed in the latter section, and we recall the determinantal formulas
from Section~\ref{objects} for the domino tilings. It turns out that, to prove
Theorem~\ref{formula}, we need to evaluate two determinants of Delannoy numbers.
However, we show in Theorem~\ref{D1=D2} that there is a simple relation between
the two determinants so that it suffices to evaluate one of the two.
This is what we do in Theorem~\ref{thm:detD1} in Section~\ref{mainproof}.
We close our paper with a conclusion section, where we provide an
overview of open questions and open ends to be explored at these
crossroads of domino tilings and 20V model configurations.

    \FloatBarrier
    
    \section{Preliminaries}
    \label{prelim}
    \subsection{Sequences of partitions}
    \begin{defn}
      A \textit{partition} $\mu = (\mu_1,\ldots,\mu_n)$ is a sequence of non-negative integers such that $\mu_1 \geq \mu_2 \geq \cdots \geq \mu_n$. We call $\mu_1, \ldots, \mu_n$ the \textit{parts} of the partition $\mu$, and $n$ the \textit{length} of the partition $\mu$. Any partition of length $n$ can also be considered
as a partition of length $m$ for any $m \geq n$ (or even $m$ infinite) by appending
0's at the end of the partition.
        
        We also define the \textit{conjugate} partition of $\mu$ to be the partition $\mu'$ of length $\mu_1$ such that $\mu'_i = \# \lbrace \mu_j \mid \mu_j \geq i \rbrace$.
        
        A \textit{Young diagram} is a collection of boxes arranged in rows such that the rows are left justified and non-decreasing from bottom to top. We can visualize a partition $\mu$ as a Young diagram with rows of lengths equal to the parts of $\mu$,
as illustrated in Figure~\ref{young}.
    \end{defn}
    
    \begin{figure}[ht]
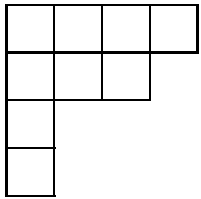

        \centering
        \begin{ytableau}
            \none & & & & \\
            \none & & & & \none \\
            \none & & \none & \none & \none \\
            \none & & \none & \none & \none
        \end{ytableau}
        \caption{The Young diagram corresponding to the partition $\mu = (4,3,1,1)$. }
        \label{young}
    \end{figure}
    
    \begin{defn}
      For partitions $\lambda, \mu$ of equal length such that $\lambda_i \geq \mu_i$ for all~$i$, we can define the \textit{skew diagram}
      $\lambda / \mu$ to be the boxes in the Young diagram of $\lambda$ that are not in the Young diagram of $\mu$. A skew diagram $\lambda / \mu$ is called a \textit{horizontal strip} if
its Young diagram contains at most one box in each column. Similarly, it is called a \textit{vertical strip} if $\lambda' / \mu'$ is a horizontal strip, or equivalently if its Young diagram contains at most one box in each row.
    \end{defn}
    
    In Figure~\ref{skew}, we illustrate a few skew diagrams, one of which is a horizontal strip and another a vertical strip.
    
    \begin{figure}[ht]
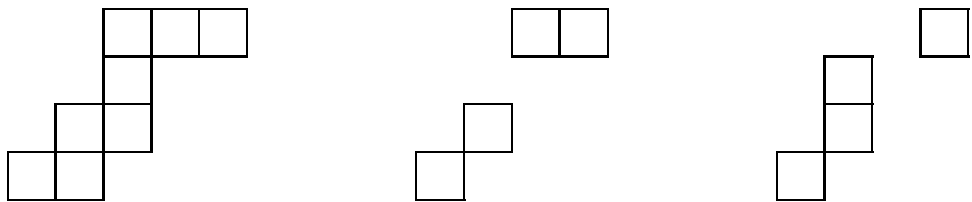

        \centering
        \begin{ytableau}
            \none & \none & \none & & & \\
            \none & \none & \none & & \none & \none \\
            \none & \none & & & \none & \none \\
            \none & & & \none & \none & \none
        \end{ytableau} \qquad \begin{ytableau}
            \none & \none & \none & \none & & \\
            \none & \none & \none & \none & \none & \none \\
            \none & \none & \none & & \none & \none \\
            \none & \none & & \none & \none & \none
        \end{ytableau} \qquad \begin{ytableau}
            \none & \none & \none & \none & \none & \\
            \none & \none & \none & & \none & \none \\
            \none & \none & \none & & \none & \none \\
            \none & \none & & \none & \none & \none
        \end{ytableau}
        \caption{Three skew diagrams $\lambda / \mu$ where $\lambda = (5,3,3,2)$. On the left we have $\mu = (2,2,1,0)$. In the
center, we have $\mu = (3,3,2,1)$ resulting in a horizontal strip. On the right we have $\mu = (4,2,2,1)$ resulting in a vertical strip.}
        \label{skew}
    \end{figure}

    \subsection{Delannoy paths}
    \begin{defn}
        For any $i,j \in \mathbb{Z}$ the \textit{Delannoy number} $D(i,j)$ is the number of paths from $(0,0)$ to $(i,j)$ using only north, northeast, or east steps. 
        \label{H1}
    \end{defn}
    
    \begin{rem}
The Delannoy numbers $D(i,j)$ can also be defined recursively by
        \begin{enumerate}
            \item $D(i,j) = 0$ if $i<0$ or $j<0$;
            \item $D(0,0) = 1$;
            \item $D(i,j) = D(i-1,j) + D(i,j-1) + D(i-1,j-1)$.
        \end{enumerate}
        \label{H1_rem}
    \end{rem}
    
We use the following standard formulas for Delannoy numbers
(cf.\ \cite[Ex.~21 in Ch.~I]{ComtAA} for information on Delannoy numbers;
the equalities are easily derived by standard generating function calculus):
\begin{align}
\notag
D(i,j)
&=[u^nv^m]\frac {1} {1-u-v-uv}\\
&=\sum_{\ell=0}^n(-1)^{n-\ell}\binom n\ell\binom {m+\ell}\ell2^\ell\\
\label{eq:f3}
&=\sum_{\ell=0}^i\binom i\ell\binom {j}\ell2^\ell\\
\label{eq:f4}
&=\sum_{\ell=0}^i\binom {i+j-\ell}{i-\ell,j-\ell,\ell}.
\end{align}
Above, and in the sequel, we use
\begin{equation}
\binom i\ell:=\begin{cases}
\displaystyle \frac {i(i-1)\cdots(i-\ell+1)} {\ell!},&\text{if }\ell\ge0,\\
0,&\text{if }\ell<0,
\end{cases}
\label{eq:binom}
\end{equation}
as definition for the binomial coefficients. In view of this,
the expression in~\eqref{eq:f3} tells us that $D(i,j)$ is a polynomial
in~$j$ of degree~$i$.
    
The Delannoy numbers will appear
in the enumeration of certain tilings. For others, we will need a slightly different family of numbers.    
        \begin{defn}
          For $i,j \in \mathbb{Z}$ we define numbers $H(i,j)$ by 
        \begin{enumerate}
            \item $H(i,j) = 0$ if $i<0$ or $j<0$, except for $(i,j) = (0,-1)$:
            \item $H(0,-1) = 1$;
            \item $H(i,j) = 2D(i-1, j) + H(i, j-1)=D(i,j)+D(i-1,j)$.
        \end{enumerate}
        \label{H2}
    \end{defn}

We can also interpret $H(i,j)$ as
number of paths --- $H(i,j)$ is the number of paths from $(0,0)$ to $(i,j+1)$ using north, northeast, or east steps that do not pass through $(i-1,j+1)$ as shown in Figure~\ref{Delannoy}. 
We call these $H-$Delannoy paths.
    A number of identities exist for $D$ and $H$.
    
    \begin{figure}[ht]
      \centering
        \begin{tikzpicture}[scale=.5]
            \draw (0,0) grid (5,4);
            \foreach \x in {0,...,4} {\draw (0,\x) -- (4-\x,4);}
            \foreach \x in {1,...,5} {\draw (\x,0) -- (5,5-\x);}
            \draw (0,0) node[left]{$(0,0)$};
            \draw (5,4) node[right]{$(i,j)$};
        \end{tikzpicture} \qquad
        \begin{tikzpicture}[scale=.5]
            \draw (0,0) grid (5,5);
            \foreach \x in {0,...,5} {\draw (0,\x) -- (5-\x,5);}
            \foreach \x in {1,...,5} {\draw (\x,0) -- (5,5-\x);}
            \draw (0,0) node[left]{$(0,0)$};
            \draw (5,4) node[right]{$(i,j)$};
            \draw (5,5) node[right]{$(i,j+1)$};
            \draw[fill=red,draw=red] (4,5) circle[radius=.1]; 
        \end{tikzpicture}
        \caption{On the left, the paths counted by $D(i,j)$ are those from $(0,0)$ to $(i,j)$ taking only north, northeast, and east steps. On the right, the paths counted by $H(i,j)$ are those from $(0,0)$ to $(i,j+1)$ using north, northeast, and east steps that do not pass through the point $(i-1,j+1)$ in red.}
        \label{Delannoy}
    \end{figure}
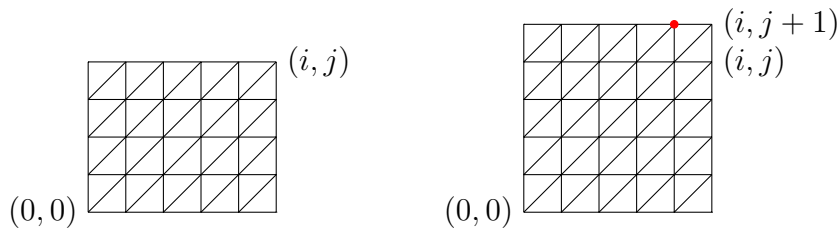
    
    \begin{prop}
        \label{H_identities}
For all non-negative integers $i,j$ we have
        \begin{align}
            D(i, j) &= D(i-1, j) + H(i, j-1)
            \label{eq:6}\\
            &= \sum_{\ell=0}^i H(\ell, j-1)
            \label{eq:7}\\
            &= \sum_\ell \binom{i}{\ell} \binom{j}{\ell} 2^\ell
            \label{eq:8}\\
            H(i, j) &= H(i-1,j) + H(i,j-1) + H(i-1,j-1)
            \label{eq:9}\\
            &= 2 \sum_{\ell=0}^{j-1} D(\ell, i)
            \label{eq:10}\\
            &= \sum_\ell \binom{i-1}{\ell-1} \binom{j+1}{\ell} 2^\ell
            \quad \text{ for } i \geq 1.
            \label{eq:11}
        \end{align}
    \end{prop}

    \begin{proof}
        If we consider the geometry of Definition~\ref{H1} and the observation above, equation \eqref{eq:6} is clear. Equation \eqref{eq:7} then follows by induction on $i$. To see equation \eqref{eq:8}, we first observe that $\binom{i}{m} \binom{j}{m}$ is the number of paths from $(0,0)$ to $(i,j)$ using only east and north steps that have exactly $m$ turns consisting of an east step followed by a north step. Then to count Delannoy paths, we obtain a factor of $2^m$ since at each of these turns we can instead take a northeast step.
        
Now equation \eqref{eq:6} gives us $H(i,j) = D(i,j+1) - D(i-1,j+1)$, and combined with
item~3.\ in Remark~\ref{H1_rem} we obtain equation \eqref{eq:9}. Equation \eqref{eq:10} follows from induction on $i$ using
item~3.\ in Definition~\ref{H2}. Finally the proof for equation \eqref{eq:11} follows from the argument for equation \eqref{eq:8}, noting that we are considering paths from $(0,0)$ to $(i,j+1)$ where the east-most column must contain one of the aforementioned east to north turns.
    \end{proof}

\subsection{Tableaux}

Given a partition $\mu=(\mu_1,\ldots ,\mu_n)$, a tableau of shape $\mu$ is a filling of the cells of the diagram of $\mu$
satisfying certain conditions.

We first recall classical tableau definitions.
    \begin{defn}[\cite{symp1,symp2,symp3}]
        A \textit{semistandard symplectic tableau}  of shape $\mu$ is a filling of boxes of the Young diagram corresponding to $\mu$ with entries $1 < \overline{1} < 2 < \overline{2} < \cdots$ such that:
        \begin{enumerate}[label=(\arabic*)]
            \item rows are weakly increasing and columns are strictly increasing;
            \item $i$ and $\overline{\imath}$ do not appear below the $i$th row.
        \end{enumerate}
        \label{defsymp}
    \end{defn}
 The first condition is referred to as the semistandard condition, and the second condition is the symplectic condition. Note that we can equivalently state the semi-standard condition of weakly increasing rows and strictly increasing columns as:
    \begin{enumerate}[label=(\roman*)]
        \item rows and columns are weakly increasing.
        \item for all~$i$, the entries~$i$ or $\overline{\imath}$ form horizontal strips.
    \end{enumerate}

\begin{defn}[\cite{kra96}]
        A \textit{super semistandard tableau} of shape $\mu$ is a filling of boxes of the Young diagram corresponding to $\mu$ with entries $1 < \overline{1} < 2 < \overline{2} < \cdots$ such that:
        \begin{enumerate}[label=(\arabic*)]
            \item rows and columns are weakly increasing;
             \item for all~$i$, the entries~$i$ form horizontal strips (i.e., there is at most one $i$ in each column);
            \item for all~$i$, the entries $\overline{\imath}$ form vertical strips (i.e., there is at most one $\overline{\imath}$ in each row).
        \end{enumerate}
    \label{defsuper}
    \end{defn}
    
\subsection{Domino tilings}

\begin{defn}
  Let $D$ be a region of the Euclidean plane that is the union of unit squares whose vertices lie in the integer lattice $\mathbb{Z}^2$. A \textit{domino} is the union of two such unit squares that share an edge. Then we
define a \textit{domino tiling} of $D$ to be a set $\mathcal{D}$ of mutually disjoint
(meaning: non-overlapping)
dominoes such that $\bigcup_{d \in \mathcal{D}} = D$. 
        \label{domain}
    \end{defn}

   \begin{figure}[ht]
        \centering
        \begin{tikzpicture}
            \draw[fill=lightgray,draw=lightgray] (0,0) rectangle (1,1);
            \draw[thick,draw=blue] (0,0) rectangle (1,2);
        \end{tikzpicture} \qquad
        \begin{tikzpicture}
            \draw[fill=lightgray,draw=lightgray] (1,0) rectangle (2,1);
            \draw[thick,draw=blue] (0,0) rectangle (2,1);
        \end{tikzpicture}\qquad
        \begin{tikzpicture}
            \draw[fill=lightgray,draw=lightgray] (0,1) rectangle (1,2);
            \draw[thick,draw=blue] (0,0) rectangle (1,2);
        \end{tikzpicture} \qquad
        \begin{tikzpicture}
            \draw[fill=lightgray,draw=lightgray] (0,0) rectangle (1,1);
            \draw[thick,draw=blue] (0,0) rectangle (2,1);
        \end{tikzpicture}
        \caption{Four types of dominoes}
        \label{dominoes}
    \end{figure}
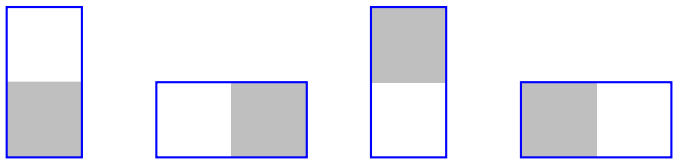

\begin{defn}
  For a domain of unit squares as described in Definition~\ref{domain}, we number its southwest to northeast diagonals from northwest to southeast, starting with diagonal 0. We will say that these diagonals begin on their southwestern-most square.
We can also color the squares of our domain as a chessboard such that even diagonals always consist of light squares. It follows then that our domain is tiled by four types of dominoes as in Figure~\ref{dominoes}. We will refer to the north or west square of a domino as its \textit{start}. Dominoes that start on an even diagonal will be called \textit{even}, and dominoes that start on an odd diagonal will be called \textit{odd}, as in Figure~\ref{dominoes}.
    \end{defn}
    
    For a domino tiling of a connected domain, we obtain a sequence of partitions as follows: Even dominoes are filled with holes and odd dominoes are filled with particles (see Figure~\ref{dominoes}). The last diagonal may contain squares not in our domain, in which case we treat those squares as the start of a domino, for example as on the right in Figure~\ref{exdomain}. Then the $i$th diagonal determines the $i$th partition $\lambda^{(i)}$ via the following rule: the $j$th part $\lambda^{(i)}_j$ of this partition is exactly the number of holes to the southwest of the $j$th-to-last particle.
    
    \begin{rem}
        Each partition here is obtained via its Maya diagram. A Young diagram, and hence a partition, is determined by the edges that form its southeast boundary. In particular this will be a series of north and east steps beginning with the southwest point of the bottom row of the Young diagram. In this way, for each diagonal in a domino tiling of our domain $D$, we obtain a partition by treating holes as east steps and particles as north steps.
        
        Moreover on each diagonal, we can treat the infinite sequence of squares to the southwest not in our domain as having particles, and the infinite sequence of squares to the northeast not in our domain as having holes. See Figure~\ref{boundary}.
        \label{holesparts}
    \end{rem}
    
    \begin{figure}[ht]
        \centering
        \begin{tikzpicture}[scale=.5]
            \draw[ultra thick] (0,7) rectangle (1,8);
\draw[ultra thick] (1,8) rectangle (2,9);
\draw[ultra thick] (2,9) rectangle (3,10);
\draw[ultra thick] (0,6) rectangle (1,7);
\draw[ultra thick] (1,7) rectangle (2,8);
\draw[ultra thick] (2,8) rectangle (3,9);
\draw[ultra thick] (3,9) rectangle (4,10);
\draw[ultra thick] (0,5) rectangle (1,6);
\draw[ultra thick] (1,6) rectangle (2,7);
\draw[ultra thick] (2,7) rectangle (3,8);
\draw[ultra thick] (3,8) rectangle (4,9);
\draw[ultra thick] (0,4) rectangle (1,5);
\draw[ultra thick] (1,5) rectangle (2,6);
\draw[ultra thick] (2,6) rectangle (3,7);
\draw[ultra thick] (3,7) rectangle (4,8);
\draw[ultra thick] (4,8) rectangle (5,9);
\draw[ultra thick] (0,3) rectangle (1,4);
\draw[ultra thick] (1,4) rectangle (2,5);
\draw[ultra thick] (2,5) rectangle (3,6);
\draw[ultra thick] (3,6) rectangle (4,7);
\draw[ultra thick] (4,7) rectangle (5,8);
\draw[ultra thick] (0,2) rectangle (1,3);
\draw[ultra thick] (1,3) rectangle (2,4);
\draw[ultra thick] (2,4) rectangle (3,5);
\draw[ultra thick] (3,5) rectangle (4,6);
\draw[ultra thick] (4,6) rectangle (5,7);
\draw[ultra thick] (5,7) rectangle (6,8);
\draw[ultra thick] (0,1) rectangle (1,2);
\draw[ultra thick] (1,2) rectangle (2,3);
\draw[ultra thick] (2,3) rectangle (3,4);
\draw[ultra thick] (3,4) rectangle (4,5);
\draw[ultra thick] (4,5) rectangle (5,6);
\draw[ultra thick] (5,6) rectangle (6,7);
\draw[ultra thick] (0,0) rectangle (1,1);
\draw[ultra thick] (2,2) rectangle (3,3);
\draw[ultra thick] (5,5) rectangle (6,6);
\draw[fill=white,draw=gray] (0,7) rectangle (1,8);
\draw[fill=white,draw=gray] (1,8) rectangle (2,9);
\draw[fill=white,draw=gray] (2,9) rectangle (3,10);
\draw[fill=lightgray,draw=gray] (0,6) rectangle (1,7);
\draw[fill=lightgray,draw=gray] (1,7) rectangle (2,8);
\draw[fill=lightgray,draw=gray] (2,8) rectangle (3,9);
\draw[fill=lightgray,draw=gray] (3,9) rectangle (4,10);
\draw[fill=white,draw=gray] (0,5) rectangle (1,6);
\draw[fill=white,draw=gray] (1,6) rectangle (2,7);
\draw[fill=white,draw=gray] (2,7) rectangle (3,8);
\draw[fill=white,draw=gray] (3,8) rectangle (4,9);
\draw[fill=lightgray,draw=gray] (0,4) rectangle (1,5);
\draw[fill=lightgray,draw=gray] (1,5) rectangle (2,6);
\draw[fill=lightgray,draw=gray] (2,6) rectangle (3,7);
\draw[fill=lightgray,draw=gray] (3,7) rectangle (4,8);
\draw[fill=lightgray,draw=gray] (4,8) rectangle (5,9);
\draw[fill=white,draw=gray] (0,3) rectangle (1,4);
\draw[fill=white,draw=gray] (1,4) rectangle (2,5);
\draw[fill=white,draw=gray] (2,5) rectangle (3,6);
\draw[fill=white,draw=gray] (3,6) rectangle (4,7);
\draw[fill=white,draw=gray] (4,7) rectangle (5,8);
\draw[fill=lightgray,draw=gray] (0,2) rectangle (1,3);
\draw[fill=lightgray,draw=gray] (1,3) rectangle (2,4);
\draw[fill=lightgray,draw=gray] (2,4) rectangle (3,5);
\draw[fill=lightgray,draw=gray] (3,5) rectangle (4,6);
\draw[fill=lightgray,draw=gray] (4,6) rectangle (5,7);
\draw[fill=lightgray,draw=gray] (5,7) rectangle (6,8);
\draw[fill=white,draw=gray] (0,1) rectangle (1,2);
\draw[fill=white,draw=gray] (1,2) rectangle (2,3);
\draw[fill=white,draw=gray] (2,3) rectangle (3,4);
\draw[fill=white,draw=gray] (3,4) rectangle (4,5);
\draw[fill=white,draw=gray] (4,5) rectangle (5,6);
\draw[fill=white,draw=gray] (5,6) rectangle (6,7);
\draw[fill=lightgray,draw=gray] (0,0) rectangle (1,1);
\draw[fill=lightgray,draw=gray] (2,2) rectangle (3,3);
\draw[fill=lightgray,draw=gray] (5,5) rectangle (6,6);
        \end{tikzpicture} \quad \begin{tikzpicture}[scale=.5]
            \draw[ultra thick] (0,7) rectangle (1,8);
\draw[ultra thick] (1,8) rectangle (2,9);
\draw[ultra thick] (2,9) rectangle (3,10);
\draw[ultra thick] (0,6) rectangle (1,7);
\draw[ultra thick] (1,7) rectangle (2,8);
\draw[ultra thick] (2,8) rectangle (3,9);
\draw[ultra thick] (3,9) rectangle (4,10);
\draw[ultra thick] (0,5) rectangle (1,6);
\draw[ultra thick] (1,6) rectangle (2,7);
\draw[ultra thick] (2,7) rectangle (3,8);
\draw[ultra thick] (3,8) rectangle (4,9);
\draw[ultra thick] (0,4) rectangle (1,5);
\draw[ultra thick] (1,5) rectangle (2,6);
\draw[ultra thick] (2,6) rectangle (3,7);
\draw[ultra thick] (3,7) rectangle (4,8);
\draw[ultra thick] (4,8) rectangle (5,9);
\draw[ultra thick] (0,3) rectangle (1,4);
\draw[ultra thick] (1,4) rectangle (2,5);
\draw[ultra thick] (2,5) rectangle (3,6);
\draw[ultra thick] (3,6) rectangle (4,7);
\draw[ultra thick] (4,7) rectangle (5,8);
\draw[ultra thick] (0,2) rectangle (1,3);
\draw[ultra thick] (1,3) rectangle (2,4);
\draw[ultra thick] (2,4) rectangle (3,5);
\draw[ultra thick] (3,5) rectangle (4,6);
\draw[ultra thick] (4,6) rectangle (5,7);
\draw[ultra thick] (5,7) rectangle (6,8);
\draw[ultra thick] (0,1) rectangle (1,2);
\draw[ultra thick] (1,2) rectangle (2,3);
\draw[ultra thick] (2,3) rectangle (3,4);
\draw[ultra thick] (3,4) rectangle (4,5);
\draw[ultra thick] (4,5) rectangle (5,6);
\draw[ultra thick] (5,6) rectangle (6,7);
\draw[ultra thick] (0,0) rectangle (1,1);
\draw[ultra thick] (2,2) rectangle (3,3);
\draw[ultra thick] (5,5) rectangle (6,6);
\draw[fill=white,draw=gray] (0,7) rectangle (1,8);
\draw[fill=white,draw=gray] (1,8) rectangle (2,9);
\draw[fill=white,draw=gray] (2,9) rectangle (3,10);
\draw[fill=lightgray,draw=gray] (0,6) rectangle (1,7);
\draw[fill=lightgray,draw=gray] (1,7) rectangle (2,8);
\draw[fill=lightgray,draw=gray] (2,8) rectangle (3,9);
\draw[fill=lightgray,draw=gray] (3,9) rectangle (4,10);
\draw[fill=white,draw=gray] (0,5) rectangle (1,6);
\draw[fill=white,draw=gray] (1,6) rectangle (2,7);
\draw[fill=white,draw=gray] (2,7) rectangle (3,8);
\draw[fill=white,draw=gray] (3,8) rectangle (4,9);
\draw[fill=lightgray,draw=gray] (0,4) rectangle (1,5);
\draw[fill=lightgray,draw=gray] (1,5) rectangle (2,6);
\draw[fill=lightgray,draw=gray] (2,6) rectangle (3,7);
\draw[fill=lightgray,draw=gray] (3,7) rectangle (4,8);
\draw[fill=lightgray,draw=gray] (4,8) rectangle (5,9);
\draw[fill=white,draw=gray] (0,3) rectangle (1,4);
\draw[fill=white,draw=gray] (1,4) rectangle (2,5);
\draw[fill=white,draw=gray] (2,5) rectangle (3,6);
\draw[fill=white,draw=gray] (3,6) rectangle (4,7);
\draw[fill=white,draw=gray] (4,7) rectangle (5,8);
\draw[fill=lightgray,draw=gray] (0,2) rectangle (1,3);
\draw[fill=lightgray,draw=gray] (1,3) rectangle (2,4);
\draw[fill=lightgray,draw=gray] (2,4) rectangle (3,5);
\draw[fill=lightgray,draw=gray] (3,5) rectangle (4,6);
\draw[fill=lightgray,draw=gray] (4,6) rectangle (5,7);
\draw[fill=lightgray,draw=gray] (5,7) rectangle (6,8);
\draw[fill=white,draw=gray] (0,1) rectangle (1,2);
\draw[fill=white,draw=gray] (1,2) rectangle (2,3);
\draw[fill=white,draw=gray] (2,3) rectangle (3,4);
\draw[fill=white,draw=gray] (3,4) rectangle (4,5);
\draw[fill=white,draw=gray] (4,5) rectangle (5,6);
\draw[fill=white,draw=gray] (5,6) rectangle (6,7);
\draw[fill=lightgray,draw=gray] (0,0) rectangle (1,1);
\draw[fill=lightgray,draw=gray] (2,2) rectangle (3,3);
\draw[fill=lightgray,draw=gray] (5,5) rectangle (6,6);
            \input{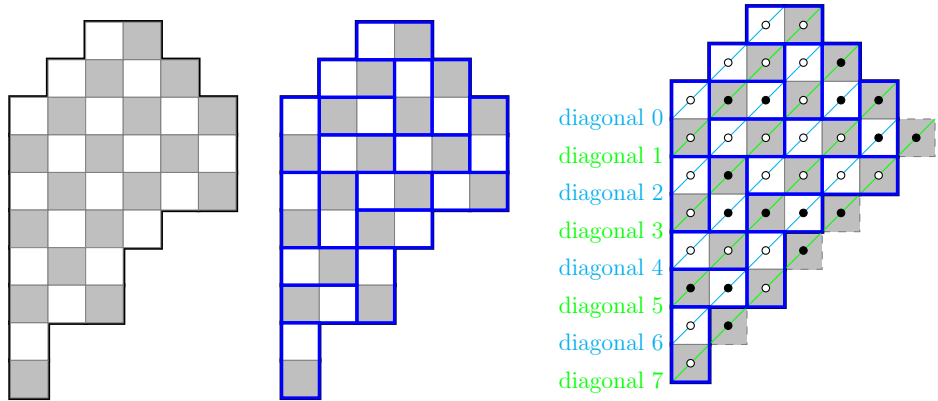}
        \end{tikzpicture} \quad \begin{tikzpicture}[scale=.5]
            \draw[fill=lightgray,draw=gray,dashed] (1,1) rectangle (2,2);
\draw[fill=lightgray,draw=gray,dashed] (3,3) rectangle (4,4);
\draw[fill=lightgray,draw=gray,dashed] (4,4) rectangle (5,5);
\draw[fill=lightgray,draw=gray,dashed] (6,6) rectangle (7,7);
            \draw[ultra thick] (0,7) rectangle (1,8);
\draw[ultra thick] (1,8) rectangle (2,9);
\draw[ultra thick] (2,9) rectangle (3,10);
\draw[ultra thick] (0,6) rectangle (1,7);
\draw[ultra thick] (1,7) rectangle (2,8);
\draw[ultra thick] (2,8) rectangle (3,9);
\draw[ultra thick] (3,9) rectangle (4,10);
\draw[ultra thick] (0,5) rectangle (1,6);
\draw[ultra thick] (1,6) rectangle (2,7);
\draw[ultra thick] (2,7) rectangle (3,8);
\draw[ultra thick] (3,8) rectangle (4,9);
\draw[ultra thick] (0,4) rectangle (1,5);
\draw[ultra thick] (1,5) rectangle (2,6);
\draw[ultra thick] (2,6) rectangle (3,7);
\draw[ultra thick] (3,7) rectangle (4,8);
\draw[ultra thick] (4,8) rectangle (5,9);
\draw[ultra thick] (0,3) rectangle (1,4);
\draw[ultra thick] (1,4) rectangle (2,5);
\draw[ultra thick] (2,5) rectangle (3,6);
\draw[ultra thick] (3,6) rectangle (4,7);
\draw[ultra thick] (4,7) rectangle (5,8);
\draw[ultra thick] (0,2) rectangle (1,3);
\draw[ultra thick] (1,3) rectangle (2,4);
\draw[ultra thick] (2,4) rectangle (3,5);
\draw[ultra thick] (3,5) rectangle (4,6);
\draw[ultra thick] (4,6) rectangle (5,7);
\draw[ultra thick] (5,7) rectangle (6,8);
\draw[ultra thick] (0,1) rectangle (1,2);
\draw[ultra thick] (1,2) rectangle (2,3);
\draw[ultra thick] (2,3) rectangle (3,4);
\draw[ultra thick] (3,4) rectangle (4,5);
\draw[ultra thick] (4,5) rectangle (5,6);
\draw[ultra thick] (5,6) rectangle (6,7);
\draw[ultra thick] (0,0) rectangle (1,1);
\draw[ultra thick] (2,2) rectangle (3,3);
\draw[ultra thick] (5,5) rectangle (6,6);
\draw[fill=white,draw=gray] (0,7) rectangle (1,8);
\draw[fill=white,draw=gray] (1,8) rectangle (2,9);
\draw[fill=white,draw=gray] (2,9) rectangle (3,10);
\draw[fill=lightgray,draw=gray] (0,6) rectangle (1,7);
\draw[fill=lightgray,draw=gray] (1,7) rectangle (2,8);
\draw[fill=lightgray,draw=gray] (2,8) rectangle (3,9);
\draw[fill=lightgray,draw=gray] (3,9) rectangle (4,10);
\draw[fill=white,draw=gray] (0,5) rectangle (1,6);
\draw[fill=white,draw=gray] (1,6) rectangle (2,7);
\draw[fill=white,draw=gray] (2,7) rectangle (3,8);
\draw[fill=white,draw=gray] (3,8) rectangle (4,9);
\draw[fill=lightgray,draw=gray] (0,4) rectangle (1,5);
\draw[fill=lightgray,draw=gray] (1,5) rectangle (2,6);
\draw[fill=lightgray,draw=gray] (2,6) rectangle (3,7);
\draw[fill=lightgray,draw=gray] (3,7) rectangle (4,8);
\draw[fill=lightgray,draw=gray] (4,8) rectangle (5,9);
\draw[fill=white,draw=gray] (0,3) rectangle (1,4);
\draw[fill=white,draw=gray] (1,4) rectangle (2,5);
\draw[fill=white,draw=gray] (2,5) rectangle (3,6);
\draw[fill=white,draw=gray] (3,6) rectangle (4,7);
\draw[fill=white,draw=gray] (4,7) rectangle (5,8);
\draw[fill=lightgray,draw=gray] (0,2) rectangle (1,3);
\draw[fill=lightgray,draw=gray] (1,3) rectangle (2,4);
\draw[fill=lightgray,draw=gray] (2,4) rectangle (3,5);
\draw[fill=lightgray,draw=gray] (3,5) rectangle (4,6);
\draw[fill=lightgray,draw=gray] (4,6) rectangle (5,7);
\draw[fill=lightgray,draw=gray] (5,7) rectangle (6,8);
\draw[fill=white,draw=gray] (0,1) rectangle (1,2);
\draw[fill=white,draw=gray] (1,2) rectangle (2,3);
\draw[fill=white,draw=gray] (2,3) rectangle (3,4);
\draw[fill=white,draw=gray] (3,4) rectangle (4,5);
\draw[fill=white,draw=gray] (4,5) rectangle (5,6);
\draw[fill=white,draw=gray] (5,6) rectangle (6,7);
\draw[fill=lightgray,draw=gray] (0,0) rectangle (1,1);
\draw[fill=lightgray,draw=gray] (2,2) rectangle (3,3);
\draw[fill=lightgray,draw=gray] (5,5) rectangle (6,6);
            \draw[color=cyan] (0,7) node[left,scale=.75]{diagonal 0} -- (3,10);
\draw[color=green] (0,6) node[left,scale=.75]{diagonal 1} -- (4,10);
\draw[color=cyan] (0,5) node[left,scale=.75]{diagonal 2} -- (4,9);
\draw[color=green] (0,4) node[left,scale=.75]{diagonal 3} -- (5,9);
\draw[color=cyan] (0,3) node[left,scale=.75]{diagonal 4} -- (5,8);
\draw[color=green] (0,2) node[left,scale=.75]{diagonal 5} -- (6,8);
\draw[color=cyan] (0,1) node[left,scale=.75]{diagonal 6} -- (6,7);
\draw[color=green] (0,0) node[left,scale=.75]{diagonal 7} -- (7,7);
            \draw[fill=black] (1.5,1.5) circle[radius=.1];
\draw[fill=black] (3.5,3.5) circle[radius=.1];
\draw[fill=black] (4.5,4.5) circle[radius=.1];
\draw[fill=black] (6.5,6.5) circle[radius=.1];
            \foreach\i\j in {0/2, 0/8, 3/9, 4/9, 5/8, 0/6, 1/6, 2/4}
{
\draw[very thick,draw=blue] (\i,\j) rectangle (\i+1,\j-2);
\pgfmathparse{mod(\i+\j, 2)}
\let\parity\pgfmathresult
\ifthenelse{\equal\parity{0.0}}
{\draw[fill=white] (\i+.5,\j-.5) circle[radius=.1];
\draw[fill=white] (\i+.5,\j-1.5) circle[radius=.1];}
{\draw[fill=black] (\i+.5,\j-.5) circle[radius=.1];
\draw[fill=black] (\i+.5,\j-1.5) circle[radius=.1];}
}

\foreach\i\j in {1/9, 2/10, 4/6, 3/7, 1/8, 2/6, 1/7, 2/5, 0/4, 0/3}
{
\draw[very thick,draw=blue] (\i,\j) rectangle (\i+2,\j-1);
\pgfmathparse{mod(\i+\j, 2)}
\let\parity\pgfmathresult
\ifthenelse{\equal\parity{0.0}}
{\draw[fill=white] (\i+.5,\j-.5) circle[radius=.1];
\draw[fill=white] (\i+1.5,\j-.5) circle[radius=.1];}
{\draw[fill=black] (\i+1.5,\j-.5) circle[radius=.1];
\draw[fill=black] (\i+.5,\j-.5) circle[radius=.1];}
}
        \end{tikzpicture}
        \caption{A particular domain, a domino tiling on that domain, and a filling of that domino tiling with holes and particles. The corresponding sequence of partitions is $\lambda^{(0)} = (0,0,0,0)$, $\lambda^{(1)} = (1,0,0,0)$, $\lambda^{(2)} = (2,0,0,0)$, $\lambda^{(3)} = \lambda^{(4)} = \lambda^{(5)} = (3,1,0,0)$, $\lambda^{(6)} = (3,2,1,0)$, $\lambda^{(7)} = (3,2,2,1)$.}
        \label{exdomain}        
    \end{figure}

    \begin{figure}[ht]
        \centering
        \begin{tikzpicture}[scale=.5]
            \draw (0,0) grid (1,1);
            \draw (0,1) grid (3,2);
            \draw (0,2) grid (3,3);
            \draw (0,3) grid (6,4);
            \draw[draw=red, thick] (0,0) -- (1,0) -- (1,1) -- (3,1) -- (3,3) -- (6,3) -- (6,4);
            \draw[draw=blue, thick] (0,-3) -- (0,0);
            \draw[draw=blue, thick] (6,4) -- (9,4);
            \foreach\i\j in {0/-3,0/-2,0/-1,1/0,3/1,3/2,6/3} {
                \draw[black,fill=black] (\i+.25,\j+.5) circle (.1);
            }
            \foreach\i\j in {0/0,1/1,2/1,3/3,4/3,5/3,6/4,7/4,8/4} {
                \draw[black,fill=white] (\i+.5,\j-.25) circle (.1);
            }
        \end{tikzpicture} \[
            \ldots \ \tikz\draw[black,fill=black] (0,0) circle (.5ex); \ \tikz\draw[black,fill=black] (0,0) circle (.5ex); \ \tikz\draw[black,fill=black] (0,0) circle (.5ex); \ \tikz\draw[black,fill=white] (0,0) circle (.5ex);\ \tikz\draw[black,fill=black] (0,0) circle (.5ex); \ \tikz\draw[black,fill=white] (0,0) circle (.5ex);\ \tikz\draw[black,fill=white] (0,0) circle (.5ex); \ \tikz\draw[black,fill=black] (0,0) circle (.5ex); \ \tikz\draw[black,fill=black] (0,0) circle (.5ex); \ \tikz\draw[black,fill=white] (0,0) circle (.5ex); \ \tikz\draw[black,fill=white] (0,0) circle (.5ex); \ \tikz\draw[black,fill=white] (0,0) circle (.5ex); \ \tikz\draw[black,fill=black] (0,0) circle (.5ex); \ \tikz\draw[black,fill=white] (0,0) circle (.5ex); \ \tikz\draw[black,fill=white] (0,0) circle (.5ex); \ \tikz\draw[black,fill=white] (0,0) circle (.5ex); \ \ldots
        \] 
        \caption{The Young diagram for the partition $(6,3,3,1)$ and its Maya diagram. Its shape is completely determined by the north and east steps in red. We can also consider this as an infinite sequence of north and east steps by prepending the blue north steps and appending the blue east steps. }
        \label{boundary}
    \end{figure}
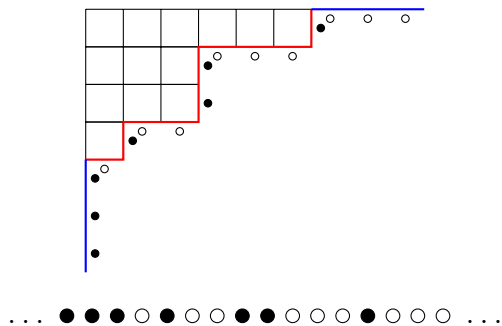

    \FloatBarrier

\section{Combinatorial objects related to domino tilings of generalized Aztec triangles}
\label{objects}

In this subsection all the objects are indexed by a partition $\mu=(\mu_1,\ldots ,\mu_n)$ into
$n$ non-negative parts and a parameter $\ell$ such that $\ell=2n$ or $\ell=2n+1$. We will refer
to them as Case~1 and
Case~2, respectively. We will look first at sequences of partitions, then tableaux, and finally
non-intersecting paths. In the next subsection, we will give a bijection between
these objects and domino tilings of the generalized Aztec triangles.

\subsection{Sequences of partitions}

  We now define certain sequences of partitions, with conditions similar to those found in \cite{steep,interlacing}.
Here we add an extra condition on the number of parts of each partition.
    \begin{defn}
        \label{cases}
        Given a partition $\mu = (\mu_1,\ldots,\mu_n)$ and parameter $\ell=2n$ or $\ell=2n+1$, consider sequences of partitions $\lambda^{(0)},\ldots ,\lambda^{(\ell-1)}$ such that:
        \begin{enumerate}[label=(\arabic*)]
            \item $\lambda^{(0)}$ is empty;
            \item $\lambda^{({\ell-1})}=\mu$; 
            \item $\lambda^{({2i+1})}/\lambda^{({2i})}$ is a horizontal strip; 
            \item $\lambda^{(2i)}/\lambda^{(2i-1)}$ is a vertical strip; 
            \item $\lambda^{(i)}$ has at most $\lceil \frac{i}{2} \rceil$ nonzero parts.
        \end{enumerate}
        We will refer to sequences $\lambda^{(0)},\ldots ,\lambda^{({2n-1})}$ as \textit{Case~1}, and sequences $\lambda^{(0)},\ldots ,\lambda^{({2n})}$ as \textit{Case~2}.
    \end{defn}

\subsection{Super sympectic tableaux}

Here we identify a family
of combinatorial objects that are in bijection with Case~1 and
Case~2 sequences. We do this by combining the conditions of the super and symplectic tableaux
from Definitions~\ref{defsymp} and
\ref{defsuper}.
\begin{defn}
        A \textit{super symplectic semistandard tableau} of shape $\mu$ is a filling of boxes of the Young diagram corresponding to $\mu$ with entries $1 < \overline{1} < 2 < \overline{2} < \cdots$ such that:
        \begin{enumerate}[label=(\arabic*)]
            \item rows and columns are weakly increasing;
             \item for all~$i$, the entries~$i$ form horizontal strips;
            \item for all~$i$, the entries $\overline{\imath}$ form vertical strips;
            \item $i$ and $\overline{\imath}$ do not appear below the $i$th row.
        \end{enumerate}
\end{defn}

\begin{prop}
For any $\mu = (\mu_1,\ldots,\mu_n)$, Case~1 sequences 
are in bijection with super symplectic semistandard tableaux of shape  $\mu$ with entries $1 < \overline{1} < 2 < \overline{2} < \cdots < n$.
        
        \label{tableaux}
    \end{prop}
    
    \begin{proof}
The bijection is established by
putting entries~$i$ in the cells corresponding to $\lambda_{2i-1} / \lambda_{2i-2}$, and entries $\overline{\imath}$ in the cells corresponding to $\lambda_{2i} / \lambda_{2i-1}$. The restriction on the number of nonzero parts in any $\lambda_i$ guarantees our symplectic condition. Conversely, given such a tableau we can recover $\lambda_{2i}$ by taking the cells of the tableaux with entries at most
$\overline\imath$, and we can recover $\lambda_{2i+1}$ by taking the cells of the tableaux with entries at most $i+1$.
    \end{proof}
    See Figure~\ref{tableau} for an example with $\mu=(5,3,2,1)$ and $\ell=7$. 
    The sequence of partitions is $\lambda^{(0)}=\emptyset, \ \lambda^{(1)}=(2), \ 
    \lambda^{(2)}=(3),\ \lambda^{(3)}=(3,1), \ \lambda^{(4)}=(3,2), \ \lambda^{(5)}=(4,3,2),
\ \lambda^{(6)}=(5,3,2),$ and  $\lambda^{(7)}=(5,3,2,1)$.

\begin{rem}
    The rows of our super symplectic semistandard tableaux are overpartitions \cite{CL04} and, without the symplectic conditions, our tableaux would be super semi-standard tableaux \cite{kra96}.
\end{rem}

    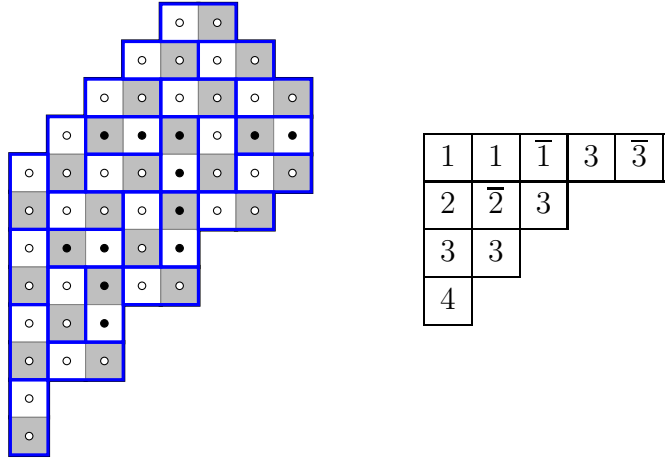
\begin{figure}[ht]
        \centering
        \[
        \vcenter{\hbox{\begin{tikzpicture}[scale=.5]
            \draw[ultra thick] (0,7) rectangle (1,8);
\draw[ultra thick] (1,8) rectangle (2,9);
\draw[ultra thick] (2,9) rectangle (3,10);
\draw[ultra thick] (3,10) rectangle (4,11);
\draw[ultra thick] (4,11) rectangle (5,12);
\draw[ultra thick] (0,6) rectangle (1,7);
\draw[ultra thick] (1,7) rectangle (2,8);
\draw[ultra thick] (2,8) rectangle (3,9);
\draw[ultra thick] (3,9) rectangle (4,10);
\draw[ultra thick] (4,10) rectangle (5,11);
\draw[ultra thick] (5,11) rectangle (6,12);
\draw[ultra thick] (0,5) rectangle (1,6);
\draw[ultra thick] (1,6) rectangle (2,7);
\draw[ultra thick] (2,7) rectangle (3,8);
\draw[ultra thick] (3,8) rectangle (4,9);
\draw[ultra thick] (4,9) rectangle (5,10);
\draw[ultra thick] (5,10) rectangle (6,11);
\draw[ultra thick] (0,4) rectangle (1,5);
\draw[ultra thick] (1,5) rectangle (2,6);
\draw[ultra thick] (2,6) rectangle (3,7);
\draw[ultra thick] (3,7) rectangle (4,8);
\draw[ultra thick] (4,8) rectangle (5,9);
\draw[ultra thick] (5,9) rectangle (6,10);
\draw[ultra thick] (6,10) rectangle (7,11);
\draw[ultra thick] (0,3) rectangle (1,4);
\draw[ultra thick] (1,4) rectangle (2,5);
\draw[ultra thick] (2,5) rectangle (3,6);
\draw[ultra thick] (3,6) rectangle (4,7);
\draw[ultra thick] (4,7) rectangle (5,8);
\draw[ultra thick] (5,8) rectangle (6,9);
\draw[ultra thick] (6,9) rectangle (7,10);
\draw[ultra thick] (0,2) rectangle (1,3);
\draw[ultra thick] (1,3) rectangle (2,4);
\draw[ultra thick] (2,4) rectangle (3,5);
\draw[ultra thick] (3,5) rectangle (4,6);
\draw[ultra thick] (4,6) rectangle (5,7);
\draw[ultra thick] (5,7) rectangle (6,8);
\draw[ultra thick] (6,8) rectangle (7,9);
\draw[ultra thick] (7,9) rectangle (8,10);
\draw[ultra thick] (0,1) rectangle (1,2);
\draw[ultra thick] (1,2) rectangle (2,3);
\draw[ultra thick] (2,3) rectangle (3,4);
\draw[ultra thick] (3,4) rectangle (4,5);
\draw[ultra thick] (4,5) rectangle (5,6);
\draw[ultra thick] (5,6) rectangle (6,7);
\draw[ultra thick] (6,7) rectangle (7,8);
\draw[ultra thick] (7,8) rectangle (8,9);
\draw[ultra thick] (0,0) rectangle (1,1);
\draw[ultra thick] (2,2) rectangle (3,3);
\draw[ultra thick] (4,4) rectangle (5,5);
\draw[ultra thick] (6,6) rectangle (7,7);
\draw[ultra thick] (7,7) rectangle (8,8);
\draw[fill=white,draw=gray] (0,7) rectangle (1,8);
\draw[fill=white,draw=gray] (1,8) rectangle (2,9);
\draw[fill=white,draw=gray] (2,9) rectangle (3,10);
\draw[fill=white,draw=gray] (3,10) rectangle (4,11);
\draw[fill=white,draw=gray] (4,11) rectangle (5,12);
\draw[fill=lightgray,draw=gray] (0,6) rectangle (1,7);
\draw[fill=lightgray,draw=gray] (1,7) rectangle (2,8);
\draw[fill=lightgray,draw=gray] (2,8) rectangle (3,9);
\draw[fill=lightgray,draw=gray] (3,9) rectangle (4,10);
\draw[fill=lightgray,draw=gray] (4,10) rectangle (5,11);
\draw[fill=lightgray,draw=gray] (5,11) rectangle (6,12);
\draw[fill=white,draw=gray] (0,5) rectangle (1,6);
\draw[fill=white,draw=gray] (1,6) rectangle (2,7);
\draw[fill=white,draw=gray] (2,7) rectangle (3,8);
\draw[fill=white,draw=gray] (3,8) rectangle (4,9);
\draw[fill=white,draw=gray] (4,9) rectangle (5,10);
\draw[fill=white,draw=gray] (5,10) rectangle (6,11);
\draw[fill=lightgray,draw=gray] (0,4) rectangle (1,5);
\draw[fill=lightgray,draw=gray] (1,5) rectangle (2,6);
\draw[fill=lightgray,draw=gray] (2,6) rectangle (3,7);
\draw[fill=lightgray,draw=gray] (3,7) rectangle (4,8);
\draw[fill=lightgray,draw=gray] (4,8) rectangle (5,9);
\draw[fill=lightgray,draw=gray] (5,9) rectangle (6,10);
\draw[fill=lightgray,draw=gray] (6,10) rectangle (7,11);
\draw[fill=white,draw=gray] (0,3) rectangle (1,4);
\draw[fill=white,draw=gray] (1,4) rectangle (2,5);
\draw[fill=white,draw=gray] (2,5) rectangle (3,6);
\draw[fill=white,draw=gray] (3,6) rectangle (4,7);
\draw[fill=white,draw=gray] (4,7) rectangle (5,8);
\draw[fill=white,draw=gray] (5,8) rectangle (6,9);
\draw[fill=white,draw=gray] (6,9) rectangle (7,10);
\draw[fill=lightgray,draw=gray] (0,2) rectangle (1,3);
\draw[fill=lightgray,draw=gray] (1,3) rectangle (2,4);
\draw[fill=lightgray,draw=gray] (2,4) rectangle (3,5);
\draw[fill=lightgray,draw=gray] (3,5) rectangle (4,6);
\draw[fill=lightgray,draw=gray] (4,6) rectangle (5,7);
\draw[fill=lightgray,draw=gray] (5,7) rectangle (6,8);
\draw[fill=lightgray,draw=gray] (6,8) rectangle (7,9);
\draw[fill=lightgray,draw=gray] (7,9) rectangle (8,10);
\draw[fill=white,draw=gray] (0,1) rectangle (1,2);
\draw[fill=white,draw=gray] (1,2) rectangle (2,3);
\draw[fill=white,draw=gray] (2,3) rectangle (3,4);
\draw[fill=white,draw=gray] (3,4) rectangle (4,5);
\draw[fill=white,draw=gray] (4,5) rectangle (5,6);
\draw[fill=white,draw=gray] (5,6) rectangle (6,7);
\draw[fill=white,draw=gray] (6,7) rectangle (7,8);
\draw[fill=white,draw=gray] (7,8) rectangle (8,9);
\draw[fill=lightgray,draw=gray] (0,0) rectangle (1,1);
\draw[fill=lightgray,draw=gray] (2,2) rectangle (3,3);
\draw[fill=lightgray,draw=gray] (4,4) rectangle (5,5);
\draw[fill=lightgray,draw=gray] (6,6) rectangle (7,7);
\draw[fill=lightgray,draw=gray] (7,7) rectangle (8,8);
            \foreach\i\j in {0/2, 0/8, 1/9, 4/9, 0/6, 2/5, 4/7, 1/5, 3/7, 5/9, 0/4}
{
\draw[very thick,draw=blue] (\i,\j) rectangle (\i+1,\j-2);
\pgfmathparse{mod(\i+\j, 2)}
\let\parity\pgfmathresult
\ifthenelse{\equal\parity{0.0}}
{\draw[fill=white] (\i+.5,\j-.5) circle[radius=.1];
\draw[fill=white] (\i+.5,\j-1.5) circle[radius=.1];}
{\draw[fill=black] (\i+.5,\j-.5) circle[radius=.1];
\draw[fill=black] (\i+.5,\j-1.5) circle[radius=.1];}
}

\foreach\i\j in {2/10, 3/11, 4/12, 2/9, 1/6, 1/7, 2/8, 4/10, 5/11, 6/9, 6/10, 3/5, 5/7, 6/8, 1/3}
{
\draw[very thick,draw=blue] (\i,\j) rectangle (\i+2,\j-1);
\pgfmathparse{mod(\i+\j, 2)}
\let\parity\pgfmathresult
\ifthenelse{\equal\parity{0.0}}
{\draw[fill=white] (\i+.5,\j-.5) circle[radius=.1];
\draw[fill=white] (\i+1.5,\j-.5) circle[radius=.1];}
{\draw[fill=black] (\i+1.5,\j-.5) circle[radius=.1];
\draw[fill=black] (\i+.5,\j-.5) circle[radius=.1];}
}
        \end{tikzpicture}}} \qquad \vcenter{\hbox{\begin{ytableau}
            \none & 1 & 1 & \overline{1} & 3 & \overline{3} \\
            \none & 2 & \overline{2} & 3 & \none & \none \\
            \none & 3 & 3 & \none & \none & \none \\
            \none & 4 & \none & \none & \none & \none
        \end{ytableau}}}
        \]
        \caption{A domino tiling on the domain, described in Theorem \ref{domain_thm}, for $\mu = (5,3,2,1)$ and $\ell=7$, and the corresponding super symplectic semistandard tableau.}
        \label{tableau}
    \end{figure}
    
    Similarly, Case~2 sequences are in bijection with fillings of entries $1 < \overline{1} < 2 < \overline{2} < \cdots < n < \overline{n}$ subject to the same conditions.
    \begin{prop}
For any $\mu = (\mu_1,\ldots,\mu_n)$, Case~2 sequences 
are in bijection with super symplectic semistandard tableaux
of shape $\mu$ with entries  in $1 < \overline{1} < 2 < \overline{2} < \cdots < n < \overline{n}$. 
        \label{tableaux2}
    \end{prop}

 We call these {\em super symplectic semistandard tableaux of type 1 and 2 and shape $\mu=(\mu_1,\ldots ,\mu_n)$}.

\subsection{Non-intersecting paths}
\label{nonintersect section}

We now propose a bijection between super symplectic semistandard tableaux of
types~1 and~2 and shape $\mu$ and tuples of non-intersecting Delannoy paths. 

\begin{prop}
    \label{paths type 1}

    There exists a bijection
between super symplectic semistandard tab\-leaux of type 1 of shape $\mu = (\mu_1, \ldots, \mu_n)$ and
sets of non-intersecting Delannoy paths beginning at the points $u_j=(-j,j)$ and ending at the points $v_j=(\mu_j-j,n)$ for $j=1,\ldots,n$.
\end{prop}

\begin{proof}
    Given a super symplectic semistandard tableau of shape $\mu = (\mu_1, \ldots, \mu_n)$, the unbarred entries in row~$j$ of the tableau give the locations of the east steps of the path from $u_j$ to $v_j$, and the barred entries give the locations of the northeast steps. Specifically if row $j$ contains the unbarred entry~$i$ exactly $m$ times, then the path from $u_j$ to $v_j$ contains exactly $m$ east steps on the line $y=i$. If row $j$ contains the barred entry $\overline{\imath}$, then that path has a northeast step that begins from a lattice point on the line $y=i$. From here, there is a unique way to insert north steps to complete the paths from $\lbrace u_i \rbrace$ to $\lbrace v_i \rbrace$.

    The conditions of a super symplectic semistandard tableau make it so that this map is well defined --- it does actually send every tableau to a set of non-intersecting paths. Then this map is certainly injective, as well as surjective.
\end{proof}

Continuing the example from Figure~\ref{tableau}, in Figure~\ref{nonintersect} we illustrate the bijection to non-intersecting paths.
We have a similar proposition for super symplectic standard tableaux of type 2.

\begin{figure}[ht]
    \centering
    \[
    \vcenter{\hbox{\begin{ytableau}
            \none & 1 & 1 & \overline{1} & 3 & \overline{3} \\
            \none & 2 & \overline{2} & 3 & \none & \none \\
            \none & 3 & 3 & \none & \none & \none \\
            \none & 4 & \none & \none & \none & \none
        \end{ytableau}}} \qquad \vcenter{\hbox{\begin{tikzpicture}[scale=.9]
            \draw[lightgray] (-4,1) grid (4,4);
            \filldraw[black] (-1,1) circle[radius=2pt] node[anchor=east] {$u_1$}; 
            \filldraw[black] (-2,2) circle[radius=2pt] node[anchor=east] {$u_2$}; 
            \filldraw[black] (-3,3) circle[radius=2pt] node[anchor=east] {$u_3$}; 
            \filldraw[black] (-4,4) circle[radius=2pt] node[anchor=east] {$u_4$}; 
            \filldraw[black] (4,4) circle[radius=2pt] node[anchor=south] {$v_1$};
            \filldraw[black] (1,4) circle[radius=2pt] node[anchor=south] {$v_2$};
            \filldraw[black] (-1,4) circle[radius=2pt] node[anchor=south] {$v_3$};
            \filldraw[black] (-3,4) circle[radius=2pt] node[anchor=south] {$v_4$};
            \draw[thick] (-4,4) -- (-3,4) node[midway,anchor=north] {4};
            \draw[thick] (-3,3) -- (-2,3) node[midway,anchor=north] {3};
            \draw[thick] (-2,3) -- (-1,3) node[midway,anchor=north] {3};
            \draw[thick] (-1,3) -- (-1,4);
            \draw[thick] (-2,2) -- (-1,2) node[midway,anchor=north] {2};
            \draw[thick] (-1,2) -- (0,3) node[midway,anchor=north] {$\overline{2}$};
            \draw[thick] (0,3) -- (1,3) node[midway,anchor=north] {3};
            \draw[thick] (1,3) -- (1,4);
            \draw[thick] (-1,1) -- (0,1) node[midway,anchor=north] {1};
            \draw[thick] (0,1) -- (1,1) node[midway,anchor=north] {1};
            \draw[thick] (1,1) -- (2,2) node[midway,anchor=north] {$\overline{1}$};
            \draw[thick] (2,2) -- (2,3);
            \draw[thick] (2,3) -- (3,3) node[midway,anchor=north] {3};
            \draw[thick] (3,3) -- (4,4) node[midway,anchor=north] {$\overline{3}$};
        \end{tikzpicture}}}
    \]
    \caption{The bijection from a super symplectic semistandard tableau to non-intersecting paths.}
    \label{nonintersect}
\end{figure}
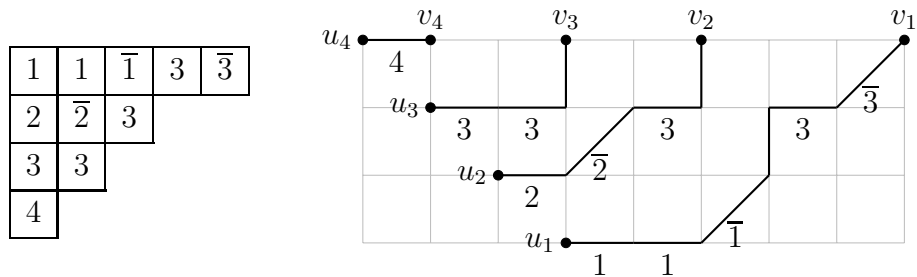

\begin{prop}
\label{paths type 2}
There exists a bijection
between super symplectic semistandard tab\-leaux of type 2 of shape $\mu = (\mu_1, \ldots, \mu_n)$
and non-intersecting H-Delannoy paths beginning at the points $u_j=(-j,j)$ and ending at the points $v_j=(\mu_j-j,n+1)$ for $j=1,\ldots,n$.
\end{prop}

\begin{proof}
The proof proceeds exactly
in the same way as the one above for Proposition~\ref{paths type 1}, except that now we see that our paths end on the line $y=n+1$.
Since our super symplectic semistandard tableaux of type 2 can contain entries up to $\overline{n}$ (and not $n+1$), our paths cannot end with an east step and hence are indeed H-Delannoy paths.
\end{proof}

\subsection{Domino tilings of generalized Aztec triangles}

The goal of this subsection is to construct domains $D$ such that domino tilings of $D$ are in bijection with Case~1 and Case~2 sequences defined in Definition~\ref{cases}. These types of bijections have been studied many times before
see \cite{steep,interlacing}, and we will use many of the same ideas.

Recall that, for any partition $\mu = (\mu_1,\ldots,\mu_n)$, $\ell=2n$, Case~1 sequences are $\lambda^{(0)},\ldots ,\lambda^{({\ell-1})}$ such that:
    \begin{enumerate}[label=(\arabic*)]
        \item $\lambda^{(0)}$ is empty;
        \item $\lambda^{({2n-1})}=\mu$;
        \item $\lambda^{({2i+1})}/\lambda^{({2i})}$ is a horizontal strip; 
        \item $\lambda^{(2i)}/\lambda^{(2i-1)}$ is a vertical strip; 
        \item $\lambda^{(i)}$ has at most $\lceil \frac{i}{2} \rceil$ nonzero parts.
    \end{enumerate}
    
    \begin{thm}
        \label{domain_thm}
        The domain $D$ whose domino tilings are in bijection with Case~1 sequences for a given $\mu = (\mu_1,\ldots,\mu_n)$ and $\ell=2n$ can be constructed by its diagonals as follows:
        \begin{enumerate}[label=(\roman*)]
        \item[\em(i)] Diagonal $0$ is of length $\mu_1$.
            \item[\em(ii)] The length of each odd diagonal $2i+1$ is one more than the length of the previous diagonal $2i$. The first (i.e., left-most) square of diagonal $2i+1$ lies 
                directly below the first square of diagonal $2i$. 
            \item[\em(iii)] Each even diagonal $2i$ has length equal to the length of the previous diagonal $2i-1$. The first square of diagonal $2i$ lies
               directly below the first square of diagonal $2i-1$. 
             \item[\em(iv)] The last diagonal $2n-1$ follows
{\em(ii)} above, but consists of squares in $D$ and squares not in $D$:
convert the partition $\mu$ into its sequence of holes and particles.
Then any hole corresponds to a square in $D$ and a particle corresponds to a square not in $D$.
        \end{enumerate}
    \end{thm}
    
    Figure~\ref{last_diag} shows examples of (iv) in Theorem~\ref{domain_thm}, and Figure~\ref{domains} illustrates Theorem~\ref{domain_thm} for a particular $\mu$.
    
    \begin{figure}[ht]
        \centering
        \begin{tikzpicture}[scale=.5]
            \foreach\x in {1,3,5} {\draw[fill=lightgray,draw=gray,dashed] (\x,\x) rectangle (\x+1,\x+1);
                \draw[fill=black] (.5+\x,.5+\x) circle[radius=.1];}
            \foreach\x in {1,3} {\draw[thick] (\x,\x) -- (\x,\x+1) -- (\x+1,\x+1);}
            \foreach\x in {0,2,4} {\draw[fill=lightgray,draw=gray] (\x,\x) rectangle (\x+1,\x+1);
                \draw[fill=white] (.5+\x,.5+\x) circle[radius=.1];}
            \foreach\x in {2,4} {\draw[thick] (\x,\x) -- (\x+1,\x) -- (\x+1,\x+1);}
            \draw[thick] (0,1) -- (0,0) -- (1,0) -- (1,1);
            \draw[thick] (5,5) -- (5,6);
        \end{tikzpicture}
        \begin{tikzpicture}[scale=.5]
            \foreach\x in {0,2,4,6} {\draw[fill=lightgray,draw=gray,dashed] (\x,\x) rectangle (\x+1,\x+1);
                \draw[fill=black] (.5+\x,.5+\x) circle[radius=.1];}
            \foreach\x in {2,4} {\draw[thick] (\x,\x) -- (\x,\x+1) -- (\x+1,\x+1);}
            \foreach\x in {1,3,5} {\draw[fill=lightgray,draw=gray] (\x,\x) rectangle (\x+1,\x+1);
                \draw[fill=white] (.5+\x,.5+\x) circle[radius=.1];
                \draw[thick] (\x,\x) -- (\x+1,\x) -- (\x+1,\x+1);}
            \draw[thick] (0,1) -- (1,1);
            \draw[thick] (6,6) -- (6,7);
        \end{tikzpicture} \qquad
        \begin{tikzpicture}[scale=.5]
            \foreach\x in {0,1,3,4,7} {\draw[fill=lightgray,draw=gray,dashed] (\x,\x) rectangle (\x+1,\x+1);
                \draw[fill=black] (.5+\x,.5+\x) circle[radius=.1];}
            \foreach\x in {1,3,4} {\draw[thick] (\x,\x) -- (\x,\x+1) -- (\x+1,\x+1);}
            \foreach\x in {2,5,6} {\draw[fill=lightgray,draw=gray] (\x,\x) rectangle (\x+1,\x+1);
                \draw[fill=white] (.5+\x,.5+\x) circle[radius=.1];
                \draw[thick] (\x,\x) -- (\x+1,\x) -- (\x+1,\x+1);}
            \draw[thick] (0,1) -- (1,1);
            \draw[thick] (7,7) -- (7,8);
        \end{tikzpicture}
        \caption{On the left, the last diagonal for partition $\mu = (3,2,1)$. In the
center $\mu = (3,2,1,0)$, and on the right $\mu =(4,1,1,0,0)$. The resulting southeast boundaries are also drawn.}
        \label{last_diag}
    \end{figure}
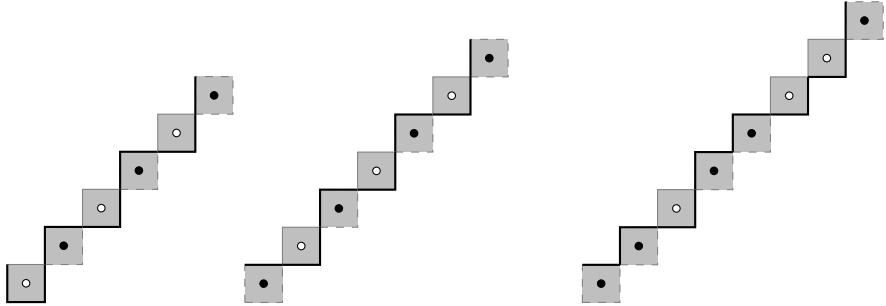
    
    \begin{figure}[ht]
        \centering
        \begin{tikzpicture}[scale=.4]
            \draw[fill=lightgray,draw=gray,dashed] (0,0) rectangle (1,1);
\draw[fill=lightgray,draw=gray,dashed] (2,2) rectangle (3,3);
\draw[fill=lightgray,draw=gray,dashed] (4,4) rectangle (5,5);
\draw[fill=lightgray,draw=gray,dashed] (6,6) rectangle (7,7);
            \draw[ultra thick] (0,7) rectangle (1,8);
\draw[ultra thick] (1,8) rectangle (2,9);
\draw[ultra thick] (2,9) rectangle (3,10);
\draw[ultra thick] (0,6) rectangle (1,7);
\draw[ultra thick] (1,7) rectangle (2,8);
\draw[ultra thick] (2,8) rectangle (3,9);
\draw[ultra thick] (3,9) rectangle (4,10);
\draw[ultra thick] (0,5) rectangle (1,6);
\draw[ultra thick] (1,6) rectangle (2,7);
\draw[ultra thick] (2,7) rectangle (3,8);
\draw[ultra thick] (3,8) rectangle (4,9);
\draw[ultra thick] (0,4) rectangle (1,5);
\draw[ultra thick] (1,5) rectangle (2,6);
\draw[ultra thick] (2,6) rectangle (3,7);
\draw[ultra thick] (3,7) rectangle (4,8);
\draw[ultra thick] (4,8) rectangle (5,9);
\draw[ultra thick] (0,3) rectangle (1,4);
\draw[ultra thick] (1,4) rectangle (2,5);
\draw[ultra thick] (2,5) rectangle (3,6);
\draw[ultra thick] (3,6) rectangle (4,7);
\draw[ultra thick] (4,7) rectangle (5,8);
\draw[ultra thick] (0,2) rectangle (1,3);
\draw[ultra thick] (1,3) rectangle (2,4);
\draw[ultra thick] (2,4) rectangle (3,5);
\draw[ultra thick] (3,5) rectangle (4,6);
\draw[ultra thick] (4,6) rectangle (5,7);
\draw[ultra thick] (5,7) rectangle (6,8);
\draw[ultra thick] (0,1) rectangle (1,2);
\draw[ultra thick] (1,2) rectangle (2,3);
\draw[ultra thick] (2,3) rectangle (3,4);
\draw[ultra thick] (3,4) rectangle (4,5);
\draw[ultra thick] (4,5) rectangle (5,6);
\draw[ultra thick] (5,6) rectangle (6,7);
\draw[ultra thick] (1,1) rectangle (2,2);
\draw[ultra thick] (3,3) rectangle (4,4);
\draw[ultra thick] (5,5) rectangle (6,6);
\draw[fill=white,draw=gray] (0,7) rectangle (1,8);
\draw[fill=white,draw=gray] (1,8) rectangle (2,9);
\draw[fill=white,draw=gray] (2,9) rectangle (3,10);
\draw[fill=lightgray,draw=gray] (0,6) rectangle (1,7);
\draw[fill=lightgray,draw=gray] (1,7) rectangle (2,8);
\draw[fill=lightgray,draw=gray] (2,8) rectangle (3,9);
\draw[fill=lightgray,draw=gray] (3,9) rectangle (4,10);
\draw[fill=white,draw=gray] (0,5) rectangle (1,6);
\draw[fill=white,draw=gray] (1,6) rectangle (2,7);
\draw[fill=white,draw=gray] (2,7) rectangle (3,8);
\draw[fill=white,draw=gray] (3,8) rectangle (4,9);
\draw[fill=lightgray,draw=gray] (0,4) rectangle (1,5);
\draw[fill=lightgray,draw=gray] (1,5) rectangle (2,6);
\draw[fill=lightgray,draw=gray] (2,6) rectangle (3,7);
\draw[fill=lightgray,draw=gray] (3,7) rectangle (4,8);
\draw[fill=lightgray,draw=gray] (4,8) rectangle (5,9);
\draw[fill=white,draw=gray] (0,3) rectangle (1,4);
\draw[fill=white,draw=gray] (1,4) rectangle (2,5);
\draw[fill=white,draw=gray] (2,5) rectangle (3,6);
\draw[fill=white,draw=gray] (3,6) rectangle (4,7);
\draw[fill=white,draw=gray] (4,7) rectangle (5,8);
\draw[fill=lightgray,draw=gray] (0,2) rectangle (1,3);
\draw[fill=lightgray,draw=gray] (1,3) rectangle (2,4);
\draw[fill=lightgray,draw=gray] (2,4) rectangle (3,5);
\draw[fill=lightgray,draw=gray] (3,5) rectangle (4,6);
\draw[fill=lightgray,draw=gray] (4,6) rectangle (5,7);
\draw[fill=lightgray,draw=gray] (5,7) rectangle (6,8);
\draw[fill=white,draw=gray] (0,1) rectangle (1,2);
\draw[fill=white,draw=gray] (1,2) rectangle (2,3);
\draw[fill=white,draw=gray] (2,3) rectangle (3,4);
\draw[fill=white,draw=gray] (3,4) rectangle (4,5);
\draw[fill=white,draw=gray] (4,5) rectangle (5,6);
\draw[fill=white,draw=gray] (5,6) rectangle (6,7);
\draw[fill=lightgray,draw=gray] (1,1) rectangle (2,2);
\draw[fill=lightgray,draw=gray] (3,3) rectangle (4,4);
\draw[fill=lightgray,draw=gray] (5,5) rectangle (6,6);
            \draw[color=cyan] (0,7) node[left,scale=.75]{diagonal 0} -- (3,10);
\draw[color=green] (0,6) node[left,scale=.75]{diagonal 1} -- (4,10);
\draw[color=cyan] (0,5) node[left,scale=.75]{diagonal 2} -- (4,9);
\draw[color=green] (0,4) node[left,scale=.75]{diagonal 3} -- (5,9);
\draw[color=cyan] (0,3) node[left,scale=.75]{diagonal 4} -- (5,8);
\draw[color=green] (0,2) node[left,scale=.75]{diagonal 5} -- (6,8);
\draw[color=cyan] (0,1) node[left,scale=.75]{diagonal 6} -- (6,7);
\draw[color=green] (0,0) node[left,scale=.75]{diagonal 7} -- (7,7);
        \end{tikzpicture} 
        \caption{
        The domain $D$ for the partition $\mu = (3,2,1,0)$ and $\ell=6$.
        }
        \label{domains}
    \end{figure}
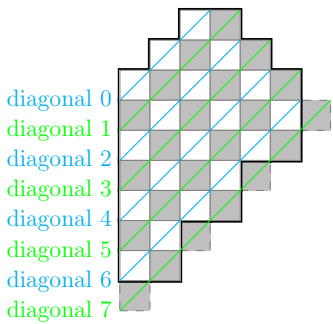
    
    \begin{proof}
We begin by showing that our domain $D$ as defined satisfies our conditions
(1)--(5) above. Given that diagonal 0 contains only the start (i.e., the north or west square) of dominoes, it is clear that (1) must be true. Furthermore, due to condition (iv) and since the squares
on diagonal $2n-1$ that are in $D$ are not starts (i.e., only south or east squares) of dominoes, it is clear that (2) must also be true. It can also be seen by induction that each diagonal $i$ has $n$ holes and $\lceil \frac{i}{2} \rceil$ particles, hence we have (5). 
        
        Now consider condition (4). 
        Given the configuration of holes and particles on a diagonal $2i-1$, the configuration of holes and particles on diagonal $2i$ is determined by whether the dominoes starting on the previous diagonal are vertical or horizontal. We see that any such vertical
        domino keeps the particle at the same place from diagonal $2i-1$ to diagonal $2i$, and any horizontal domino moves the particle one square later in the next diagonal. Moreover, the restrictions of the domino tiling
imply that a horizontal domino requires that any particles immediately following it must also be the start of a horizontal domino. In terms of the corresponding partition this means that, going from $\lambda^{(2i-1)}$ to $\lambda^{(2i)}$, each part can increase by at most one with the additional condition that, for parts of equal size, if one of these parts increases by one then the earlier parts of equal size must also increase by one. This is exactly the condition that $\lambda^{(2i)} / \lambda^{(2i-1)}$ is a vertical strip.
        
For condition~(3) we proceed similarly as for condition~(4).
In this case   
        we see that the configuration of particles on diagonal $2i+1$ is determined by whether the dominoes starting on diagonal $2i$ are vertical or horizontal. However, since dominoes starting on diagonal $2i$ contain holes instead of the particles when proving condition~(4), we now require
        \[
          {\lambda'}^{(2i+1)} / {\lambda'}^{(2i)} \text{ is a vertical strip} \implies \lambda^{(2i+1)} / \lambda^{(2i)} \text{ is a horizontal strip}
.
        \]
        Note that, since the length of diagonal $2i+1$ is greater than the length of diagonal $2i$, we are permitted to increase the length of $\lambda^{(2i)}_n$.

        Hence we have shown that any domino tiling gives us a sequence of partitions as desired. In fact the same argument shows that any sequence of partitions satisfying (1)--(5) must give us a domino tiling of $D$. In particular conditions~(3) and~(4) guarantee that we get a domino tiling of some domain $D'$. Then condition~(5) tells us that the length of each diagonal in this domain $D'$ is the same
as in our domain $D$, and conditions~(1) and~(2) guarantee that the first and last rows
of $D'$ agree with $D$. This completely determines the shape of $D'$.
    \end{proof}

The domain on which domino tilings are in bijection with Case~2 sequences can be constructed similarly to that of Case~1 sequences. 
    \begin{thm}
        \label{domain_thm2}
        The domain $D$ whose domino tilings are in bijection with Case~2 sequences for a given $\mu = (\mu_1,\ldots,\mu_n)$, and $\ell=2n$ can be constructed by its diagonals as follows:
        \begin{enumerate}[label=(\roman*)]
            \item[\em(i)] Diagonal $0$ is of length $m$.
            \item[\em(ii)] The length of each odd diagonal $2i+1$ is one more than the length of the previous diagonal $2i$. The diagonal starts so that
                the first square of diagonal $2i+1$ is directly below the first square of diagonal. $2i$ 
         \item[\em(iii)] Each even diagonal $2i$ has length equal to the length of the previous diagonal $2i-1$. The first square of diagonal $2i$ lies
                directly below the first square of diagonal $2i-1$. 
              \item[\em(iv)] The last diagonal $2n$ consists of squares in $D$ and squares not in $D$. Convert the partition $\mu$ into its sequence of holes and particles.
Then any particle corresponds to a square in $D$ and a hole corresponds to a square not in $D$.
        \end{enumerate}
    \end{thm}
We illustrate Theorem~\ref{domain_thm2} with an example
in Figure~\ref{domains2}. Theorem~\ref{domain_thm2} differs from Theorem~\ref{domain_thm} only by (iv), and its proof proceeds similarly.
    
As mentioned previously, the Aztec triangles in \cite{difrancesco2} are exactly the Case~1 domains corresponding to
the partition $\mu = (n, n-1, \ldots, 1)$ and $\ell=2n$. Hence we will refer to these domains, both Cases~1 and~2, as {\em generalized Aztec triangles of types 1 and 2.}
    
    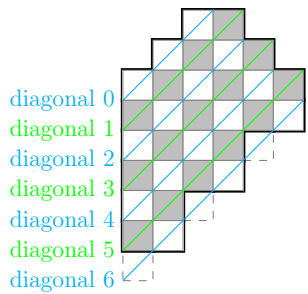
\begin{figure}[ht]
        \centering
        \begin{tikzpicture}[scale=.4]
            \draw[fill=white,draw=gray,dashed] (0,-1) rectangle (1,0);
\draw[fill=white,draw=gray,dashed] (2,1) rectangle (3,2);
\draw[fill=white,draw=gray,dashed] (4,3) rectangle (5,4);
            \draw[ultra thick] (0,5) rectangle (1,6);
\draw[ultra thick] (1,6) rectangle (2,7);
\draw[ultra thick] (2,7) rectangle (3,8);
\draw[ultra thick] (0,4) rectangle (1,5);
\draw[ultra thick] (1,5) rectangle (2,6);
\draw[ultra thick] (2,6) rectangle (3,7);
\draw[ultra thick] (3,7) rectangle (4,8);
\draw[ultra thick] (0,3) rectangle (1,4);
\draw[ultra thick] (1,4) rectangle (2,5);
\draw[ultra thick] (2,5) rectangle (3,6);
\draw[ultra thick] (3,6) rectangle (4,7);
\draw[ultra thick] (0,2) rectangle (1,3);
\draw[ultra thick] (1,3) rectangle (2,4);
\draw[ultra thick] (2,4) rectangle (3,5);
\draw[ultra thick] (3,5) rectangle (4,6);
\draw[ultra thick] (4,6) rectangle (5,7);
\draw[ultra thick] (0,1) rectangle (1,2);
\draw[ultra thick] (1,2) rectangle (2,3);
\draw[ultra thick] (2,3) rectangle (3,4);
\draw[ultra thick] (3,4) rectangle (4,5);
\draw[ultra thick] (4,5) rectangle (5,6);
\draw[ultra thick] (0,0) rectangle (1,1);
\draw[ultra thick] (1,1) rectangle (2,2);
\draw[ultra thick] (2,2) rectangle (3,3);
\draw[ultra thick] (3,3) rectangle (4,4);
\draw[ultra thick] (4,4) rectangle (5,5);
\draw[ultra thick] (5,5) rectangle (6,6);
\draw[ultra thick] (1,0) rectangle (2,1);
\draw[ultra thick] (3,2) rectangle (4,3);
\draw[ultra thick] (5,4) rectangle (6,5);
\draw[fill=white,draw=gray] (0,5) rectangle (1,6);
\draw[fill=white,draw=gray] (1,6) rectangle (2,7);
\draw[fill=white,draw=gray] (2,7) rectangle (3,8);
\draw[fill=lightgray,draw=gray] (0,4) rectangle (1,5);
\draw[fill=lightgray,draw=gray] (1,5) rectangle (2,6);
\draw[fill=lightgray,draw=gray] (2,6) rectangle (3,7);
\draw[fill=lightgray,draw=gray] (3,7) rectangle (4,8);
\draw[fill=white,draw=gray] (0,3) rectangle (1,4);
\draw[fill=white,draw=gray] (1,4) rectangle (2,5);
\draw[fill=white,draw=gray] (2,5) rectangle (3,6);
\draw[fill=white,draw=gray] (3,6) rectangle (4,7);
\draw[fill=lightgray,draw=gray] (0,2) rectangle (1,3);
\draw[fill=lightgray,draw=gray] (1,3) rectangle (2,4);
\draw[fill=lightgray,draw=gray] (2,4) rectangle (3,5);
\draw[fill=lightgray,draw=gray] (3,5) rectangle (4,6);
\draw[fill=lightgray,draw=gray] (4,6) rectangle (5,7);
\draw[fill=white,draw=gray] (0,1) rectangle (1,2);
\draw[fill=white,draw=gray] (1,2) rectangle (2,3);
\draw[fill=white,draw=gray] (2,3) rectangle (3,4);
\draw[fill=white,draw=gray] (3,4) rectangle (4,5);
\draw[fill=white,draw=gray] (4,5) rectangle (5,6);
\draw[fill=lightgray,draw=gray] (0,0) rectangle (1,1);
\draw[fill=lightgray,draw=gray] (1,1) rectangle (2,2);
\draw[fill=lightgray,draw=gray] (2,2) rectangle (3,3);
\draw[fill=lightgray,draw=gray] (3,3) rectangle (4,4);
\draw[fill=lightgray,draw=gray] (4,4) rectangle (5,5);
\draw[fill=lightgray,draw=gray] (5,5) rectangle (6,6);
\draw[fill=white,draw=gray] (1,0) rectangle (2,1);
\draw[fill=white,draw=gray] (3,2) rectangle (4,3);
\draw[fill=white,draw=gray] (5,4) rectangle (6,5);
            \draw[color=cyan] (0,5) node[left,scale=.75]{diagonal 0} -- (3,8);
\draw[color=green] (0,4) node[left,scale=.75]{diagonal 1} -- (4,8);
\draw[color=cyan] (0,3) node[left,scale=.75]{diagonal 2} -- (4,7);
\draw[color=green] (0,2) node[left,scale=.75]{diagonal 3} -- (5,7);
\draw[color=cyan] (0,1) node[left,scale=.75]{diagonal 4} -- (5,6);
\draw[color=green] (0,0) node[left,scale=.75]{diagonal 5} -- (6,6);
\draw[color=cyan] (0,-1) node[left,scale=.75]{diagonal 6} -- (6,5);
        \end{tikzpicture} 
        \caption{The domain $D$ for the partition $\mu = (3,2,1)$ and $\ell=7$.}
        \label{domains2}
    \end{figure}

\begin{rem}
The non-intersecting paths of
Subsection~\ref{nonintersect section} can also be obtained from a domino tiling, as in \cite{difrancesco1}, by assigning to each type of domino a step in a path. Figure~\ref{paths} shows this assignment based on the orientation and parity of the domino. One can check that this gives exactly the non-intersecting paths
of Propositions~\ref{paths type 1} and \ref{paths type 2} up to some rotation and reflection. See Figure~\ref{tiling1} for an example of the non-intersecting paths obtained from a domino tiling of a generalized Aztec triangle. 
\end{rem}

    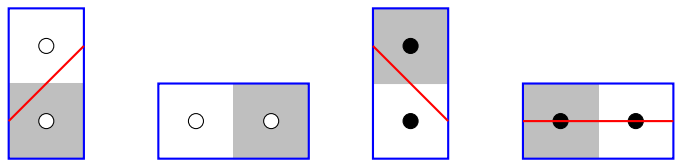
\begin{figure}[ht]
        \centering
        \begin{tikzpicture}
            \draw[fill=lightgray,draw=lightgray] (0,0) rectangle (1,1);
            \draw[thick,draw=blue] (0,0) rectangle (1,2);
            \foreach\x in {0,1} {\draw[fill=white] (.5,\x+.5) circle[radius=.1];}
            \draw[thick,draw=red] (0,.5) -- (1,1.5);
        \end{tikzpicture} \qquad
        \begin{tikzpicture}
            \draw[fill=lightgray,draw=lightgray] (1,0) rectangle (2,1);
            \draw[thick,draw=blue] (0,0) rectangle (2,1);
            \foreach\x in {0,1} {\draw[fill=white] (\x+.5,.5) circle[radius=.1];}
        \end{tikzpicture}\qquad
        \begin{tikzpicture}
            \draw[fill=lightgray,draw=lightgray] (0,1) rectangle (1,2);
            \draw[thick,draw=blue] (0,0) rectangle (1,2);
            \foreach\x in {0,1} {\draw[fill=black] (.5,\x+.5) circle[radius=.1];}
            \draw[thick,draw=red] (0,1.5) -- (1,.5);
        \end{tikzpicture} \qquad
        \begin{tikzpicture}
            \draw[fill=lightgray,draw=lightgray] (0,0) rectangle (1,1);
            \draw[thick,draw=blue] (0,0) rectangle (2,1);
            \foreach\x in {0,1} {\draw[fill=black] (\x+.5,.5) circle[radius=.1];}
            \draw[thick,draw=red] (0,.5) -- (2,.5);
        \end{tikzpicture}
        \caption{For each domino, the path is
indicated in red. A parity argument explains why dominoes of the second type do not permit a path.}
        \label{paths}
    \end{figure}

    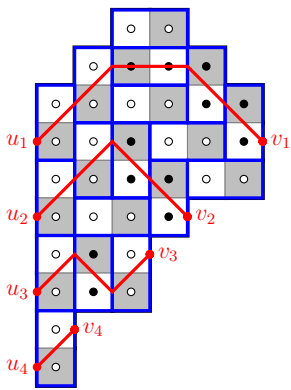
\begin{figure}[ht]
        \centering
        \begin{tikzpicture}[scale=.5]
            \draw[ultra thick] (0,7) rectangle (1,8);
\draw[ultra thick] (1,8) rectangle (2,9);
\draw[ultra thick] (2,9) rectangle (3,10);
\draw[ultra thick] (0,6) rectangle (1,7);
\draw[ultra thick] (1,7) rectangle (2,8);
\draw[ultra thick] (2,8) rectangle (3,9);
\draw[ultra thick] (3,9) rectangle (4,10);
\draw[ultra thick] (0,5) rectangle (1,6);
\draw[ultra thick] (1,6) rectangle (2,7);
\draw[ultra thick] (2,7) rectangle (3,8);
\draw[ultra thick] (3,8) rectangle (4,9);
\draw[ultra thick] (0,4) rectangle (1,5);
\draw[ultra thick] (1,5) rectangle (2,6);
\draw[ultra thick] (2,6) rectangle (3,7);
\draw[ultra thick] (3,7) rectangle (4,8);
\draw[ultra thick] (4,8) rectangle (5,9);
\draw[ultra thick] (0,3) rectangle (1,4);
\draw[ultra thick] (1,4) rectangle (2,5);
\draw[ultra thick] (2,5) rectangle (3,6);
\draw[ultra thick] (3,6) rectangle (4,7);
\draw[ultra thick] (4,7) rectangle (5,8);
\draw[ultra thick] (0,2) rectangle (1,3);
\draw[ultra thick] (1,3) rectangle (2,4);
\draw[ultra thick] (2,4) rectangle (3,5);
\draw[ultra thick] (3,5) rectangle (4,6);
\draw[ultra thick] (4,6) rectangle (5,7);
\draw[ultra thick] (5,7) rectangle (6,8);
\draw[ultra thick] (0,1) rectangle (1,2);
\draw[ultra thick] (1,2) rectangle (2,3);
\draw[ultra thick] (2,3) rectangle (3,4);
\draw[ultra thick] (3,4) rectangle (4,5);
\draw[ultra thick] (4,5) rectangle (5,6);
\draw[ultra thick] (5,6) rectangle (6,7);
\draw[ultra thick] (0,0) rectangle (1,1);
\draw[ultra thick] (2,2) rectangle (3,3);
\draw[ultra thick] (5,5) rectangle (6,6);
\draw[fill=white,draw=gray] (0,7) rectangle (1,8);
\draw[fill=white,draw=gray] (1,8) rectangle (2,9);
\draw[fill=white,draw=gray] (2,9) rectangle (3,10);
\draw[fill=lightgray,draw=gray] (0,6) rectangle (1,7);
\draw[fill=lightgray,draw=gray] (1,7) rectangle (2,8);
\draw[fill=lightgray,draw=gray] (2,8) rectangle (3,9);
\draw[fill=lightgray,draw=gray] (3,9) rectangle (4,10);
\draw[fill=white,draw=gray] (0,5) rectangle (1,6);
\draw[fill=white,draw=gray] (1,6) rectangle (2,7);
\draw[fill=white,draw=gray] (2,7) rectangle (3,8);
\draw[fill=white,draw=gray] (3,8) rectangle (4,9);
\draw[fill=lightgray,draw=gray] (0,4) rectangle (1,5);
\draw[fill=lightgray,draw=gray] (1,5) rectangle (2,6);
\draw[fill=lightgray,draw=gray] (2,6) rectangle (3,7);
\draw[fill=lightgray,draw=gray] (3,7) rectangle (4,8);
\draw[fill=lightgray,draw=gray] (4,8) rectangle (5,9);
\draw[fill=white,draw=gray] (0,3) rectangle (1,4);
\draw[fill=white,draw=gray] (1,4) rectangle (2,5);
\draw[fill=white,draw=gray] (2,5) rectangle (3,6);
\draw[fill=white,draw=gray] (3,6) rectangle (4,7);
\draw[fill=white,draw=gray] (4,7) rectangle (5,8);
\draw[fill=lightgray,draw=gray] (0,2) rectangle (1,3);
\draw[fill=lightgray,draw=gray] (1,3) rectangle (2,4);
\draw[fill=lightgray,draw=gray] (2,4) rectangle (3,5);
\draw[fill=lightgray,draw=gray] (3,5) rectangle (4,6);
\draw[fill=lightgray,draw=gray] (4,6) rectangle (5,7);
\draw[fill=lightgray,draw=gray] (5,7) rectangle (6,8);
\draw[fill=white,draw=gray] (0,1) rectangle (1,2);
\draw[fill=white,draw=gray] (1,2) rectangle (2,3);
\draw[fill=white,draw=gray] (2,3) rectangle (3,4);
\draw[fill=white,draw=gray] (3,4) rectangle (4,5);
\draw[fill=white,draw=gray] (4,5) rectangle (5,6);
\draw[fill=white,draw=gray] (5,6) rectangle (6,7);
\draw[fill=lightgray,draw=gray] (0,0) rectangle (1,1);
\draw[fill=lightgray,draw=gray] (2,2) rectangle (3,3);
\draw[fill=lightgray,draw=gray] (5,5) rectangle (6,6);
            \foreach\i\j in {0/2, 1/9, 0/8, 4/9, 5/8, 0/6, 1/7, 2/7, 3/6, 1/4, 0/4, 2/4}
{
\draw[very thick,draw=blue] (\i,\j) rectangle (\i+1,\j-2);
\pgfmathparse{mod(\i+\j, 2)}
\let\parity\pgfmathresult
\ifthenelse{\equal\parity{0.0}}
{\draw[fill=white] (\i+.5,\j-.5) circle[radius=.1];
\draw[fill=white] (\i+.5,\j-1.5) circle[radius=.1];}
{\draw[fill=black] (\i+.5,\j-.5) circle[radius=.1];
\draw[fill=black] (\i+.5,\j-1.5) circle[radius=.1];}
}

\foreach\i\j in {2/10, 2/9, 4/6, 3/7, 2/8, 1/5}
{
\draw[very thick,draw=blue] (\i,\j) rectangle (\i+2,\j-1);
\pgfmathparse{mod(\i+\j, 2)}
\let\parity\pgfmathresult
\ifthenelse{\equal\parity{0.0}}
{\draw[fill=white] (\i+.5,\j-.5) circle[radius=.1];
\draw[fill=white] (\i+1.5,\j-.5) circle[radius=.1];}
{\draw[fill=black] (\i+1.5,\j-.5) circle[radius=.1];
\draw[fill=black] (\i+.5,\j-.5) circle[radius=.1];}
}
            \draw[very thick,draw=red] (0,0.5) -- (1,1.5);
\draw[very thick,draw=red] (0,2.5) -- (1,3.5) -- (2,2.5) -- (3,3.5);
\draw[very thick,draw=red] (0,4.5) -- (1,5.5) -- (2,6.5) -- (3,5.5) -- (4,4.5);
\draw[very thick,draw=red] (0,6.5) -- (1,7.5) -- (2,8.5) -- (4,8.5) -- (5,7.5) -- (6,6.5);
            \draw[fill=red,draw=red,color=red] (0,0.5) node[left,scale=.75]{$u_4$} circle[radius=.1];
\draw[fill=red,draw=red,color=red] (0,2.5) node[left,scale=.75]{$u_3$} circle[radius=.1];
\draw[fill=red,draw=red,color=red] (0,4.5) node[left,scale=.75]{$u_2$} circle[radius=.1];
\draw[fill=red,draw=red,color=red] (0,6.5) node[left,scale=.75]{$u_1$} circle[radius=.1];
\draw[fill=red,draw=red,color=red] (1,1.5) circle[radius=.1] node[right,scale=.75]{$v_4$} circle[radius=.1];
\draw[fill=red,draw=red,color=red] (3,3.5) circle[radius=.1] node[right,scale=.75]{$v_3$} circle[radius=.1];
\draw[fill=red,draw=red,color=red] (4,4.5) circle[radius=.1] node[right,scale=.75]{$v_2$} circle[radius=.1];
\draw[fill=red,draw=red,color=red] (6,6.5) circle[radius=.1] node[right,scale=.75]{$v_1$} circle[radius=.1];
        \end{tikzpicture}
        \caption{A particular domino tiling of the domain for Case 1 with $\mu = (3,2,2,1)$ and $\ell=8$
        }
        \label{tiling1}
        \end{figure}

\subsection{Enumeration}

We can now put the objects together.
\begin{thm}
Given a partition $\mu=(\mu_1,\ldots,\mu_n)$, there exists a bijection between:
\begin{enumerate}
    \item[\em1.] Case~1 sequences of partitions with final partition $\mu$ and $\ell=2n$ defined in Definition~\ref{cases}.
    \item[\em2.]  Super symplectic semistandard tableaux of shape $\mu$ and entries less than $\overline{n}$.
    \item[\em3.]  Tuples of non-intersecting Delannoy paths where the starting points
    are $(-i,i)$ for $1 \leq i\leq n$ and ending points are $(\mu_{i}-i,n)$ for $1 \leq i\leq n$.
    \item[\em4.]  Domino tilings of the generalized Aztec triangle of type 1 corresponding to $\mu$ as defined in Theorem~\ref{domain_thm}. 
\end{enumerate}
\end{thm}

The Lindstr\"om--Gessel--Viennot lemma~\cite{LGV1} has the following consequence.
\begin{cor}
The number of Case~1 sequences of partitions with final partition $\mu=(\mu_1,\ldots ,\mu_n)$
is equal to the determinant of
the $n\times n$ matrix $A_1=(A_{i,j})_{0\le i,j\le n-1}$ whose $(i,j)$-entry is
\[
D(\mu_{i+1}-i+j,n-j).
\]
\label{det1}
\end{cor}

Similarly, we get the following result.
\begin{thm}
Given a partition $\mu=(\mu_1,\ldots,\mu_n)$, there exists a bijection between:
\begin{enumerate}
    \item[\em1.] Case~2 sequences of partitions with final partition $\mu$ and $\ell=2n+1$ defined in Definition~\ref{cases}.
    \item[\em2.] Super symplectic semistandard tableaux of shape $\mu$ and entries less than $n+1$.
    \item[\em3.] Tuples of non-intersecting H-Delannoy paths where the starting points
    are $(-i,i)$ for $1 \leq i\leq n$ and ending points are $(\mu_{i}-i,n+1)$ for $1 \leq i\leq n$.
    \item[\em4.] Domino tilings of the generalized Aztec triangle of type 2 corresponding to $\mu$ as defined in Theorem~\ref{domain_thm2}. 
\end{enumerate}
\end{thm}

\begin{cor}
The number of Case~2 sequences of partitions with final partition $\mu=(\mu_1,\ldots ,\mu_n)$
is equal to the determinant of
the $n\times n$ matrix $A_2=(A_{i,j})_{0\le i,j\le n-1}$ whose $(i,j)$-entry is
\[
H(\mu_{i+1}-i+j,n-j).
\]
\label{det2}
\end{cor}

The proofs are similar to the proof in Section 4 of \cite{difrancesco2}.

    \section{Product formulas}
    \label{conj}

    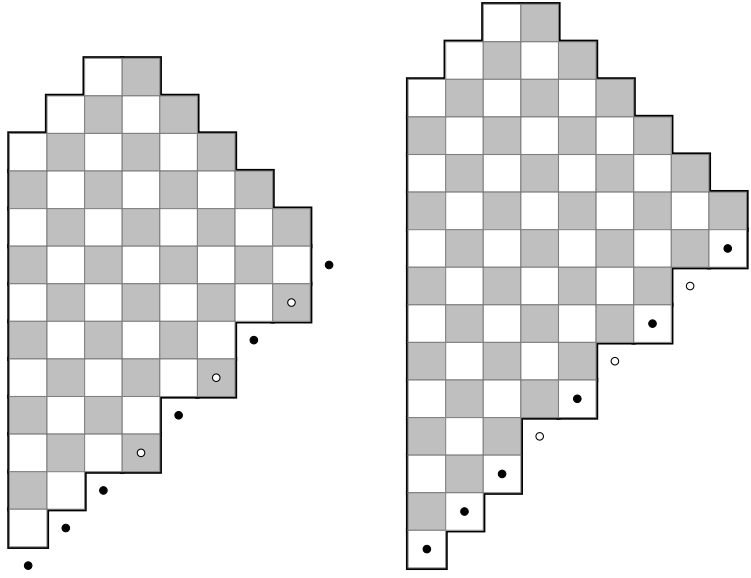
\begin{figure}[ht]
        \centering
        \begin{tikzpicture}[scale=.5]
            \draw[ultra thick] (0,11) rectangle (1,12);
\draw[ultra thick] (1,12) rectangle (2,13);
\draw[ultra thick] (2,13) rectangle (3,14);
\draw[ultra thick] (0,10) rectangle (1,11);
\draw[ultra thick] (1,11) rectangle (2,12);
\draw[ultra thick] (2,12) rectangle (3,13);
\draw[ultra thick] (3,13) rectangle (4,14);
\draw[ultra thick] (0,9) rectangle (1,10);
\draw[ultra thick] (1,10) rectangle (2,11);
\draw[ultra thick] (2,11) rectangle (3,12);
\draw[ultra thick] (3,12) rectangle (4,13);
\draw[ultra thick] (0,8) rectangle (1,9);
\draw[ultra thick] (1,9) rectangle (2,10);
\draw[ultra thick] (2,10) rectangle (3,11);
\draw[ultra thick] (3,11) rectangle (4,12);
\draw[ultra thick] (4,12) rectangle (5,13);
\draw[ultra thick] (0,7) rectangle (1,8);
\draw[ultra thick] (1,8) rectangle (2,9);
\draw[ultra thick] (2,9) rectangle (3,10);
\draw[ultra thick] (3,10) rectangle (4,11);
\draw[ultra thick] (4,11) rectangle (5,12);
\draw[ultra thick] (0,6) rectangle (1,7);
\draw[ultra thick] (1,7) rectangle (2,8);
\draw[ultra thick] (2,8) rectangle (3,9);
\draw[ultra thick] (3,9) rectangle (4,10);
\draw[ultra thick] (4,10) rectangle (5,11);
\draw[ultra thick] (5,11) rectangle (6,12);
\draw[ultra thick] (0,5) rectangle (1,6);
\draw[ultra thick] (1,6) rectangle (2,7);
\draw[ultra thick] (2,7) rectangle (3,8);
\draw[ultra thick] (3,8) rectangle (4,9);
\draw[ultra thick] (4,9) rectangle (5,10);
\draw[ultra thick] (5,10) rectangle (6,11);
\draw[ultra thick] (0,4) rectangle (1,5);
\draw[ultra thick] (1,5) rectangle (2,6);
\draw[ultra thick] (2,6) rectangle (3,7);
\draw[ultra thick] (3,7) rectangle (4,8);
\draw[ultra thick] (4,8) rectangle (5,9);
\draw[ultra thick] (5,9) rectangle (6,10);
\draw[ultra thick] (6,10) rectangle (7,11);
\draw[ultra thick] (0,3) rectangle (1,4);
\draw[ultra thick] (1,4) rectangle (2,5);
\draw[ultra thick] (2,5) rectangle (3,6);
\draw[ultra thick] (3,6) rectangle (4,7);
\draw[ultra thick] (4,7) rectangle (5,8);
\draw[ultra thick] (5,8) rectangle (6,9);
\draw[ultra thick] (6,9) rectangle (7,10);
\draw[ultra thick] (0,2) rectangle (1,3);
\draw[ultra thick] (1,3) rectangle (2,4);
\draw[ultra thick] (2,4) rectangle (3,5);
\draw[ultra thick] (3,5) rectangle (4,6);
\draw[ultra thick] (4,6) rectangle (5,7);
\draw[ultra thick] (5,7) rectangle (6,8);
\draw[ultra thick] (6,8) rectangle (7,9);
\draw[ultra thick] (7,9) rectangle (8,10);
\draw[ultra thick] (0,1) rectangle (1,2);
\draw[ultra thick] (1,2) rectangle (2,3);
\draw[ultra thick] (2,3) rectangle (3,4);
\draw[ultra thick] (3,4) rectangle (4,5);
\draw[ultra thick] (4,5) rectangle (5,6);
\draw[ultra thick] (5,6) rectangle (6,7);
\draw[ultra thick] (6,7) rectangle (7,8);
\draw[ultra thick] (7,8) rectangle (8,9);
\draw[ultra thick] (3,3) rectangle (4,4);
\draw[ultra thick] (5,5) rectangle (6,6);
\draw[ultra thick] (7,7) rectangle (8,8);
\draw[fill=white,draw=gray] (0,11) rectangle (1,12);
\draw[fill=white,draw=gray] (1,12) rectangle (2,13);
\draw[fill=white,draw=gray] (2,13) rectangle (3,14);
\draw[fill=lightgray,draw=gray] (0,10) rectangle (1,11);
\draw[fill=lightgray,draw=gray] (1,11) rectangle (2,12);
\draw[fill=lightgray,draw=gray] (2,12) rectangle (3,13);
\draw[fill=lightgray,draw=gray] (3,13) rectangle (4,14);
\draw[fill=white,draw=gray] (0,9) rectangle (1,10);
\draw[fill=white,draw=gray] (1,10) rectangle (2,11);
\draw[fill=white,draw=gray] (2,11) rectangle (3,12);
\draw[fill=white,draw=gray] (3,12) rectangle (4,13);
\draw[fill=lightgray,draw=gray] (0,8) rectangle (1,9);
\draw[fill=lightgray,draw=gray] (1,9) rectangle (2,10);
\draw[fill=lightgray,draw=gray] (2,10) rectangle (3,11);
\draw[fill=lightgray,draw=gray] (3,11) rectangle (4,12);
\draw[fill=lightgray,draw=gray] (4,12) rectangle (5,13);
\draw[fill=white,draw=gray] (0,7) rectangle (1,8);
\draw[fill=white,draw=gray] (1,8) rectangle (2,9);
\draw[fill=white,draw=gray] (2,9) rectangle (3,10);
\draw[fill=white,draw=gray] (3,10) rectangle (4,11);
\draw[fill=white,draw=gray] (4,11) rectangle (5,12);
\draw[fill=lightgray,draw=gray] (0,6) rectangle (1,7);
\draw[fill=lightgray,draw=gray] (1,7) rectangle (2,8);
\draw[fill=lightgray,draw=gray] (2,8) rectangle (3,9);
\draw[fill=lightgray,draw=gray] (3,9) rectangle (4,10);
\draw[fill=lightgray,draw=gray] (4,10) rectangle (5,11);
\draw[fill=lightgray,draw=gray] (5,11) rectangle (6,12);
\draw[fill=white,draw=gray] (0,5) rectangle (1,6);
\draw[fill=white,draw=gray] (1,6) rectangle (2,7);
\draw[fill=white,draw=gray] (2,7) rectangle (3,8);
\draw[fill=white,draw=gray] (3,8) rectangle (4,9);
\draw[fill=white,draw=gray] (4,9) rectangle (5,10);
\draw[fill=white,draw=gray] (5,10) rectangle (6,11);
\draw[fill=lightgray,draw=gray] (0,4) rectangle (1,5);
\draw[fill=lightgray,draw=gray] (1,5) rectangle (2,6);
\draw[fill=lightgray,draw=gray] (2,6) rectangle (3,7);
\draw[fill=lightgray,draw=gray] (3,7) rectangle (4,8);
\draw[fill=lightgray,draw=gray] (4,8) rectangle (5,9);
\draw[fill=lightgray,draw=gray] (5,9) rectangle (6,10);
\draw[fill=lightgray,draw=gray] (6,10) rectangle (7,11);
\draw[fill=white,draw=gray] (0,3) rectangle (1,4);
\draw[fill=white,draw=gray] (1,4) rectangle (2,5);
\draw[fill=white,draw=gray] (2,5) rectangle (3,6);
\draw[fill=white,draw=gray] (3,6) rectangle (4,7);
\draw[fill=white,draw=gray] (4,7) rectangle (5,8);
\draw[fill=white,draw=gray] (5,8) rectangle (6,9);
\draw[fill=white,draw=gray] (6,9) rectangle (7,10);
\draw[fill=lightgray,draw=gray] (0,2) rectangle (1,3);
\draw[fill=lightgray,draw=gray] (1,3) rectangle (2,4);
\draw[fill=lightgray,draw=gray] (2,4) rectangle (3,5);
\draw[fill=lightgray,draw=gray] (3,5) rectangle (4,6);
\draw[fill=lightgray,draw=gray] (4,6) rectangle (5,7);
\draw[fill=lightgray,draw=gray] (5,7) rectangle (6,8);
\draw[fill=lightgray,draw=gray] (6,8) rectangle (7,9);
\draw[fill=lightgray,draw=gray] (7,9) rectangle (8,10);
\draw[fill=white,draw=gray] (0,1) rectangle (1,2);
\draw[fill=white,draw=gray] (1,2) rectangle (2,3);
\draw[fill=white,draw=gray] (2,3) rectangle (3,4);
\draw[fill=white,draw=gray] (3,4) rectangle (4,5);
\draw[fill=white,draw=gray] (4,5) rectangle (5,6);
\draw[fill=white,draw=gray] (5,6) rectangle (6,7);
\draw[fill=white,draw=gray] (6,7) rectangle (7,8);
\draw[fill=white,draw=gray] (7,8) rectangle (8,9);
\draw[fill=lightgray,draw=gray] (3,3) rectangle (4,4);
\draw[fill=lightgray,draw=gray] (5,5) rectangle (6,6);
\draw[fill=lightgray,draw=gray] (7,7) rectangle (8,8);
            \draw[fill=black] (0.5,0.5) circle[radius=.1];
\draw[fill=black] (1.5,1.5) circle[radius=.1];
\draw[fill=black] (2.5,2.5) circle[radius=.1];
\draw[fill=black] (4.5,4.5) circle[radius=.1];
\draw[fill=black] (6.5,6.5) circle[radius=.1];
\draw[fill=black] (8.5,8.5) circle[radius=.1];
\draw[fill=white] (3.5,3.5) circle[radius=.1];
\draw[fill=white] (5.5,5.5) circle[radius=.1];
\draw[fill=white] (7.5,7.5) circle[radius=.1];
        \end{tikzpicture} \qquad \begin{tikzpicture}[scale=.5]
            \draw[ultra thick] (0,11) rectangle (1,12);
\draw[ultra thick] (1,12) rectangle (2,13);
\draw[ultra thick] (2,13) rectangle (3,14);
\draw[ultra thick] (0,10) rectangle (1,11);
\draw[ultra thick] (1,11) rectangle (2,12);
\draw[ultra thick] (2,12) rectangle (3,13);
\draw[ultra thick] (3,13) rectangle (4,14);
\draw[ultra thick] (0,9) rectangle (1,10);
\draw[ultra thick] (1,10) rectangle (2,11);
\draw[ultra thick] (2,11) rectangle (3,12);
\draw[ultra thick] (3,12) rectangle (4,13);
\draw[ultra thick] (0,8) rectangle (1,9);
\draw[ultra thick] (1,9) rectangle (2,10);
\draw[ultra thick] (2,10) rectangle (3,11);
\draw[ultra thick] (3,11) rectangle (4,12);
\draw[ultra thick] (4,12) rectangle (5,13);
\draw[ultra thick] (0,7) rectangle (1,8);
\draw[ultra thick] (1,8) rectangle (2,9);
\draw[ultra thick] (2,9) rectangle (3,10);
\draw[ultra thick] (3,10) rectangle (4,11);
\draw[ultra thick] (4,11) rectangle (5,12);
\draw[ultra thick] (0,6) rectangle (1,7);
\draw[ultra thick] (1,7) rectangle (2,8);
\draw[ultra thick] (2,8) rectangle (3,9);
\draw[ultra thick] (3,9) rectangle (4,10);
\draw[ultra thick] (4,10) rectangle (5,11);
\draw[ultra thick] (5,11) rectangle (6,12);
\draw[ultra thick] (0,5) rectangle (1,6);
\draw[ultra thick] (1,6) rectangle (2,7);
\draw[ultra thick] (2,7) rectangle (3,8);
\draw[ultra thick] (3,8) rectangle (4,9);
\draw[ultra thick] (4,9) rectangle (5,10);
\draw[ultra thick] (5,10) rectangle (6,11);
\draw[ultra thick] (0,4) rectangle (1,5);
\draw[ultra thick] (1,5) rectangle (2,6);
\draw[ultra thick] (2,6) rectangle (3,7);
\draw[ultra thick] (3,7) rectangle (4,8);
\draw[ultra thick] (4,8) rectangle (5,9);
\draw[ultra thick] (5,9) rectangle (6,10);
\draw[ultra thick] (6,10) rectangle (7,11);
\draw[ultra thick] (0,3) rectangle (1,4);
\draw[ultra thick] (1,4) rectangle (2,5);
\draw[ultra thick] (2,5) rectangle (3,6);
\draw[ultra thick] (3,6) rectangle (4,7);
\draw[ultra thick] (4,7) rectangle (5,8);
\draw[ultra thick] (5,8) rectangle (6,9);
\draw[ultra thick] (6,9) rectangle (7,10);
\draw[ultra thick] (0,2) rectangle (1,3);
\draw[ultra thick] (1,3) rectangle (2,4);
\draw[ultra thick] (2,4) rectangle (3,5);
\draw[ultra thick] (3,5) rectangle (4,6);
\draw[ultra thick] (4,6) rectangle (5,7);
\draw[ultra thick] (5,7) rectangle (6,8);
\draw[ultra thick] (6,8) rectangle (7,9);
\draw[ultra thick] (7,9) rectangle (8,10);
\draw[ultra thick] (0,1) rectangle (1,2);
\draw[ultra thick] (1,2) rectangle (2,3);
\draw[ultra thick] (2,3) rectangle (3,4);
\draw[ultra thick] (3,4) rectangle (4,5);
\draw[ultra thick] (4,5) rectangle (5,6);
\draw[ultra thick] (5,6) rectangle (6,7);
\draw[ultra thick] (6,7) rectangle (7,8);
\draw[ultra thick] (7,8) rectangle (8,9);
\draw[ultra thick] (0,0) rectangle (1,1);
\draw[ultra thick] (1,1) rectangle (2,2);
\draw[ultra thick] (2,2) rectangle (3,3);
\draw[ultra thick] (3,3) rectangle (4,4);
\draw[ultra thick] (4,4) rectangle (5,5);
\draw[ultra thick] (5,5) rectangle (6,6);
\draw[ultra thick] (6,6) rectangle (7,7);
\draw[ultra thick] (7,7) rectangle (8,8);
\draw[ultra thick] (8,8) rectangle (9,9);
\draw[ultra thick] (0,-1) rectangle (1,0);
\draw[ultra thick] (1,0) rectangle (2,1);
\draw[ultra thick] (2,1) rectangle (3,2);
\draw[ultra thick] (4,3) rectangle (5,4);
\draw[ultra thick] (6,5) rectangle (7,6);
\draw[ultra thick] (8,7) rectangle (9,8);
\draw[fill=white,draw=gray] (0,11) rectangle (1,12);
\draw[fill=white,draw=gray] (1,12) rectangle (2,13);
\draw[fill=white,draw=gray] (2,13) rectangle (3,14);
\draw[fill=lightgray,draw=gray] (0,10) rectangle (1,11);
\draw[fill=lightgray,draw=gray] (1,11) rectangle (2,12);
\draw[fill=lightgray,draw=gray] (2,12) rectangle (3,13);
\draw[fill=lightgray,draw=gray] (3,13) rectangle (4,14);
\draw[fill=white,draw=gray] (0,9) rectangle (1,10);
\draw[fill=white,draw=gray] (1,10) rectangle (2,11);
\draw[fill=white,draw=gray] (2,11) rectangle (3,12);
\draw[fill=white,draw=gray] (3,12) rectangle (4,13);
\draw[fill=lightgray,draw=gray] (0,8) rectangle (1,9);
\draw[fill=lightgray,draw=gray] (1,9) rectangle (2,10);
\draw[fill=lightgray,draw=gray] (2,10) rectangle (3,11);
\draw[fill=lightgray,draw=gray] (3,11) rectangle (4,12);
\draw[fill=lightgray,draw=gray] (4,12) rectangle (5,13);
\draw[fill=white,draw=gray] (0,7) rectangle (1,8);
\draw[fill=white,draw=gray] (1,8) rectangle (2,9);
\draw[fill=white,draw=gray] (2,9) rectangle (3,10);
\draw[fill=white,draw=gray] (3,10) rectangle (4,11);
\draw[fill=white,draw=gray] (4,11) rectangle (5,12);
\draw[fill=lightgray,draw=gray] (0,6) rectangle (1,7);
\draw[fill=lightgray,draw=gray] (1,7) rectangle (2,8);
\draw[fill=lightgray,draw=gray] (2,8) rectangle (3,9);
\draw[fill=lightgray,draw=gray] (3,9) rectangle (4,10);
\draw[fill=lightgray,draw=gray] (4,10) rectangle (5,11);
\draw[fill=lightgray,draw=gray] (5,11) rectangle (6,12);
\draw[fill=white,draw=gray] (0,5) rectangle (1,6);
\draw[fill=white,draw=gray] (1,6) rectangle (2,7);
\draw[fill=white,draw=gray] (2,7) rectangle (3,8);
\draw[fill=white,draw=gray] (3,8) rectangle (4,9);
\draw[fill=white,draw=gray] (4,9) rectangle (5,10);
\draw[fill=white,draw=gray] (5,10) rectangle (6,11);
\draw[fill=lightgray,draw=gray] (0,4) rectangle (1,5);
\draw[fill=lightgray,draw=gray] (1,5) rectangle (2,6);
\draw[fill=lightgray,draw=gray] (2,6) rectangle (3,7);
\draw[fill=lightgray,draw=gray] (3,7) rectangle (4,8);
\draw[fill=lightgray,draw=gray] (4,8) rectangle (5,9);
\draw[fill=lightgray,draw=gray] (5,9) rectangle (6,10);
\draw[fill=lightgray,draw=gray] (6,10) rectangle (7,11);
\draw[fill=white,draw=gray] (0,3) rectangle (1,4);
\draw[fill=white,draw=gray] (1,4) rectangle (2,5);
\draw[fill=white,draw=gray] (2,5) rectangle (3,6);
\draw[fill=white,draw=gray] (3,6) rectangle (4,7);
\draw[fill=white,draw=gray] (4,7) rectangle (5,8);
\draw[fill=white,draw=gray] (5,8) rectangle (6,9);
\draw[fill=white,draw=gray] (6,9) rectangle (7,10);
\draw[fill=lightgray,draw=gray] (0,2) rectangle (1,3);
\draw[fill=lightgray,draw=gray] (1,3) rectangle (2,4);
\draw[fill=lightgray,draw=gray] (2,4) rectangle (3,5);
\draw[fill=lightgray,draw=gray] (3,5) rectangle (4,6);
\draw[fill=lightgray,draw=gray] (4,6) rectangle (5,7);
\draw[fill=lightgray,draw=gray] (5,7) rectangle (6,8);
\draw[fill=lightgray,draw=gray] (6,8) rectangle (7,9);
\draw[fill=lightgray,draw=gray] (7,9) rectangle (8,10);
\draw[fill=white,draw=gray] (0,1) rectangle (1,2);
\draw[fill=white,draw=gray] (1,2) rectangle (2,3);
\draw[fill=white,draw=gray] (2,3) rectangle (3,4);
\draw[fill=white,draw=gray] (3,4) rectangle (4,5);
\draw[fill=white,draw=gray] (4,5) rectangle (5,6);
\draw[fill=white,draw=gray] (5,6) rectangle (6,7);
\draw[fill=white,draw=gray] (6,7) rectangle (7,8);
\draw[fill=white,draw=gray] (7,8) rectangle (8,9);
\draw[fill=lightgray,draw=gray] (0,0) rectangle (1,1);
\draw[fill=lightgray,draw=gray] (1,1) rectangle (2,2);
\draw[fill=lightgray,draw=gray] (2,2) rectangle (3,3);
\draw[fill=lightgray,draw=gray] (3,3) rectangle (4,4);
\draw[fill=lightgray,draw=gray] (4,4) rectangle (5,5);
\draw[fill=lightgray,draw=gray] (5,5) rectangle (6,6);
\draw[fill=lightgray,draw=gray] (6,6) rectangle (7,7);
\draw[fill=lightgray,draw=gray] (7,7) rectangle (8,8);
\draw[fill=lightgray,draw=gray] (8,8) rectangle (9,9);
\draw[fill=white,draw=gray] (0,-1) rectangle (1,0);
\draw[fill=white,draw=gray] (1,0) rectangle (2,1);
\draw[fill=white,draw=gray] (2,1) rectangle (3,2);
\draw[fill=white,draw=gray] (4,3) rectangle (5,4);
\draw[fill=white,draw=gray] (6,5) rectangle (7,6);
\draw[fill=white,draw=gray] (8,7) rectangle (9,8);
            \draw[fill=white] (3.5,2.5) circle[radius=.1];
\draw[fill=white] (5.5,4.5) circle[radius=.1];
\draw[fill=white] (7.5,6.5) circle[radius=.1];
\draw[fill=black] (0.5,-0.5) circle[radius=.1];
\draw[fill=black] (1.5,0.5) circle[radius=.1];
\draw[fill=black] (2.5,1.5) circle[radius=.1];
\draw[fill=black] (4.5,3.5) circle[radius=.1];
\draw[fill=black] (6.5,5.5) circle[radius=.1];
\draw[fill=black] (8.5,7.5) circle[radius=.1];
        \end{tikzpicture}
        \caption{Case 1 (left) and Case 2 (right) domains for $\mu = (3,2,1,0,0,0)$, the type of partitions found in Section~\ref{conj}. By Theorem~\ref{case12eq}, these two generalized Aztec triangles have the same number of domino tilings.}
        \label{generic}
    \end{figure}

In this section we present product formulas to enumerate the
Case~1 and Case~2 sequences for the partition $\mu = (k, k-1, \ldots,
1, 0^{n-k})$.
Equivalently, these formulas count domino tilings of some generalized Aztec triangles, as
illustrated in Figure~\ref{generic}.
The following theorem restates Theorem~\ref{thm:1} from the introduction
ç1in a more explicit form.

    \begin{thm}
        \label{formula}
For $\mu = (k, k-1, \ldots, 1, 0^{n-k})$ and $\ell=2n$, the number of Case~1 sequences is given by
\begin{equation}
        \left.\prod_{i\ge 0} \left( \prod_{s=-2k+4i+1}^{-k+2i}(2n+s)\prod_{s=k-2i}^{2k-4i-2}(2n+s) \right) \middle/ \prod_{i=1}^{k-1} (2i+1)^{k-i} \right.
        .
\label{eq:F1}
\end{equation}
For $\mu = (k, k-1, \ldots, 1, 0^{n-k})$ and $\ell=2n+1$, the number of Case~2 sequences is given by
\begin{equation}
        \left. \prod_{i\ge 0} \left( \prod_{s=-2k+4i+1}^{-k+2i}(2n+s+1)\prod_{s=k-2i}^{2k-4i-2}(2n+s+1) \right) \middle/ \prod_{i=1}^{k-1} (2i+1)^{k-i} \right.
.
\label{eq:F2}
\end{equation}
    \end{thm}

There are two formulas to prove, namely \eqref{eq:F1} and~\eqref{eq:F2}.
We shall prove~\eqref{eq:F1} in Theorem~\ref{thm:detD1}, however not in the
form~\eqref{eq:F1}. We need to note that
\begin{align}
\notag
\prod_{i\ge 0} &\left.\left( \prod_{s=-2k+4i+1}^{-k+2i}(2n+s)\prod_{s=k-2i}^{2k-4i-2}(2n+s) \right) \middle/ \prod_{i=1}^{k-1} (2i+1)^{k-i} \right.\\
\notag
        & = \left.\prod_{i=0}^{\lfloor \frac{k-1}{2} \rfloor} \prod_{s=-2k+4i+1}^{-k+2i} (2n+s) \cdot \prod_{i=0}^{\lfloor \frac{k-2}{2} \rfloor} \prod_{s=k-2i}^{2k-4i-2} (2n+s) \cdot \prod_{i=1}^{k} \frac{(2(i-1))!!}{(2i-1)!} \right.\\
\notag
        & = 2^{\binom{k+1}{2}} \prod_{i=0}^{\lfloor \frac{k-1}{2} \rfloor} \prod_{s=-2k+4i+3}^{-k+2i+2} \left( n+\tfrac{s}{2}-1 \right) \cdot \prod_{i=0}^{\lfloor \frac{k-2}{2} \rfloor} \prod_{s=k-2i+2}^{2k-4i} \left( n+\tfrac{s}{2}-1 \right)\\
\notag
        &\qquad \cdot 2^{\binom{k}{2}} \prod_{i=1}^k \frac{(i-1)!}{(2i-1)!}\\
        & = 2^{k^2} \prod_{i=1}^k \frac{1}{(i)_i} \cdot \prod_{i=0}^{\lfloor \frac{k-1}{2} \rfloor} \prod_{s=-2k+4i+3}^{-k+2i+2} \left( n+\tfrac{s}{2}-1 \right) \cdot \prod_{i=0}^{\lfloor \frac{k-2}{2} \rfloor} \prod_{s=k-2i+2}^{2k-4i} \left( n+\tfrac{s}{2}-1 \right),
\label{eq:F1A}
\end{align}
    where $i!!$ is the product of all integers from 1 to $i$ with the same parity as $i$ and $(x)_i$ is the shifted factorial $(x)_i = \prod_{\alpha=0}^{i-1} (x+\alpha)$.
For even $k$ we further evaluate
\begin{align}
\notag
        \prod_{i=0}^{\lfloor \frac{k-1}{2} \rfloor} \prod_{s=-2k+4i+3}^{-k+2i+2}& \left( n+\tfrac{s}{2}-1 \right) \cdot \prod_{i=0}^{\lfloor \frac{k-2}{2} \rfloor} \prod_{s=k-2i+2}^{2k-4i} \left( n+\tfrac{s}{2}-1 \right)\\
\notag
        &= \prod_{i=0}^{\frac{k-2}{2}} \prod_{s=-k+2i+2}^{-\frac{k}{2}+i+1} (n+s-1) \cdot \prod_{i=0}^{\frac{k-2}{2}} \prod_{s=-k+2i+1}^{-\frac{k}{2}+i} (n+s-\tfrac{1}{2})\\
\notag
        &\qquad \cdot \prod_{i=0}^{\frac{k-2}{2}} \prod_{s=\frac{k}{2}-i+1}^{k-2i} (n+s-1) \cdot \prod_{i=0}^{\frac{k-2}{2}} \prod_{s=\frac{k}{2}-i+1}^{k-2i-1} (n+s-\tfrac{1}{2})\\
\notag
        &= \prod_{s=0}^{k-2} (n-s-1)^{\min \lbrace \lfloor (k-s)/2 \rfloor, \lfloor (s+2)/2 \rfloor \rbrace}\\
\notag
        &\qquad \cdot \prod_{s=0}^{k-1} (n-s-\tfrac{1}{2})^{\min \lbrace \lfloor (k-s+1)/2 \rfloor, \lfloor (s+1)/2 \rfloor \rbrace}\\
\notag
        &\qquad \cdot \prod_{s=0}^{k-2} (n+k-s-1)^{\min \lbrace \lfloor (k-s)/2 \rfloor, \lfloor (s+2)/2 \rfloor \rbrace}\\
        &\qquad \cdot \prod_{s=0}^{k-2} (n+k-s-\tfrac{1}{2})^{\min \lbrace \lfloor (k-s-1)/2 \rfloor, \lfloor (s+1)/2 \rfloor \rbrace}.
\label{eq:F1B}
\end{align}
    On the other hand, for odd $k$ we have
    \begingroup
    \allowdisplaybreaks
    \begin{align}
\notag
        \prod_{i=0}^{\lfloor \frac{k-1}{2} \rfloor} \prod_{s=-2k+4i+3}^{-k+2i+2} &\left( n+\tfrac{s}{2}-1 \right) \cdot \prod_{i=0}^{\lfloor \frac{k-2}{2} \rfloor} \prod_{s=k-2i+2}^{2k-4i} \left( n+\tfrac{s}{2}-1 \right)\\
\notag
        &= \prod_{i=0}^{\frac{k-1}{2}} \prod_{s=-k+2i+2}^{-\frac{k}{2}+i+1} (n+s-1) \cdot \prod_{i=0}^{\frac{k-1}{2}} \prod_{s=-k+2i+1}^{-\frac{k}{2}+i} (n+s-\tfrac{1}{2})\\
\notag
        &\qquad \cdot \prod_{i=0}^{\frac{k-3}{2}} \prod_{s=\frac{k}{2}-i+1}^{k-2i} (n+s-1) \cdot \prod_{i=0}^{\frac{k-3}{2}} \prod_{s=\frac{k}{2}-i+1}^{k-2i-1} (n+s-\tfrac{1}{2})\\
\notag
        &= \prod_{s=0}^{k-2} (n-s-1)^{\min \lbrace \lfloor (k-s)/2 \rfloor, \lfloor (s+1)/2 \rfloor \rbrace}\\
\notag
        &\qquad \cdot \prod_{s=0}^{k-1} (n-s-\tfrac{1}{2})^{\min \lbrace \lfloor (k-s+1)/2 \rfloor, \lfloor (s+2)/2 \rfloor \rbrace}\\
\notag
        &\qquad \cdot \prod_{s=0}^{k-2} (n+k-s-1)^{\min \lbrace \lfloor (k-s-1)/2 \rfloor, \lfloor (s+2)/2 \rfloor \rbrace}\\
        &\qquad \cdot \prod_{s=0}^{k-2} (n+k-s-\tfrac{1}{2})^{\min \lbrace \lfloor (k-s)/2 \rfloor, \lfloor (s+1)/2 \rfloor \rbrace}.
\label{eq:F1C}
    \end{align}
    \endgroup
By \eqref{eq:F1A}--\eqref{eq:F1C}, we see that the right-hand side in~\eqref{eq:Formel2}
agrees with~\eqref{eq:F1}.

We now argue that it is sufficient to prove one of~\eqref{eq:F1} and~\eqref{eq:F2}.
In order to see this, 
we observe that the formula~\eqref{eq:F2} for Case~2 sequences can be obtained
from that of Case~1 sequences, namely~\eqref{eq:F1},
by replacing~$n$ in~\eqref{eq:F1} by~$n+\frac {1} {2}$.

    In order to complete the argument that it suffices to prove one of~\eqref{eq:F1}
and~\eqref{eq:F2},
we now show that a similar relation holds when computing
the determinants from Corollaries~\ref{det1} and~\ref{det2}, which provide formulas
for the Case~1 and Case~2 sequences in Theorem~\ref{formula}.
More precisely,
in the case of $\mu = (k, \ldots, 1, 0^{n-k})$, Corollary~\ref{det1} gives us
the $n \times n$ matrix
        \[
        A_1 = \left( D(\mu_{n-j+1}+j-i, i-1) \right)_{1 \leq i,j \leq n},
        \]
        whose determinant is determined by its bottom right $k \times k$ entries
\begin{equation}
        D_1 = \left( D(2j-i, i+n-k-1) \right)_{1 \leq i,j \leq k}.
        \label{D1first}
\end{equation}
        Similarly,
for the same choice of~$\mu$, Corollary~\ref{det2} gives us
the $n \times n$ matrix
        \[
        A_2 = \left( H(\mu_{n-j+1}+j-i, i-1) \right)_{1 \leq i,j \leq n},
        \]
whose determinant is again
determined by the bottom right $k \times k$ entries
        \[
        D_2 = \left( H(2j-i, i+n-k-1) \right)_{1 \leq i,j \leq k}.
        \]
In particular $D_1$ and $D_2$ depend only on
$k$ and $n$, so we can define $D_1(k;n)$ and $D_2(k;n)$ to be the $k \times k$ matrices above.
These two determinants satisfy a relation analogous to the one observed
for the product formulas in~\eqref{eq:F1} and \eqref{eq:F2}.
        
\begin{thm} \label{D1=D2}
For  $n,k \in \mathbb{Z}_{\geq 0}$, we have
        \[
        \det(D_1(k; n+\tfrac{1}{2})) = \det(D_2(k;n)).
        \]
        \label{detprop}
\end{thm}

The entries of $D_1$ are of the form $D(i,j)$ for some $i,j \in \mathbb{Z}$. To make sense of this theorem let us redefine $D(i,j)$ using
identity~\eqref{eq:f3} in Proposition~\ref{H_identities}:
\[
D(i,j) = \sum_{m=0}^i \binom{i}{m} \binom{j}{m} 2^m.
\]
However, now we allow any $j \in \mathbb{R}$.
The reader should recall that our binomial coefficients are defined according
to the convention in~\eqref{eq:binom}.

\begin{proof}[Proof of Theorem \ref{detprop}]
Our claim follows from Proposition~\ref{H1half} below, where we show 
that, for $i,j \in \mathbb{Z}_{\geq -1}$, we have
    \[
    D(i,j+\tfrac{1}{2}) = \sum_{l=0}^{\lceil i/2 \rceil} (-1)^l \binom{-\tfrac{1}{2}}{l} H(i-2l, j).
    \]
Then our definitions of $D_1$ and $D_2$
give us that $D_1(k; n+\tfrac{1}{2})$ can be obtained by column operations
from $D_2(k;n)$, hence proving that their determinants are equal.
\end{proof}

    \begin{lem}
        \label{recurrence}
        For all $i \in \mathbb{Z}$ and $ j \in \mathbb{R}$ we have
        \[
        D(i,j) = D(i-1, j) + D(i-1, j-1) + D(i, j-1),
        \]
where $D(i,j)$ is defined by
        \[
        D(i,j) = \sum_{m=0}^i \binom{i}{m} \binom{j}{m} 2^m.
        \]
\label{recford}
    \end{lem}
    
    \begin{proof}
This is part of
Remark~\ref{H1_rem}, but as we cannot use our combinatorial definition of $D(i,j)$ we compute directly using the classical recurrence for binomial coefficients:
        \begin{align*}
            D(i,j) &= \sum_{m=0}^i \binom{i}{m} \binom{j}{m} 2^m\\
            &= \sum_{m=0}^{i-1} \binom{i-1}{m} \binom{j}{m} 2^m + \sum_{m=1}^{i} \binom{i-1}{m-1} \binom{j}{m} 2^m\\
            &= D(i-1, j) + 2\sum_{m=1}^{i} \binom{i-1}{m-1} \binom{j-1}{m-1} 2^{m-1} \sum_{m=1}^{i} \binom{i-1}{m-1} \binom{j-1}{m} 2^m \\
            &= D(i-1, j) + D(i-1, j-1) + \sum_{m=1}^{i} \binom{i-1}{m-1} \binom{j-1}{m} 2^m \\
            &\qquad+ \sum_{m=0}^{i-1} \binom{i-1}{m} \binom{j-1}{m} 2^m\\
            &= D(i-1, j) + D(i-1, j-1) + D(i, j-1),
        \end{align*}
        noting that
        \[
        \binom{i}{m} = \binom{i-1}{m-1}+ \binom{i-1}{m}
        \]
        still holds for $i \in \mathbb{R}$.
    \end{proof}
    
    \begin{prop}
        \label{H1half}
For $i,j \in \mathbb{Z}_{\geq -1}$, we have
        \[
        D(i,j+\tfrac{1}{2}) = \sum_{\ell=0}^{\lceil i/2 \rceil} (-1)^l \binom{-\tfrac{1}{2}}{\ell} H(i-2\ell, j),
        \]
        where $D(i,j)$ is defined by
        \[
        D(i,j) = \sum_{m=0}^i \binom{i}{m} \binom{j}{m} 2^m.
        \]
    \end{prop}
    
    \begin{proof}
We do an induction on $i$ and $j$. For
the inductive step, we use Lemma~\ref{recford} to find for $i,j \geq 0$ that
        \begin{align*}
            D(i, j+\tfrac{1}{2}) &= D(i-1, j+\tfrac{1}{2}) + D(i-1, j-\tfrac{1}{2}) + D(i, j-\tfrac{1}{2})\\
            &= \sum_{l=0}^{\lceil (i-1)/2 \rceil} (-1)^l \binom{-\tfrac{1}{2}}{l} H(i-1-2l, j)\\
            &\qquad+ \sum_{l=0}^{\lceil (i-1)/2 \rceil} (-1)^l \binom{-\tfrac{1}{2}}{l} H(i-1-2l, j-1)\\
            &\qquad+ \sum_{l=0}^{\lceil i/2 \rceil} (-1)^l \binom{-\tfrac{1}{2}}{l} H(i-2l, j-1)\\
            &= \sum_{l=0}^{\lceil i/2 \rceil} (-1)^l \binom{-\tfrac{1}{2}}{l} H(i-2l, j),
        \end{align*}
        noting that $H(i,j) = 0$ for $i < 0$ and all~$j$. Now, for the base case, when $i = -1$ we have
        \[
        D(-1, j+\tfrac{1}{2}) = 0 = H(-1, j)
        \]
        for all $j$, as desired.
        
        It remains to be seen that the case $j=-1$ holds, that is
        \[
        D(i,-\tfrac{1}{2}) = \sum_{l=0}^{\lceil i/2 \rceil} (-1)^l \binom{-\tfrac{1}{2}}{l} H(i-2l, -1)
        \]
        for all $i \geq 0$. Now 
        \[
        H(i, -1) = \begin{cases} 1, & \text{if }i = 0, \\ 0, & \text{otherwise}, \end{cases}
        \]
        so we wish to show
        \begin{equation}
        D(i,-\tfrac{1}{2}) = \sum_{m=0}^{i} \binom{i}{m} \binom{-\tfrac{1}{2}}{m} 2^m = \begin{cases} 0, & \text{if }i \text{ is odd}, \\ \left| \dbinom{-\tfrac{1}{2}}{\tfrac{i}{2}} \right| ,& \text{if }i \text{ is even}. \end{cases}
\label{eq:wish}
        \end{equation}
Using the notation of hypergeometric series (cf.\ \cite{SlatAC}),  we have
        \[
        \sum_{m=0}^{i} \binom{i}{m} \binom{-\tfrac{1}{2}}{m} 2^m
         \sum_{m=0}^{i} \binom{i}{i-m} \binom{-\tfrac{1}{2}}{i-m} 2^{i-m} =
\frac {\big(\frac {1} {2}-i\big)_i2^i} {i!}
        _2F_1 \left[
          \begin{matrix} -i,-i\\\frac {1} {2}-i \end{matrix}; \frac {1} {2} \right].
        \]
The $_FF_1$-series can be evaluated by means of Gau\ss's second
summation theorem (see \cite[Eq.~(1.7.1.9); Appendix~(III.6)]{SlatAC})
$$
{} _{2} F _{1} \!\left [ \begin{matrix} { 2 a, 2 b}\\ { {\frac 1 2} + a + b}\end{matrix}
   ; {\displaystyle {\frac 1 2}}\right ]  =
\frac {\Gamma( {\frac 1 2})\,\Gamma( {\frac 1 2} + a + b)}
      {\Gamma({\frac 1 2} + a)\,\Gamma(   {\frac 1 2} + b)}.
$$
Thus, after simplification using the well-known facts that $\Gamma(\frac {1} {2})=\sqrt\pi$ and $\Gamma(x)\Gamma(1-x)=\pi/\sin(\pi x)$,  we obtain
$$
\sum_{m=0}^{i} \binom{i}{m} \binom{-\tfrac{1}{2}}{m} 2^m
=\frac {2^i\Gamma^2\big(\frac {i+1} {2}\big)\sin^2\big(\frac {1-i} {2}\big)} {\pi\cdot i!}.
$$
Here we see that the sine vanishes for odd~$i$. On the other hand, if $i$ is even, then we have
\begin{align*}
\frac {2^i\Gamma^2\big(\frac {i+1} {2}\big)\sin^2\big(\frac {1-i} {2}\big)} {\pi\,i!}
&=\frac {2^i\left(\frac {i-1} {2}\right)^2
  \left(\frac {i-3} {2}\right)^2\cdots\left(\frac {1} {2}\right)^2
\Gamma^2\left(\frac {1} {2}\right) } {\pi\,i!}\\
&=\frac {\left(\frac {i-1} {2}\right)
  \left(\frac {i-3} {2}\right)\cdots\left(\frac {1} {2}\right)
  } {\left(\frac {i} {2}\right)!},
\end{align*}
which is equal to the second alternative on the right-hand side of~\eqref{eq:wish}.
        This concludes the base case of our induction on $j$ and we have our result.
    \end{proof}

Next we restate --- and prove, conditional on the truth of Theorem~\ref{thm:1} ---
Theorem~\ref{thmbij} from the introduction.

    \begin{thm} \label{thm:DD}
      The number of Case~1 sequences for the partition $\mu=(k+1, k, \ldots, 1)$ with $\ell = 2k+2$ is equal to the number of Case~2 sequences for the partition $\mu = (k, k-1, \ldots, 1, 0)$ with $\ell = 2k+3$.
        \label{case12eq}
    \end{thm}

Figure~\ref{generic} shows two generalized Aztec triangles that,
according to Theorem~\ref{case12eq},
must have an equal number of domino tilings. 

    \begin{proof}[Proof of Theorem \ref{case12eq}]
Using our formulas from Theorem~\ref{formula}
we wish to show that
        \begin{align*}
            &\left. \prod_{i=0}^{\lfloor \frac{k}{2} \rfloor}\ \prod_{s=-2k+4i-1}^{-k+2i-1} (2k+s+2) \cdot \prod_{i=0}^{\lfloor \frac{k-1}{2} \rfloor} \prod_{s=k-2i+1}^{2k-4i} (2k+s+2) \middle/ \prod_{i=1}^{k} (2i+1)^{k+1-i} \right. \\
          & = \left. \prod_{i=0}^{\lfloor \frac{k-1}{2} \rfloor} \prod_{s=-2k+4i+1}^{-k+2i} (2k+s+3) \cdot \prod_{i=0}^{\lfloor \frac{k-2}{2} \rfloor}
\prod_{s=k-2i}^{2k-4i-2} (2k+s+3) \middle/ \prod_{i=1}^{k-1} (2i+1)^{k-i} \right.
.
        \end{align*}
        This is an equality of two fractions; we bring both numerators to the left-hand side and both denominators to the right-hand side:
        \begin{align*}
            &\frac{\prod_{i=0}^{\lfloor \frac{k}{2} \rfloor} \prod_{s=-2k+4i-1}^{-k+2i-1} (2k+s+2)}{\prod_{i=0}^{\lfloor \frac{k-1}{2} \rfloor} \prod_{s=-2k+4i+2}^{-k+2i+1} (2k+s+2)} \cdot \frac{\prod_{i=0}^{\lfloor \frac{k-1}{2} \rfloor} \prod_{s=k-2i+1}^{2k-4i} (2k+s+2)}{\prod_{i=0}^{\lfloor \frac{k-2}{2} \rfloor} \prod_{s=k-2i+1}^{2k-4i-1} (2k+s+2)}= \prod_{i=1}^{k} (2i+1).
        \end{align*}
If $k$ is even this results in
$$
(2k+1) \prod_{i=0}^{\frac{k}{2} - 1} \frac{(4i+1)(4i+2)(4i+3)(4k-4i+2)}{(k+2i+2)(k+2i+2)} = \prod_{i=1}^{k} (2i+1),
$$
or, after cancellation,
$$
            \prod_{i=0}^{\frac{k}{2} - 1} \frac{(4i+2)(4k-4i+2)}{(k+2i+2)(k+2i+3)} = 1 .
$$
Similarly, if $k$ is odd we have
$$
\frac{2k+4}{2k+4} \prod_{i=0}^{\frac{k-1}{2}} \frac{(4i+1)(4i+2)(4i+3)(4k-4i+2)}{(k+2i+2)(k+2i+3)} = \prod_{i=1}^{k} (2i+1),
$$
or, after cancellation,
$$
            \prod_{i=0}^{\frac{k-1}{2}} \frac{(4i+2)(4k-4i+2)}{(k+2i+2)(k+2i+3)} = 1.
$$
        We can verify that both of these identities do hold true. For even $k$ the product on the left-hand side expands as
        \begin{align*}
            \frac{2 \cdot 6 \cdots (2k-2) \cdot (2k+6)(2k+10) \cdots (4k+2)}{(k+2)(k+3) \cdots (2k+1)} &= \frac{2^{k/2} (k-1)!!\, 2^{k/2} \frac{(2k+1)!!}{(k+1)!!}}{\frac{(2k+1)!}{(k+1)!}}\\
            &= \frac{2^k (k-1)!!\,(k+1)!}{(k+1)!!} \cdot \frac{(2k+1)!!}{(2k+1)!}\\
            &= 2^k k! \cdot \frac{1}{(2k)!!}= 1,
        \end{align*}
        where again we use 
the double factorial to indicate a product of integers with the same parity. 
        Similarly, for odd $k$ we have
        \begin{align*}
            \frac{2 \cdot 6 \cdots (2k) \cdot (2k+4)(2k+8) \cdots (4k+2)}{(k+2)(k+3) \cdots (2k+2)} &= \frac{2^{(k+1)/2} k!!\, 2^{(k+1)/2} \frac{(2k+1)!!}{k!!}}{\frac{(2k+2)!}{(k+1)!}}\\
            &= 2^{k+1} (k+1)! \cdot \frac{(2k+1)!!}{(2k+2)!} = 1. \qedhere      
        \end{align*}
    \end{proof}

\begin{rem}
Finally let us verify that Theorem~\ref{formula}
is indeed a generalization of Di~Francesco's conjectured formula
(cf.\ \cite[Conj.~8.1]{difrancesco2})
        \[
        F(n) := 2^{n(n-1)/2} \prod_{i=0}^{n-1} \frac{(4i+2)!}{(n+2i+1)!}.
        \]
Indeed,
if $\mu = (n, n-1, \ldots, 1)$, Theorem~\ref{formula} gives us that the number of Case~1 sequences is equal to
        \begin{align*}
            G(n) &:= \left.\prod_{i \ge 0} \left( \prod_{s=-2n+4i+1}^{-n+2i}(2n+s)\prod_{s=n-2i}^{2n-4i-2}(2n+s) \right) \middle/ \prod_{i=1}^{n-1} (2i+1)^{n-i} \right. \\
            &= \left.\prod_{i \ge 0} \left( \prod_{s=4i+1}^{n+2i} s \prod_{s=3n-2i}^{4n-4i-2} s \right) \middle/ \prod_{i=1}^{n-1} (2i+1)^{n-i} \right.
.            
        \end{align*}
        Now we compute for $n \geq 1$
        \begin{align*}
            \frac{F(n+1)}{F(n)} &= 2^n \cdot \frac{(4n+2)!}{(3n+2)!} \prod_{i=0}^{n-1} \frac{(n+2i+1)!}{(n+2i+2)!} = 2^n \cdot \frac{(4n+2)!}{(3n+2)!} \prod_{i=0}^{n-1} \frac{1}{n+2i+2}\\
            &= \frac{2^n (4n+2)! n!!}{(3n+2)! (3n)!!}.
        \end{align*}
        
        Now we compute $\frac{G(n+1)}{G(n)}$ for all $n \geq 1$. We first note that for even $n$ we have
        \[
        G(n) = \left. \left( \prod_{i=0}^{\frac{n-2}{2}} \prod_{s=4i+1}^{n+2i} s \cdot \prod_{i=0}^{\frac{n-2}{2}} \prod_{s=3n-2i}^{4n-4i-2} s \right) \middle / \prod_{i=1}^{n-1} (2i+1)^{n-i} \right.
,
        \]
        and for odd $n$ we have
        \[
        G(n) = \left. \left( \prod_{i=0}^{\frac{n-1}{2}} \prod_{s=4i+1}^{n+2i} s \cdot \prod_{i=0}^{\frac{n-3}{2}} \prod_{s=3n-2i}^{4n-4i-2} s \right) \middle / \prod_{i=1}^{n-1} (2i+1)^{n-i} \right.
.
        \]
        Then, if $n$ is odd, we have
        \begin{align*}
            \frac{G(n+1)}{G(n)} &= \left. \left( \frac{\prod_{i=0}^{\frac{n-1}{2}} \prod_{s=4i+1}^{n+2i+1} s \cdot \prod_{i=0}^{\frac{n-1}{2}} \prod_{s=3n-2i+3}^{4n-4i+2} s}{\prod_{i=0}^{\frac{n-1}{2}} \prod_{s=4i+1}^{n+2i} s \cdot \prod_{i=0}^{\frac{n-3}{2}} \prod_{s=3n-2i}^{4n-4i-2} s} \right) \middle / \left( \frac{\prod_{i=0}^{n} (2i+1)^{n+1-i}}{\prod_{i=0}^{n-1} (2i+1)^{n-i}} \right) \right.\\
            &\kern-3pt
            = \left. \prod_{i=0}^{\frac{n-1}{2}} (n+2i+1)  \cdot \frac{(2n+4)!}{(2n+3)!} \cdot \prod_{i=0}^{\frac{n-3}{2}} \frac{(4n-4i-1) \cdots (4n-4i+2)}{(3n-2i) \cdots (3n-2i+2)} \middle / (2n+1)!! \right.\\
            &\kern-3pt
            = \frac{(2n)!!}{(n-1)!!} \cdot \frac{(2n+4)!}{(2n+3)!} \cdot \frac{\frac{(4n+2)!}{(2n+4)!}}{\frac{(3n+2)!}{(2n+3)!} \cdot \frac{(3n)!!}{(2n+1)!!}} \cdot \frac{1}{(2n+1)!!}\\
            &\kern-3pt
            =\frac{2^n n! (4n+2)!}{(n-1)!! (3n+2)! (3n)!!} = \frac{2^n (4n+2)! n!!}{(3n+2)! (3n)!!}.
        \end{align*}
        On the other hand, if $n$ is even, we have
        \begin{align*}
            \frac{G(n+1)}{G(n)} &= \left. \left( \frac{\prod_{i=0}^{\frac{n}{2}} \prod_{s=4i+1}^{n+2i+1} s \cdot \prod_{i=0}^{\frac{n-2}{2}} \prod_{s=3n-2i+3}^{4n-4i+2} s}{\prod_{i=0}^{\frac{n-2}{2}} \prod_{s=4i+1}^{n+2i} s \cdot \prod_{i=0}^{\frac{n-2}{2}} \prod_{s=3n-2i}^{4n-4i-2} s} \right) \middle / \left( \frac{\prod_{i=0}^{n} (2i+1)^{n+1-i}}{\prod_{i=0}^{n-1} (2i+1)^{n-i}} \right) \right.\\
            &= \left. (2n+1) \cdot \prod_{i=0}^{\frac{n-2}{2}} \left( (n+2i+1) \frac{(4n-4i-1) \cdots (4n-4i+2)}{(3n-2i) \cdots (3n-2i+2)} \right) \middle / (2n+1)!! \right.\\
            &= (2n+1) \cdot \frac{(2n-1)!!}{(n-1)!!} \cdot \frac{\frac{(4n+2)!}{(2n+2)!}}{\frac{(3n+2)!}{(2n+2)!} \cdot \frac{(3n)!!}{(2n)!!}} \cdot \frac{1}{(2n+1)!!}\\
            &=  \frac{2^n (4n+2)! n!!}{(3n+2)! (3n)!!}.
        \end{align*}
        Along with the fact that $F(1) = G(1) = 1$, we 
conclude that $F(n) = G(n)$ for all $n \geq 1$. 
Hence indeed, Di~Francesco's conjectured formula \cite[Conj.~8.1]{difrancesco2} results
as a special case from Theorem~\ref{formula}.
    \end{rem}

    \section{Proof of Theorem \ref{formula}}
    \label{mainproof}

\def\al{\alpha}
\def\be{\beta}
\def\de{\delta}
\def\ga{\gamma}
\def\ep{\varepsilon}
\def\la{\lambda}
\def\om{\omega}
\def\Om{\Omega}
\def\rh{\rho}
\def\si{\sigma}
\def\ta{\tau}
\def\De{\bigtriangleup}
\def\Na{\bigtriangledown}
\def\Z{{\mathbb Z}}
\def\V#1{\Vert #1\Vert}
\def\po#1#2{(#1)_#2}
\def\coef#1{\left\langle#1\right\rangle}
\def\ord{\operatorname{ord}}
\def\iso{\operatorname{\#iso}}
\def\End{\operatorname{end}}
\def\Estring{\operatorname{\#1str}}
\def\Pol{\operatorname{Pol}}
\def\aa{{\mathbf a}}
\def\bb{{\mathbf b}}
\def\cc{{\mathbf c}}
\def\kk{{\mathbf k}}
\def\ddd{j}
\def\fl#1{\left\lfloor#1\right\rfloor}
\def\cl#1{\left\lceil#1\right\rceil}
\def\ii{\mathbf i}
\def\CC{\mathcal C}
\def\vv{\sqrt v}

We have shown in the previous section that, for a proof of Theorem~\ref{formula},
it suffices to prove one of~\eqref{eq:F1} and~\eqref{eq:F2}. In the theorem below we
prove~\eqref{eq:F1}. 

\begin{thm} \label{thm:detD1}
For all non-negative integers~$k$ and arbitrary~${n}$, we have
\begin{align} 
\notag
D_1(k;{n})&:=\det_{0\le i,j\le k-1}\big(D(k - 2 i + j, {n} - j - 1)\big)\\
\notag
&=2^{k^2} \prod_{i=1} ^{k}\frac {1} {(i)_i}
\prod_{s=0} ^{k-2}({n} - s - 1)^{\min\{\fl{(s + 1 + \chi(k\text{ even}))/2}, 
  \fl{(k - s)/2}\}}\\
\notag
&\kern 1cm\cdot
 \prod_{s=0} ^{k-1}({n} - s - \tfrac {1} {2})^{\min\{\fl{(s + 1 + \chi(k\text{ odd}))/2}, 
     \fl{(k - s + 1)/2}\}}\\
\notag
&\kern1cm\cdot
 \prod_{s=0} ^{k-2}({n} + k - s - 1)^{\min\{\fl{(s + 2)/2}, 
     \fl{(k - s - \chi(k\text{ odd}))/2}\}}\\
&\kern1cm\cdot
 \prod_{s=1} ^{k-2}({n} + k - s - \tfrac {1} {2})^{\min\{\fl{(s + 1)/2}, 
     \fl{(k - s - \chi(k\text{ even}))/2}\}}.
\label{eq:Formel2}
\end{align}
Here, $\chi(\mathcal S)=1$ if $\mathcal S$ is true and $\chi(\mathcal S)=0$
otherwise.
\label{main}
\end{thm}

The reader should 
note that the above determinant agrees with the one
in~\eqref{D1first}
by applying the replacements $i \mapsto k-j$, $j \mapsto k-i$.
Furthermore, it should be recalled that by~\eqref{eq:F1A}--\eqref{eq:F1C}
the right-hand side of~\eqref{eq:Formel2} indeed agrees with~\eqref{eq:F1}.

\begin{proof}[Proof of Theorem \ref{main}]
We apply the identification of factors method as
described in \cite[Sec.~2.4]{KratBN}. An outline is as follows.
In Steps~1--4 we show that the products over~$s$ in~\eqref{eq:Formel2}
are indeed factors of the determinant $D_1(k;{n})$ as a polynomial
in~${n}$. In Step~5 we prove that the degree of $D_1(k;{n})$ as a
polynomial in~${n}$ is at most $\binom {k+1}2$. Since the total number
of factors provided by the products over~$s$ in~\eqref{eq:Formel2} is
exactly $\binom {k+1}2$, the two sides in~\eqref{eq:Formel2} must be
the same up
to a multiplicative constant. This constant is then computed
by determining the leading coefficient of $D_1(k;{n})$.

\medskip\noindent
{\sc Step 1:} {\it  The term $({n}-s-1)^{\min\{
\fl{(s+1+\chi(k\text{ \rm even}))/2},\fl{(k-s)/2}\}}$ is a factor of $D_1(k;{n})$ as a
polynomial in~${n}$.} 
We claim that, for $0\le a\le s$, $k\ge 2s-a+2$, and
\hbox{$k\equiv a$}~(mod~2), we have
\begin{equation} \label{eq:lin1}
\sum_{j=a}^{2s-a+1}\binom {2s-2a+1}{j-a}\cdot
\big(\text{column $j$ of $D_1(k;s+1)$}\big)=0.
\end{equation}
The validity of these linear relations does indeed imply that the mentioned
power of\break $(n-s-1)$ divides $D_1(k,n)$ as a polynomial in~$n$. For, with varying~$a$,
all these linear relations are linearly independent of each other since all
of them start with a different column. Hence, the number of these linear
relations --- that is, the number of different~$a$'s satisfying the above conditions ---
gives (a lower bound of) the exponent of $(n-s-1)$ (cf.\ \cite[p.~13]{KratBN}). 
Now, on the one hand, the condition $0\le a\le s$ together with the
restriction $k\equiv a$~(mod~2) implies that there are at most
$\fl{(s+1+\chi(k\text{ even}))/2}$ such~$a$'s. On the other hand,
the conditions $k\ge 2s-a+2$ and  $a\le s$ together with the
restriction $k\equiv a$~(mod~2) imply that there are at most
$\fl{(k-s)/2}$ such~$a$'s. Consequently, the minimum of the two is
(a lower bound for) the exponent.

Since there will be a very similar claim to prove in Step~2, we start
the required computation at a more general level than needed for the
current step. Namely, we compute the entry in row~$i$ of
\begin{equation} \label{eq:lin1a}
\sum_{j=a}^{a+b}\binom {b}{j-a}\cdot
\big(\text{column $j$ of $D_1(k;{n})$}\big),
\end{equation}
with the understanding that, here, we need the case where $b=2s-2a+1$
and~${n-1}=s$. Indeed, using the expression~\eqref{eq:f4} for $D(k,{n-1})$,
the entry in row~$i$ of \eqref{eq:lin1a} equals
\begin{align}
\notag
\sum_{j=a}^{a+b}\binom b{j-a}&
\sum_{\ell=0}^{k-2i+j}\binom {{n-1}+k-2i-\ell}{{n-1}-j-\ell,k-2i+j-\ell,\ell}\\
\notag
&=\sum_{\ell=0}^{k-2i+a+b}\binom {{n-1}+k-2i-\ell}{\ell}
\sum_{j=a}^{a+b}\binom b{a+b-j}\binom {{n-1}+k-2i-2\ell}{k-2i+j-\ell}\\
\label{eq:Ausdr1}
&=\sum_{\ell=0}^{k-2i+a+b}\binom {{n-1}+k-2i-\ell}{\ell}\binom {{n-1}+k+b-2i-2\ell}{k+a+b-2i-\ell}.
\end{align}
We now substitute ${n-1}=s$ and $b=2s+1-2a$ with $0\le a\le s$ in~\eqref{eq:Ausdr1}.
We see that the entry in row~$i$ of \eqref{eq:lin1} is given by
\begin{equation} \label{eq:S1} 
\sum_{\ell=0}^{k-2i+2s-a+1}\binom {s+k-2i-\ell}{\ell}\binom
    {k+3s-2a+1-2i-2\ell}{k+2s-a+1-2i-\ell}.
\end{equation}
We reverse the order of summation, and obtain that it is also given by
\begin{multline*}
\sum_{\ell=0}^{k-2i+2s-a+1}\binom {-s+a-1+\ell}{k-2i+2s-a+1-\ell}
\binom {-k-s-1+2i+2\ell}{\ell}\\
=\kern-2pt
\sum_{\ell=0}^{k-2i+2s-a+1}(-1)^{k-2i+2s-a+1-\ell}\binom {k+3s-2a-2i-2\ell+1}{k-2i+2s-a+1-\ell}
(-1)^\ell\binom {k+s-2i-\ell}{\ell}.
\end{multline*}
Since we assumed that $k\equiv a$~(mod~2), the sum in \eqref{eq:S1}
vanishes, proving our claim.

\medskip\noindent
{\sc Step 2:} {\it The term $({n}-s-\frac {1} {2})^{\min\{
\fl{(s+1+\chi(k\text{ \rm odd}))/2},\fl{(k-s+1)/2}\}}$ is a factor of $D_1(k;{n})$ as a
polynomial in~${n}$.} 
We claim that, for $0\le a\le s$, $k\ge 2s-a+1$, and
\hbox{$k\not\equiv a$}~(mod~2), we have
\begin{equation} \label{eq:lin2}
\sum_{j=a}^{2s-a}\binom {2s-2a}{j-a}\cdot
\big(\text{column $j$ of $D_1(k;s+1)$}\big)=0.
\end{equation}
We (re)use the earlier computation and
let ${n-1}=s-\frac {1} {2}$, $b=2s-2a$ with $0\le a\le s$ in~\eqref{eq:Ausdr1}.
We see that the entry in row~$i$ of \eqref{eq:lin2} is given by
\begin{equation} \label{eq:S2} 
\sum_{\ell=0}^{k-2i+2s-a}\binom {s-\frac {1} {2}+k-2i-\ell}{\ell}\binom
    {k+3s-2a-\frac {1} {2}-2i-2\ell}{k+2s-a-2i-\ell}.
\end{equation}
We reverse the order of summation, and obtain that it is also given by
\begin{multline*}
\sum_{\ell=0}^{k-2i+2s-a}\binom {-s-\frac {1} {2}+a+\ell}{k-2i+2s-a-\ell}\binom
    {-k-s-\frac {1} {2}+2i+2\ell}{\ell}\\
=\kern-2pt
\sum_{\ell=0}^{k-2i+2s-a}(-1)^{k-2i+2s-a-\ell}\binom {
  k-2i+3s-2a-2\ell-\frac {1} {2}}{k-2i+2s-a-\ell}(-1)^\ell\binom
    {k+s-2i-\ell-\frac {1} {2}}{\ell}.
\end{multline*}
Since we assumed that $k\not\equiv a$~(mod~2), the sum in \eqref{eq:S2}
vanishes, proving our claim.

\medskip\noindent
{\sc Step 3:} {\it  The term $({n} + k - s-1)^{\min\{\fl{(s + 2)/2}, 
     \fl{(k - s - \chi(k\text{ \rm odd}))/2}\}}$ is a factor of $D_1(k;{n})$ as a
  polynomial in~${n}$.} 
We claim that, for $0\le 2a\le s$ and $k\ge 2s-2a+2$, we have
\begin{multline} \label{eq:S3}
\sum_{i=a}^{s+1-a}(-1)^i\binom {s+1-2a}{i-a}\cdot
\big(\text{row $i$ of $D_1(k;-k+s+1)$}\big)\\
-
\sum_{i=s+1-a}^{k-1}2^{2s+2-4a}\binom {i-a-1}{s-2a}\cdot
\big(\text{row $i$ of $D_1(k;-k+s+1)$}\big)=0.
\end{multline}

Here we use a contour integral expression for $D(k,{n-1})$.
Namely, by \eqref{eq:f3}, we have
\begin{align} \notag
D(k,{n-1})&=\sum_{\ell=0}^k\binom k\ell\binom {{n-1}}\ell2^\ell\\
\notag
&=\frac {1} {(2\pi\ii)^2}\int_{\CC_x}\int_{\CC_y}\sum_{\ell\ge0}
2^\ell\frac {(1+x)^k(1+y)^{n-1}} {(xy)^{\ell+1}}\,dx\,dy\\
&=\frac {1} {(2\pi\ii)^2}\int_{\CC_x}\int_{\CC_y}
\frac {(1+x)^k(1+y)^{n-1}} {xy-2}\,dx\,dy.
\label{eq:f-int}
\end{align}
In order that the geometric series in the next-to-last line be
convergent, we need to have $|xy|>2$. Furthermore, since later we want
to use this integral expression also for negative values of~${n-1}$, which
causes a singularity of the integrand at $y=-1$, the contour $\CC_y$ must not
contain $-1$ in its interior. Thus, a valid choice is\footnote{\label{foot:1}
This choice affords 
actually {\it two} wanted properties: (1) as mentioned, it interprets $D(k,{n-1})$ as
a polynomial in~${n}$; (2) it vanishes if $k$ is a negative integer.
To see the latter, suppose that $k$ is a negative integer. Then the numerator
of the integrand in~\eqref{eq:f-int} does not contain~$x$, while the denominator
is a polynomial in~$x$ of degree at least~2. Taking recourse to a classical
argument, we now blow up the contour $\CC_x$ to $\lambda\CC_x$ with
$\lambda\to\infty$. 
The integrand being of the order $O(x^ {-2})$, we see that,
in the limit, the integral vanishes.
However, since (at $\lambda=1$) we started with a circle of
radius~5, we encounter no singularities when performing this blowup,
showing that the original integral vanishes.}
\begin{equation} \label{eq:contour} 
\CC_x=\{5e^{\ii t}:0\le  t \le 2\pi\}
\quad \text{and}\quad 
\CC_y=\{0.5e^{\ii w}:0\le w\le 2\pi\}.
\end{equation}

Next we compute the entry in column~$j$ of the first sum in~\eqref{eq:S3}:
\begin{align*}
\sum_{i=a}^{s+1-a}(-1)^i&\binom {s+1-2a}{i-a}D(k-2i+j,-k+s-j)\\
&=\sum_{i=a}^{s+1-a}(-1)^i\binom {s+1-2a}{i-a}
\frac {1} {(2\pi\ii)^2}\int_{\CC_x}\int_{\CC_y}
\frac {(1+x)^{k-2i+j}(1+y)^{-k+s-j}} {xy-2}\,dx\,dy\\
&=
\frac {1} {(2\pi\ii)^2}\int_{\CC_x}\int_{\CC_y}
\frac {(1+x)^{k-2a+j}(1+y)^{-k+s-j}} {xy-2}
\left(1-\frac {1} {(1+x)^2}\right)^{s+1-2a}\,dx\,dy\\
&=
\frac {1} {(2\pi\ii)^2}\int_{\CC_x}\int_{\CC_y}
\frac {(x(2+x))^{s+1-2a}(1+x)^{k+2a+j-2s-2}(1+y)^{-k+s-j}} {xy-2}
\,dx\,dy.
\end{align*}
We apply the residue theorem with respect to~$y$. There is only one
singularity inside the contour $\CC_y$, namely $y=2/x$. (Recall
that we chose $\CC_y$ so that it does not contain $-1$ in its interior.)
Therefore, the above integral equals
\begin{multline} \label{eq:int3}
\frac {1} {2\pi\ii}\int_{\CC_x}
\frac {(x(2+x))^{s+1-2a}(1+x)^{k+2a+j-2s-2}(1+\frac {2} {x})^{-k+s-j}} {x}
\,dx\\
=
\frac {1} {2\pi\ii}\int_{\CC_x}
 {x^{k-2a+j}(2+x)^{1-2a-k+2s-j}(1+x)^{k+2a+j-2s-2}}
 \,dx.
\end{multline}

The entry in column~$j$ of the second sum in~\eqref{eq:S3} is given by
\begin{align*}
\sum_{i=s+1-a}^{k-1}&2^{2s+2-4a}\binom {i-a-1}{s-2a}D(k-2i+j,-k+s-j)\\
&\kern-6pt
=\sum_{i\ge s+1-a}^{}2^{2s+2-4a}\binom {i-a-1}{s-2a}
\frac {1} {(2\pi\ii)^2}\int_{\CC_x}\int_{\CC_y}
\frac {(1+x)^{k-2i+j}(1+y)^{-k+s-j}} {xy-2}\,dx\,dy\\
&\kern-6pt
=\frac {2^{2s+2-4a}} {(2\pi\ii)^2}\int_{\CC_x}\int_{\CC_y}
\frac {(1+x)^{k-2s-2+2a+j}(1+y)^{-k+s-j}} {xy-2}
\left(1-\frac {1} {(1+x)^2}\right)^{-s+2a-1}\,dx\,dy\\
&\kern-6pt
=\frac {2^{2s+2-4a}} {(2\pi\ii)^2}\int_{\CC_x}\int_{\CC_y}
\frac {(x(2+x))^{-s+2a-1}(1+x)^{k-2a+j}(1+y)^{-k+s-j}} {xy-2}
\,dx\,dy.
\end{align*}
Also here, we apply the residue theorem with respect to~$y$.
Again, the value $y=2/x$ is the only
singularity inside the contour $\CC_y$.
Therefore, the above integral equals
\begin{multline*}
\frac {2^{2s+2-4a}} {2\pi\ii}\int_{\CC_x}
\frac {(x(2+x))^{-s+2a-1}(1+x)^{k-2a+j}(1+\frac {2} {x})^{-k+s-j}} {x}
\,dx\\
=
\frac {2^{2s+2-4a}} {2\pi\ii}\int_{\CC_x}
{x^{k-2s+j+2a-2}(2+x)^{-k-j+2a-1}(1+x)^{k-2a+j}} 
\,dx.
\end{multline*}
We perform the substitution $x\mapsto -\frac {2(1+x)} {2+x}$.
This transforms the above integral into
$$
-\frac {2^{2s+2-4a}} {2\pi\ii}\int_{\CC'_x}
2^{-2s+4a-2}
(2+x)^{-k+2s-j -2a+1  }
(1+x)^{k-2s+j+2a-2}
     x^{k-2a+j}
\,dx,
$$
where the contour $\CC'_x$ is given by
$$
\CC'_x=\left\{-\tfrac {2(1+5e^{\ii t})} {2+5e^{\ii t}}:0\le t\le 2\pi\right\}.
$$
This contour encircles $-2$, the only singularity of the integrand
($x=-1$ is not a singularity since $k-2s+2a
-2\ge0$ by assumption),
in negative direction. Consequently, we see that the
above expression is equal to~\eqref{eq:int3}. This means that the two
sums in~\eqref{eq:S3} cancel each other, as desired.

\medskip\noindent
{\sc Step 4:} {\it  The term $({n} + k - s - \tfrac {1} {2})^{\min\{\fl{(s + 1)/2}, 
     \fl{(k - s - \chi(k\text{ \rm even}))/2}\}}$ is a factor of $D_1(k;{n})$ as a
  polynomial in~${n}$.} 
We claim that, for $0\le a< s$ and $k\ge 4s-2a-1$, we have
\begin{equation} \label{eq:S4o} 
\sum_{i=a}^{k-1}c_1(s-a,i-a)\cdot
\big(\text{row $i$ of $D_1(k;-k+2s-\tfrac {1} {2})$}\big)=0
\end{equation}
and, for $0\le a< s$ and $k\ge 4s-2a+1$, we have
\begin{equation} \label{eq:S4e} 
\sum_{i=a}^{k-1}c_2(s-a,i-a)\cdot
\big(\text{row $i$ of $D_1(k;-k+2s+\tfrac {1} {2})$}\big)=0,
\end{equation}
where
\begin{multline*}
c_1(s,l)=-\frac {(4 l - 4s+1)\,(-1)^{s-1}(1-l)_{s-1}\,(\frac {1} {2})_s\,
  (\frac {1} {2})_{l -s}} {(2 l - 4s+1)\,l!\,(s-1)!}\\
-
(4 l - 4s+1)\sum_{r=1}^s
\frac {2^{4r-3}(2r-\frac {1} {2})_{s-r}\,(2r-\frac {1} {2})_{l-s-r}}
       {(s-r)!\,(l-s-r+1)!}
\end{multline*}
and
\begin{multline*}
  c_2(s,l)=\frac {(4 l - 4s-1)\,(-1)^{s-1}(1-l)_{s-1}
   \,(\frac {1} {2})_{s+1}\,(\frac {1} {2})_{l -s- 1}}
  {(2 l - 4s-1)\,l!\,(s-1)!}\\
  -
  (4 l - 4s-1)\sum_{r=1}^s\frac {2^{4r-1}
   (2r+\frac {1} {2})_{s-r}\,(2r+\frac {1} {2})_{l-s-r-1}}
   {(s-r)!\,(l-s-r)!}.
\end{multline*}

We begin with the proof of \eqref{eq:S4o}. We must show that,
for all integers $k,j,s,a$ with $0\le a<s$, $k\ge 4s-2a-1$
and $j\ge 0$, we have
\begin{equation} \label{eq:S5o}
\sum_{i=a}^{k-1}c_1(s-a,i-a)\cdot
D(k-2i+j,-k+2s-\tfrac {3} {2}-j)=0.
\end{equation}
Working towards a simplification, we observe that it suffices to prove
\eqref{eq:S5o} for $j=0$. Furthermore, after the replacements
$s\mapsto s+a$, $i\mapsto i+a$ and $k\mapsto k+2a$, we see that
to prove \eqref{eq:S5o} for $j=0$ is equivalent with showing that,
for all integers $k$ and $s$ with $0<s$ and $k\ge 4s-1$,  we have
\begin{equation} \label{eq:S6o}
\sum_{i=0}^{k-1}c_1(s,i)\cdot
D(k-2i,-k+2s-\tfrac {3} {2})=0.
\end{equation}
Explicitly, using the definition of $c_1(s,i)$ and the
formula~\eqref{eq:f3} for $D(k,{n-1})$, this means to prove the identity
\begin{multline} \label{eq:S7o}
  \sum_{i=0}^{k-1}\sum_{\ell=0}^{k-2i}
  \frac {(4 i - 4s+1)\,(-1)^{s-1}(1-i)_{s-1}\,(\frac {1} {2})_s\,
  (\frac {1} {2})_{i -s}} {(2 i - 4s+1)\,i!\,(s-1)!}
\binom {k-2i}\ell\binom {-k+2s-\tfrac {3} {2}}\ell2^\ell\\
+ \sum_{i=0}^{k-1}
(4 i - 4s+1)\sum_{r=1}^s
\frac {2^{4r-3}(2r-\frac {1} {2})_{s-r}\,(2r-\frac {1} {2})_{i-s-r}}
       {(s-r)!\,(i-s-r+1)!}
D(k-2i,-k+2s-\tfrac {3} {2})=0,
\end{multline}
for all integers $k$ and $s$ with $0<s$ and $k\ge 4s-1$.
Here, given the formula \eqref{eq:f3} for $D(k-2i,-k+2s-\tfrac {3}
{2})$, the second sum is a triple sum. We are going to simplify that
sum to a double sum. Namely, using the contour integral
representation~\eqref{eq:f-int} with the contours again as
in~\eqref{eq:contour}, we have
\begin{align*}
\sum_{i=0}^{k-1}&
(4 i - 4s+1)
\frac {(2r-\frac {1} {2})_{i-s-r}}
      {(i-s-r+1)!}
D(k-2i,-k+2s-\tfrac {3} {2})\\
&\kern-4pt
=\sum_{i\ge s+r-1}(4 i - 4s+1)
\frac {(2r-\frac {1} {2})_{i-s-r}}
       {(i-s-r+1)!}\frac {1} {(2\pi\ii)^2}\int_{\CC_x}\int_{\CC_y}
       \frac {(1+x)^{k-2i}(1+y)^{-k+2s-\frac {3} {2}}} {xy-2}\,dx\,dy\\
&\kern-4pt
=\sum_{i\ge s+r}
\frac {4\,(2r-\frac {1} {2})_{i-s-r}}
       {(i-s-r)!}\frac {1} {(2\pi\ii)^2}\int_{\CC_x}\int_{\CC_y}
       \frac {(1+x)^{k-2i}(1+y)^{-k+2s-\frac {3} {2}}} {xy-2}\,dx\,dy\\
&\kern1cm
+\sum_{i\ge s+r-1}
\frac {2\,(2r-\frac {3} {2})_{i-s-r+1}}
       {(i-s-r+1)!}\frac {1} {(2\pi\ii)^2}\int_{\CC_x}\int_{\CC_y}
\frac {(1+x)^{k-2i}(1+y)^{-k+2s-\frac {3} {2}}} {xy-2}\,dx\,dy\\
&\kern-4pt
=\frac {4} {(2\pi\ii)^2}\int_{\CC_x}\int_{\CC_y}
\frac {(1+x)^{k-2s-2r}
(1+y)^{-k+2s-\frac {3} {2}}} {xy-2}
  \left(1-\frac {1} {(1+x)^2}\right)^{-2r+\frac {1} {2}}\,dx\,dy\\
&\kern.5cm
+\frac {2} {(2\pi\ii)^2}\int_{\CC_x}\int_{\CC_y}
\frac {(1+x)^{k-2s-2r+2}(1+y)^{-k+2s-\frac {3} {2}}} {xy-2}
  \left(1-\frac {1} {(1+x)^2}\right)^{-2r+\frac {3} {2}}\,dx\,dy\\
&\kern-4pt
=\frac {4} {(2\pi\ii)^2}\int_{\CC_x}\int_{\CC_y}
\frac {(1+x)^{k-2s+2r-1}\big(x(2+x)\big)^{-2r+\frac {1} {2}}
(1+y)^{-k+2s-\frac {3} {2}}} {xy-2}\,dx\,dy\\
&\kern.5cm
+\frac {2} {(2\pi\ii)^2}\int_{\CC_x}\int_{\CC_y}
\frac {(1+x)^{k-2s+2r-1}\big(x(2+x)\big)^{-2r+\frac {3} {2}}
  (1+y)^{-k+2s-\frac {3} {2}}} {xy-2}\,dx\,dy.
\end{align*}
Here we used Property (2) in Footnote~\ref{foot:1} in order to be
allowed to sum over {\it all\/}~$i$ (instead of stopping at $i=k-1$).

Now we apply the residue theorem with respect to~$y$.
With our choice of the contour~$\CC_y$, the value~$-1$ is outside the
contour, making $y=2/x$ the only
singularity inside~$\CC_y$. Hence, the above expression is equal to
\begin{align*}
\frac {4} {2\pi\ii}&\int_{\CC_x}
\frac {(1+x)^{k-2s+2r-1}\big(x(2+x)\big)^{-2r+\frac {1} {2}}
(1+\frac {2} {x})^{-k+2s-\frac {3} {2}}} {x}\,dx\\
&\kern1cm
+\frac {2} {2\pi\ii}\int_{\CC_x}
\frac {(1+x)^{k-2s+2r-1}\big(x(2+x)\big)^{-2r+\frac {3} {2}}
  (1+\frac {2} {x})^{-k+2s-\frac {3} {2}}} {x}\,dx\\
&=\frac {4} {2\pi\ii}\int_{\CC_x}
{x^{k-2r-2s+1}
  (1+x)^{k-2s+2r-1}(2+x)^{-k-2r+2s-1}} {}\,dx\\
&\kern1cm
+\frac {2} {2\pi\ii}\int_{\CC_x}
{x^{k-2r-2s+2}
  (1+x)^{k-2s+2r-1}(2+x)^{-k-2r+2s}}\,dx.
\end{align*}
Here we apply the residue theorem with respect to~$x$.
Since $r\le s$ and $k\ge 4s-1$ by assumption, we have $k-2s-2r+1\ge
k-4s+1\ge0$.
Hence, the only singularity is $x=-2$. Consequently, by the shift
$x\mapsto x-2$, we obtain
\begin{align*}
4&\coef{x^{k+2r-2s}}
{(x-2)^{k-2r-2s+1}
  (x-1)^{k+2r-2s-1}} \\
&\kern1cm
+2\coef{x^{k+2r-2s-1}}
{(x-2)^{k-2r-2s+2}
  (x-1)^{k+2r-2s-1}}\\
&=(-1)^{k}\sum_{\ell=0}^{k-2r-2s+1}2^{k-2r-2s+3-\ell}\binom {k-2r-2s+1}\ell 
\binom {k+2r-2s-1}{\ell-1} 
\\
&\kern1cm
+(-1)^{k}
\sum_{\ell=0}^{k-2r-2s+2}2^{k-2r-2s+3-\ell}\binom {k-2r-2s+2}\ell 
\binom {k+2r-2s-1}{\ell}\\
&=(-1)^{k}\sum_{\ell=0}^{k-2r-2s+2}2^{k-2r-2s+3-\ell}
\frac {(k-2r-2s+1)!\,(k+2r-2s-1)!}
      {\ell!^2\,(k-2r-2s+2-\ell)!\,(k+2r-2s-\ell)!}\\
&\kern4cm
\cdot(-\ell^2 + 2 k + k^2 + 4 r - 4 r^2 - 4 s - 4 k s + 4 s^2).
\end{align*}
If we substitute this back in \eqref{eq:S7o}, then we see that we have
to prove the double sum identity
\begin{multline} \label{eq:id1} 
  \sum_{i=0}^{k-1}\sum_{\ell=0}^{k-2i}
  \frac {(4 i - 4s+1)\,(-1)^{s-1}(1-i)_{s-1}\,(\frac {1} {2})_s\,
  (\frac {1} {2})_{i -s}} {(2 i - 4s+1)\,i!\,(s-1)!}
\binom {k-2i}\ell\binom {-k+2s-\tfrac {3} {2}}\ell2^\ell\\
+ 
\sum_{r=1}^s\sum_{\ell=0}^{k-2r-2s+2}(-1)^{k}
2^{k-2r-2s+3-\ell}
\frac {2^{4r-3}(2r-\frac {1} {2})_{s-r}}
      {(s-r)!}\kern5cm\\
\cdot
\frac {(k-2r-2s+1)!\,(k+2r-2s-1)!}
      {\ell!^2\,(k-2r-2s+2-\ell)!\,(k+2r-2s-\ell)!}\\
\cdot(-\ell^2 + 2 k + k^2 + 4 r - 4 r^2 - 4 s - 4 k s + 4 s^2)=0,
\end{multline}
for all integers $k$ and $s$ with $0<s$ and $k\ge 4s-1$.
With the help of Koutschan's {\sl
  Mathematica} package {\tt HolonomicFunctions} \cite{KoutAA}, this is
a routine task. The short version is that it allows us to find a
linear recurrence in~$k$ (with polynomial coefficients) of order~6
(with non-vanishing leading coefficient in the relevant domain)
that is satisfied by {\it both\/} double sums. Thus, the identity
would be proved if we are able to confirm it for six initial values,
that is, for $k=4s-1,4s,4s+1,4s+2,\break4s+3,4s+4$. For each of these six
values of~$k$, the package {\tt HolonomicFunctions} may again be used
to find a linear recurrence in~$s$ (with polynomial coefficients) of order~5
that is satisfied by {\it both\/} (specialized) double sums. Now we
are down to verifying the identity for all pairs  $(k,s)$ in
$$\{(4s+a,s):-1\le a\le 4\text{ and }1\le s\le 5\},$$
which is straightforward to do. More details are described in the appendix.

\medskip
We turn to the proof of \eqref{eq:S4e}. We must show that,
for all integers $k,j,s,a$ with $0\le a<s$, $k\ge 4s-2a+1$
and $j\ge 0$, we have
\begin{equation} \label{eq:S5e}
\sum_{i=a}^{k-1}c_2(s-a,i-a)\cdot
D(k-2i+j,-k+2s-\tfrac {1} {2}-j)=0.
\end{equation}
Working towards a simplification, we observe that it suffices to prove
\eqref{eq:S5e} for $j=0$. Furthermore, after the replacements
$s\mapsto s+a$, $i\mapsto i+a$ and $k\mapsto k+2a$, we see that
to prove \eqref{eq:S5e} for $j=0$ is equivalent with showing that,
for all integers $k$ and $s$ with $0<s$ and $k\ge 4s+1$,  we have
\begin{equation} \label{eq:S6e}
\sum_{i=0}^{k-1}c_2(s,i)\cdot
D(k-2i,-k+2s-\tfrac {1} {2})=0.
\end{equation}
Explicitly, using the definition of $c_2(s,i)$ and the
formula~\eqref{eq:f3} for $D(k,{n-1})$, this means to prove the identity
\begin{multline} \label{eq:S7e}
  \sum_{i=0}^{k-1}\sum_{\ell=0}^{k-2i}
  \frac {(4 i - 4s-1)\,(-1)^{s-1}(1-i)_{s-1}\,(\frac {1} {2})_{s+1}\,
  (\frac {1} {2})_{i -s-1}} {(2 i - 4s-1)\,i!\,(s-1)!}
\binom {k-2i}\ell\binom {-k+2s-\tfrac {1} {2}}\ell2^\ell\\
- \sum_{i=0}^{k-1}
(4 i - 4s-1)\sum_{r=1}^s
\frac {2^{4r-1}(2r+\frac {1} {2})_{s-r}\,(2r+\frac {1} {2})_{i-s-r-1}}
      {(s-r)!\,(i-s-r
        )!}
D(k-2i,-k+2s-\tfrac {1} {2})=0,
\end{multline}
for all integers $k$ and $s$ with $0<s$ and $k\ge 4s+1$.
Here, given the formula \eqref{eq:f3} for $D(k-2i,-k+2s-\tfrac {1}
{2})$, the second sum is a triple sum. We are going to simplify that
sum to a double sum. Namely, using the contour integral
representation~\eqref{eq:f-int} with the contours again as
in~\eqref{eq:contour}, we have
\begin{align*}
\sum_{i=0}^{k-1}&
(4 i - 4s-1)
\frac {(2r+\frac {1} {2})_{i-s-r-1}}
      {(i-s-r)!}
D(k-2i,-k+2s-\tfrac {1} {2})\\
&\kern-4pt
=\sum_{i\ge s+r}(4 i - 4s-1)
\frac {(2r+\frac {1} {2})_{i-s-r-1}}
       {(i-s-r)!}\frac {1} {(2\pi\ii)^2}\int_{\CC_x}\int_{\CC_y}
       \frac {(1+x)^{k-2i}(1+y)^{-k+2s-\frac {1} {2}}} {xy-2}\,dx\,dy\\
&\kern-4pt
=\sum_{i\ge s+r+1}
\frac {4\,(2r+\frac {1} {2})_{i-s-r-1}}
       {(i-s-r-1)!}\frac {1} {(2\pi\ii)^2}\int_{\CC_x}\int_{\CC_y}
       \frac {(1+x)^{k-2i}(1+y)^{-k+2s-\frac {1} {2}}} {xy-2}\,dx\,dy\\
&\kern1cm
+\sum_{i\ge s+r}
\frac {2\,(2r-\frac {1} {2})_{i-s-r}}
       {(i-s-r)!}\frac {1} {(2\pi\ii)^2}\int_{\CC_x}\int_{\CC_y}
\frac {(1+x)^{k-2i}(1+y)^{-k+2s-\frac {1} {2}}} {xy-2}\,dx\,dy\\
&\kern-4pt
=\frac {4} {(2\pi\ii)^2}\int_{\CC_x}\int_{\CC_y}
\frac {(1+x)^{k-2s-2r-2}
(1+y)^{-k+2s-\frac {1} {2}}} {xy-2}
\left(1-\frac {1} {(1+x)^2}\right)^{-2r-
  \frac {1} {2}}\,dx\,dy\\
&\kern.5cm
+\frac {2} {(2\pi\ii)^2}\int_{\CC_x}\int_{\CC_y}
\frac {(1+x)^{k-2s-2r}(1+y)^{-k+2s-\frac {1} {2}}} {xy-2}
  \left(1-\frac {1} {(1+x)^2}\right)^{-2r+\frac {1} {2}}\,dx\,dy\\
&\kern-4pt
=\frac {4} {(2\pi\ii)^2}\int_{\CC_x}\int_{\CC_y}
\frac {(1+x)^{k-2s+2r-1}\big(x(2+x)\big)^{-2r-\frac {1} {2}}
(1+y)^{-k+2s-\frac {1} {2}}} {xy-2}\,dx\,dy\\
&\kern.5cm
+\frac {2} {(2\pi\ii)^2}\int_{\CC_x}\int_{\CC_y}
\frac {(1+x)^{k-2s+2r-1}\big(x(2+x)\big)^{-2r+\frac {1} {2}}
  (1+y)^{-k+2s-\frac {1} {2}}} {xy-2}\,dx\,dy.
\end{align*}
Now we apply the residue theorem with respect to~$y$.
With our choice of the contour~$\CC_y$, the value~$-1$ is outside the
contour, making $y=2/x$ the only
singularity inside~$\CC_y$. Hence, the above expression is equal to
\begin{align*}
\frac {4} {2\pi\ii}&\int_{\CC_x}
\frac {(1+x)^{k-2s+2r-1}\big(x(2+x)\big)^{-2r-\frac {1} {2}}
(1+\frac {2} {x})^{-k+2s-\frac {1} {2}}} {x}\,dx\\
&\kern1cm
+\frac {2} {2\pi\ii}\int_{\CC_x}
\frac {(1+x)^{k-2s+2r-1}\big(x(2+x)\big)^{-2r+\frac {1} {2}}
  (1+\frac {2} {x})^{-k+2s-\frac {1} {2}}} {x}\,dx\\
&=\frac {4} {2\pi\ii}\int_{\CC_x}
{x^{k-2r-2s-1}
  (1+x)^{k-2s+2r-1}(2+x)^{-k-2r+2s-1}} {}\,dx\\
&\kern1cm
+\frac {2} {2\pi\ii}\int_{\CC_x}
{x^{k-2r-2s}
  (1+x)^{k-2s+2r-1}(2+x)^{-k-2r+2s}}\,dx.
\end{align*}
Here we apply the residue theorem with respect to~$x$.
Since $r\le s$ and $k\ge 4s+1$ by assumption, we have $k-2s-2r\ge
k-4s>0$.
Hence, the only singularity is $x=-2$. Consequently, by the shift
$x\mapsto x-2$, we obtain
\begin{align*}
4&\coef{x^{k+2r-2s}}
{(x-2)^{k-2r-2s-1}
  (x-1)^{k+2r-2s-1}} \\
&\kern1cm
+2\coef{x^{k+2r-2s-1}}
{(x-2)^{k-2r-2s}
  (x-1)^{k+2r-2s-1}}\\
&=(-1)^{k}\sum_{\ell=0}^{k-2r-2s-1}2^{k-2r-2s+1-\ell}\binom {k-2r-2s-1}\ell 
\binom {k+2r-2s-1}{\ell-1} 
\\
&\kern1cm
+(-1)^{k}
\sum_{\ell=0}^{k-2r-2s}2^{k-2r-2s+1-\ell}\binom {k-2r-2s}\ell 
\binom {k+2r-2s-1}{\ell}\\
&=(-1)^{k}\sum_{\ell=0}^{k-2r-2s}2^{k-2r-2s+1-\ell}
\frac {(k-2r-2s-1)!\,(k+2r-2s-1)!}
      {\ell!^2\,(k-2r-2s-\ell)!\,(k+2r-2s-\ell)!}\\
&\kern4cm
\cdot(-\ell^2 + k^2 - 4 r^2 - 4 k s + 4 s^2).
\end{align*}
If we substitute this back in \eqref{eq:S7e}, then we see that we have
to prove the double sum identity
\begin{multline} \label{eq:id2} 
  \sum_{i=0}^{k-1}\sum_{\ell=0}^{k-2i}
  \frac {(4 i - 4s-1)\,(-1)^{s-1}(1-i)_{s-1}\,(\frac {1} {2})_{s+1}\,
  (\frac {1} {2})_{i -s-1}} {(2 i - 4s-1)\,i!\,(s-1)!}
\binom {k-2i}\ell\binom {-k+2s-\tfrac {1} {2}}\ell2^\ell\\
+ 
\sum_{r=1}^s\sum_{\ell=0}^{k-2r-2s}(-1)^{k-1}
2^{k-2r-2s+1-\ell}
\frac {2^{4r-1}(2r+\frac {1} {2})_{s-r}}
      {(s-r)!}\kern5cm\\
\cdot
\frac {(k-2r-2s-1)!\,(k+2r-2s-1)!}
      {\ell!^2\,(k-2r-2s-\ell)!\,(k+2r-2s-\ell)!}\\
\cdot(-\ell^2 + k^2 - 4 r^2 - 4 k s + 4 s^2)=0,
\end{multline}
for all integers $k$ and $s$ with $0<s$ and $k\ge 4s+1$.
This identity can be proved in the same way as \eqref{eq:id1},
again using Koutschan's {\sl Mathematica} package {\tt
  HolonomicFunctions}. 

\medskip\noindent
{\sc Step 5:} {\it The determinant $D_1(k;{n})$ is a polynomial in ${n}$ of degree at most $\binom
  {k+1}2$.}
From \eqref{eq:f3},
we see that $D(k,{n-1})$ is a polynomial in~${n}$ of degree~$k$.
The expansion of the determinant according to its definition implies
that the degree of $D_1(k;{n})$ is at most
$$
\max_{\si\in S_k}\left(\sum_{i=0}^{k-1}(k-2i+\si(i))\right)
=\max_{\si\in S_k}\left(k^2-\sum_{i=0}^{k-1}i\right)=\frac {k(k+1)} {2}.
$$

\medskip\noindent
{\sc Step 6:} {\it The coefficient of
$n^{k(k+1)/2}$ in $D_1(k;{n})$ is
$2^{k^2}\prod _{i=1} ^{k}\frac {1} {(i)_{i}}$.}
The argument of Step~5 shows that we obtain the
leading coefficient of $D_1(k;{n})$ --- the coefficient of
$n^{k(k+1)/2}$
--- if we compute the determinant of leading
coefficients of the entries. That is, we must compute
\begin{align*}
\det_{0\le i,j\le k-1}\left(\frac {2^{k-2i+j}} {(k-2i+j)!}\right)
&=
2^{\binom {k+1}2}\prod _{i=0} ^{k-1}\frac {1} {(2k-2i-1)!}\det_{0\le i,j\le
k-1}\big((k-2i+j+1)_{k-j-1}\big)\\
&=
2^{\binom {k+1}2}\prod _{i=0} ^{k-1}\frac {1} {(2k-2i-1)!}\det_{0\le i,j\le
k-1}\big((-2i)^{k-j-1}\big)\\
&=2^{\binom {k+1}2}\prod _{i=0} ^{k-1}\frac {1} {(2k-2i-1)!}
\prod _{0\le i<j\le k-1} ^{}((-2i)-(-2j))\\
&=2^{\binom {k+1}2+\binom k2}\prod _{i=0} ^{k-1}\frac {i!} {(2k-2i-1)!}\\
&=2^{k^2}\prod _{i=1} ^{k}\frac {1} {(i)_{i}},
\end{align*}
where in the step from the first to the second line we used
elementary column operations, and the subsequent step is just the
Vandermonde determinant evaluation.
\end{proof}

    \section{Conclusion}

    In \cite{difrancesco2}, Di~Francesco initially introduces Aztec triangles as a combinatorial object for which there exists a bijection to a certain set of configurations of the 20V model, though the explicit bijection has not been found yet. Now we have generalized this construction of the Aztec triangle, as well as established a bijection to several other combinatorial objects. A natural question now is how these generalized Aztec triangles relate back to the 20V model, and if that explicit bijection can be found. 

    In particular, does there exist a domain on which 20V model configurations are in bijection with domino tilings of these generalized Aztec triangles? Of note in \cite{difrancesco2} Di~Francesco has an additional conjecture that equates the number domino tilings of the Aztec triangle with the first $k-1$ rows (consistent with Di~Francesco's definition of $\mathcal{T}_{n,n-k}$) removed is equal to 20V model configurations on particular pentagonal domains. In fact, he proves (see Section~7) this conjecture for $k=2,3$ through combinatorial means (with $k=1$ being the entire Aztec triangle). Larger values of $k$ remain open, and combinatorial proofs in these cases seem much more complicated.
    
    We can think of $\mathcal{T}_{n,n-k}$ --- the Aztec triangle with $k-1$ rows removed --- as an Aztec triangle where the first $k-1$ rows are forced to have horizontal dominoes. Then in terms of the sequences of partitions, domino tilings of $\mathcal{T}_{n,n-k}$ correspond to sequences as described in Definition~\ref{cases} with $\mu = (n,n-1,\ldots,1)$ and the additional condition that $\lambda^{(i)}_j \leq n - (k-1) + \lfloor \frac{i}{2} \rfloor$ for all~$i,j$. So while we can also interpret domino tilings of $\mathcal{T}_{n,n-k}$ as sequences of partitions, they are not directly addressed in this paper.
    
    Furthermore we can also reformulate this in terms of tableaux: domino tilings of $\mathcal{T}_{n,n-k}$ are in bijection with super symplectic semistandard tableaux of shape $(n,n-1,\ldots,1)$ with entries $1 < \overline{1} < 2 < \overline{2} < \cdots < n$ with the additional condition that the entries~$i$ and $\overline{i-1}$ do not appear past the $(n-k+i)$th column. An example of this bijection is shown in Figure~\ref{Tdomain bijection}.

    Another question to be investigated is the topic of other classes of sequences of partitions, tableaux, non-intersecting paths, and domino tilings. If we impose additional restrictions on the sequences of partitions, for example as done above with $\mathcal{T}_{n,n-k}$, how can we establish the analogous bijections, and perhaps more importantly can we still prove a formula for their enumeration? With sequences of partitions as defined in this paper, we were not able to find other choices of $\mu$ that resulted in a ``nice'' product formula, but the possibility of an even more generalized formula remains open.

    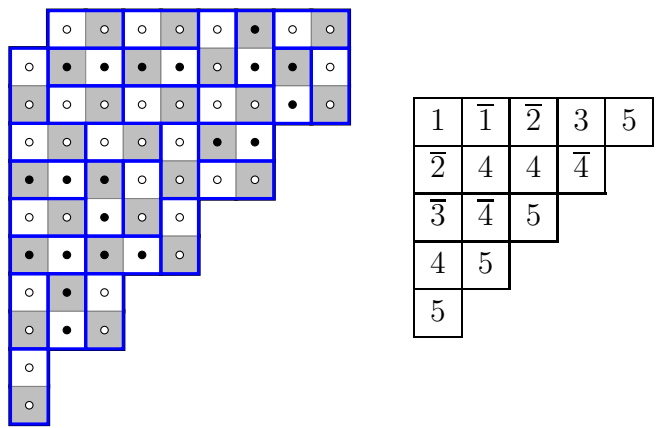
\begin{figure}[ht]
        \centering
        \[
        \vcenter{\hbox{\begin{tikzpicture}[scale=.5]
            \draw[ultra thick] (0,9) rectangle (1,10);
\draw[ultra thick] (1,10) rectangle (2,11);
\draw[ultra thick] (0,8) rectangle (1,9);
\draw[ultra thick] (1,9) rectangle (2,10);
\draw[ultra thick] (2,10) rectangle (3,11);
\draw[ultra thick] (0,7) rectangle (1,8);
\draw[ultra thick] (1,8) rectangle (2,9);
\draw[ultra thick] (2,9) rectangle (3,10);
\draw[ultra thick] (3,10) rectangle (4,11);
\draw[ultra thick] (0,6) rectangle (1,7);
\draw[ultra thick] (1,7) rectangle (2,8);
\draw[ultra thick] (2,8) rectangle (3,9);
\draw[ultra thick] (3,9) rectangle (4,10);
\draw[ultra thick] (4,10) rectangle (5,11);
\draw[ultra thick] (0,5) rectangle (1,6);
\draw[ultra thick] (1,6) rectangle (2,7);
\draw[ultra thick] (2,7) rectangle (3,8);
\draw[ultra thick] (3,8) rectangle (4,9);
\draw[ultra thick] (4,9) rectangle (5,10);
\draw[ultra thick] (5,10) rectangle (6,11);
\draw[ultra thick] (0,4) rectangle (1,5);
\draw[ultra thick] (1,5) rectangle (2,6);
\draw[ultra thick] (2,6) rectangle (3,7);
\draw[ultra thick] (3,7) rectangle (4,8);
\draw[ultra thick] (4,8) rectangle (5,9);
\draw[ultra thick] (5,9) rectangle (6,10);
\draw[ultra thick] (6,10) rectangle (7,11);
\draw[ultra thick] (0,3) rectangle (1,4);
\draw[ultra thick] (1,4) rectangle (2,5);
\draw[ultra thick] (2,5) rectangle (3,6);
\draw[ultra thick] (3,6) rectangle (4,7);
\draw[ultra thick] (4,7) rectangle (5,8);
\draw[ultra thick] (5,8) rectangle (6,9);
\draw[ultra thick] (6,9) rectangle (7,10);
\draw[ultra thick] (7,10) rectangle (8,11);
\draw[ultra thick] (0,2) rectangle (1,3);
\draw[ultra thick] (1,3) rectangle (2,4);
\draw[ultra thick] (2,4) rectangle (3,5);
\draw[ultra thick] (3,5) rectangle (4,6);
\draw[ultra thick] (4,6) rectangle (5,7);
\draw[ultra thick] (5,7) rectangle (6,8);
\draw[ultra thick] (6,8) rectangle (7,9);
\draw[ultra thick] (7,9) rectangle (8,10);
\draw[ultra thick] (8,10) rectangle (9,11);
\draw[ultra thick] (0,1) rectangle (1,2);
\draw[ultra thick] (1,2) rectangle (2,3);
\draw[ultra thick] (2,3) rectangle (3,4);
\draw[ultra thick] (3,4) rectangle (4,5);
\draw[ultra thick] (4,5) rectangle (5,6);
\draw[ultra thick] (5,6) rectangle (6,7);
\draw[ultra thick] (6,7) rectangle (7,8);
\draw[ultra thick] (7,8) rectangle (8,9);
\draw[ultra thick] (8,9) rectangle (9,10);
\draw[ultra thick] (0,0) rectangle (1,1);
\draw[ultra thick] (2,2) rectangle (3,3);
\draw[ultra thick] (4,4) rectangle (5,5);
\draw[ultra thick] (6,6) rectangle (7,7);
\draw[ultra thick] (8,8) rectangle (9,9);
\draw[fill=white,draw=gray] (0,9) rectangle (1,10);
\draw[fill=white,draw=gray] (1,10) rectangle (2,11);
\draw[fill=lightgray,draw=gray] (0,8) rectangle (1,9);
\draw[fill=lightgray,draw=gray] (1,9) rectangle (2,10);
\draw[fill=lightgray,draw=gray] (2,10) rectangle (3,11);
\draw[fill=white,draw=gray] (0,7) rectangle (1,8);
\draw[fill=white,draw=gray] (1,8) rectangle (2,9);
\draw[fill=white,draw=gray] (2,9) rectangle (3,10);
\draw[fill=white,draw=gray] (3,10) rectangle (4,11);
\draw[fill=lightgray,draw=gray] (0,6) rectangle (1,7);
\draw[fill=lightgray,draw=gray] (1,7) rectangle (2,8);
\draw[fill=lightgray,draw=gray] (2,8) rectangle (3,9);
\draw[fill=lightgray,draw=gray] (3,9) rectangle (4,10);
\draw[fill=lightgray,draw=gray] (4,10) rectangle (5,11);
\draw[fill=white,draw=gray] (0,5) rectangle (1,6);
\draw[fill=white,draw=gray] (1,6) rectangle (2,7);
\draw[fill=white,draw=gray] (2,7) rectangle (3,8);
\draw[fill=white,draw=gray] (3,8) rectangle (4,9);
\draw[fill=white,draw=gray] (4,9) rectangle (5,10);
\draw[fill=white,draw=gray] (5,10) rectangle (6,11);
\draw[fill=lightgray,draw=gray] (0,4) rectangle (1,5);
\draw[fill=lightgray,draw=gray] (1,5) rectangle (2,6);
\draw[fill=lightgray,draw=gray] (2,6) rectangle (3,7);
\draw[fill=lightgray,draw=gray] (3,7) rectangle (4,8);
\draw[fill=lightgray,draw=gray] (4,8) rectangle (5,9);
\draw[fill=lightgray,draw=gray] (5,9) rectangle (6,10);
\draw[fill=lightgray,draw=gray] (6,10) rectangle (7,11);
\draw[fill=white,draw=gray] (0,3) rectangle (1,4);
\draw[fill=white,draw=gray] (1,4) rectangle (2,5);
\draw[fill=white,draw=gray] (2,5) rectangle (3,6);
\draw[fill=white,draw=gray] (3,6) rectangle (4,7);
\draw[fill=white,draw=gray] (4,7) rectangle (5,8);
\draw[fill=white,draw=gray] (5,8) rectangle (6,9);
\draw[fill=white,draw=gray] (6,9) rectangle (7,10);
\draw[fill=white,draw=gray] (7,10) rectangle (8,11);
\draw[fill=lightgray,draw=gray] (0,2) rectangle (1,3);
\draw[fill=lightgray,draw=gray] (1,3) rectangle (2,4);
\draw[fill=lightgray,draw=gray] (2,4) rectangle (3,5);
\draw[fill=lightgray,draw=gray] (3,5) rectangle (4,6);
\draw[fill=lightgray,draw=gray] (4,6) rectangle (5,7);
\draw[fill=lightgray,draw=gray] (5,7) rectangle (6,8);
\draw[fill=lightgray,draw=gray] (6,8) rectangle (7,9);
\draw[fill=lightgray,draw=gray] (7,9) rectangle (8,10);
\draw[fill=lightgray,draw=gray] (8,10) rectangle (9,11);
\draw[fill=white,draw=gray] (0,1) rectangle (1,2);
\draw[fill=white,draw=gray] (1,2) rectangle (2,3);
\draw[fill=white,draw=gray] (2,3) rectangle (3,4);
\draw[fill=white,draw=gray] (3,4) rectangle (4,5);
\draw[fill=white,draw=gray] (4,5) rectangle (5,6);
\draw[fill=white,draw=gray] (5,6) rectangle (6,7);
\draw[fill=white,draw=gray] (6,7) rectangle (7,8);
\draw[fill=white,draw=gray] (7,8) rectangle (8,9);
\draw[fill=white,draw=gray] (8,9) rectangle (9,10);
\draw[fill=lightgray,draw=gray] (0,0) rectangle (1,1);
\draw[fill=lightgray,draw=gray] (2,2) rectangle (3,3);
\draw[fill=lightgray,draw=gray] (4,4) rectangle (5,5);
\draw[fill=lightgray,draw=gray] (6,6) rectangle (7,7);
\draw[fill=lightgray,draw=gray] (8,8) rectangle (9,9);
            \foreach\i\j in {0/2, 0/10, 5/11, 6/11, 7/10, 8/10, 2/7, 3/7, 4/8, 4/6, 0/4, 1/4, 2/4}
{
\draw[very thick,draw=blue] (\i,\j) rectangle (\i+1,\j-2);
\pgfmathparse{mod(\i+\j, 2)}
\let\parity\pgfmathresult
\ifthenelse{\equal\parity{0.0}}
{\draw[fill=white] (\i+.5,\j-.5) circle[radius=.1];
\draw[fill=white] (\i+.5,\j-1.5) circle[radius=.1];}
{\draw[fill=black] (\i+.5,\j-.5) circle[radius=.1];
\draw[fill=black] (\i+.5,\j-1.5) circle[radius=.1];}
}

\foreach\i\j in {1/10, 1/11, 3/10, 3/11, 7/11, 0/7, 5/8, 0/8, 1/9, 2/8, 3/9, 5/9, 5/7, 2/5, 0/5, 0/6}
{
\draw[very thick,draw=blue] (\i,\j) rectangle (\i+2,\j-1);
\pgfmathparse{mod(\i+\j, 2)}
\let\parity\pgfmathresult
\ifthenelse{\equal\parity{0.0}}
{\draw[fill=white] (\i+.5,\j-.5) circle[radius=.1];
\draw[fill=white] (\i+1.5,\j-.5) circle[radius=.1];}
{\draw[fill=black] (\i+1.5,\j-.5) circle[radius=.1];
\draw[fill=black] (\i+.5,\j-.5) circle[radius=.1];}
}
        \end{tikzpicture}}} \qquad \vcenter{\hbox{\begin{ytableau}
            1 & \overline{1} & \overline{2} & 3 & 5\\
            \overline{2} & 4 & 4 & \overline{4} & \none\\
            \overline{3} & \overline{4} & 5 & \none & \none\\
            4 & 5 & \none & \none & \none\\
            5 & \none & \none & \none & \none
        \end{ytableau}}}
        \]
        \caption{A domino tiling of $\mathcal{T}_{5,1}$ and its correponding super symplectic semistandard tableau.}
        \label{Tdomain bijection}
    \end{figure}


    \section{Appendix}

    \subsection{Proof of identity \texorpdfstring{\eqref{eq:id1}}{(7.28)}}

    \setcounter{equation}{0}
    \global\def\theequation{\mbox{A.\arabic{equation}}}

    We denote the summand of the first sum by $U(k,s,i,\ell)$, and the
    summand of the second sum by $V(k,s,r,\ell)$, so that,
    for all integers $k$ and $s$ with $0<s$ and
    $k\ge 4s-1$,
    we have to prove
    \begin{equation} \label{eq:S8o}
      \sum_{i=0}^{k-1}\sum_{\ell=0}^{k-2i}U(k,s,i,\ell)
    + 
    \sum_{r=1}^s\sum_{\ell=0}^{k-2r-2s+2}V(k,s,r,\ell)=0.
    \end{equation}
    
    In order to achieve this, we use Koutschan's {\sl
      Mathematica} package\break {\tt HolonomicFunctions} \cite{KoutAA} to find a recurrence
    relation that is satisfied by both double sums. In the following, we
    let $S_k$ denote the shift operator in~$k$, that is,
    $S_k\big(g(k)\big):=g(k+1)$, and similarly we shall use the shift operators
    $S_i,S_r,S_\ell$ in the variables~$i,r,\ell$, respectively.
    
    Applying the command
    $$
    \text{\tt CreativeTelescoping[Annihilator[$U(k,s,i,\ell)$,}
          \text{\{$S_k$,$S_i$,$S_\ell$\}],$S_\ell-1$,\{$S_k$,$S_i$\}]
    }
    $$
    we find operators $A_j(k,s,i,S_k,S_i)$ and $B_j(k,s,i,\ell,S_k,S_i,S_\ell)$, $j=1,2$,
    where $A_j$ is a polynomial in the given terms and $B_j$ is in
    $\mathbb Q(k,s,i,\ell)[S_k,S_i,S_\ell]$, such that
    \begin{multline*}
    A_j(k,s,i,S_k,S_i)U(k,s,i,\ell)+(S_\ell-1)B_j(k,s,i,\ell,S_k,S_i,S_\ell)U(k,s,i,\ell)=0,\\
    \text{for }j=1,2.
    \end{multline*}
    This relation is then summed over~$\ell$ in the range $0\le \ell\le k-2i$.
    Since this range has natural boundaries --- in the sense that we
    could equally well sum over all integers~$\ell$ because the
    boundaries are caused by the binomial coefficient $\binom {k-2i}\ell$ ---
    the sum over~$\ell$ of the second term is a telescoping sum, which drops
    out. Thus we arrive at
    \begin{equation*} 
    A_j(k,s,i,S_k,S_i)\sum_{\ell=0}^{k-2i}U(k,s,i,\ell)=0,\quad \text{for }j=1,2.
    \end{equation*}
    Next we use the command
    $$
    \text{\tt CreativeTelescoping[$\{A_1(k,s,i,S_k,S_i),A_2(k,s,i,S_k,S_i)\}$,$S_i-1$,$\{S_k\}$]}
    $$
    to find operators $C(k,s,S_k)$ and $D(k,s,i,S_k,S_i)$,
    where $C$ is a polynomial in the given terms and $D$ is in
    $\mathbb Q(k,s,i)[S_k,S_i]$, such that
    $$
    C(k,s,S_k)+(S_i-1)D(k,s,i,S_k,S_i)
    $$
    lies in the left ideal generated by $A_1(k,s,i,S_k,S_i)$ and
    $A_2(k,s,i,S_k,S_i)$ in\break $\mathbb Q(k,s,i)[S_k,S_i]$.
    In particular, we have
    $$
    C(k,s,S_k)\sum_{\ell=0}^{k-2i}U(k,s,i,\ell)+(S_i-1)D(k,s,i,S_k,S_i)
    \sum_{\ell=0}^{k-2i}U(k,s,i,\ell)=0.
    $$
    If we sum this relation over all~$i$ in the range $0\le i\le k-1$ ---
    which is again a range with
    natural boundaries in the above sense --- then we obtain
    \begin{equation} \label{eq:Rek-U} 
    C(k,s,S_k)\sum_{i=0}^{k-1}\sum_{\ell=0}^{k-2i}U(k,s,i,\ell)=0,
    \end{equation}
    where the operator $C(k,s,S_k)$ has the form
    $$
    C(k,s,S_k)=\sum_{l=0}^6\gamma_l(k,s)S_k^l,
    $$
    with leading coefficient
    \begin{multline} \label{eq:lead}
    \gamma_6(k,s)=(k+6) (k - 4 s+7) (2 k - 4 s+7) (2 k - 4 s+11)
    (2 k -  4 s+13)\\
    \times
    (k - 2 s+3) (k - 2 s+4) (k - 2 s+6) 
    (425613 + 
    518672 k + 230896 k^2\\
    + 44800 k^3 + 3200 k^4
    - 1084280 s - 
    946880 k s - 272128 k^2 s\\ - 25600 k^3 s + 1048112 s^2 
    + 
       590336 k s^2 + 82432 k^2 s^2
       - 387328 s^3 \\
       - 96256 k s^3 + 
       4096 k^2 s^3 - 44288 s^4 - 45056 k s^4 \\
       - 4096 k^2 s^4 
       + 
       67584 s^5 + 16384 k s^5 - 12288 s^6).
    \end{multline}
    We observe that $\gamma_6(k,s)$ is non-zero for $k\ge 4s-1$ since the larger
    irreducible factor is visibly an odd integer and thus non-zero.
    
    Now we turn to the second double sum in \eqref{eq:S8o}. The arguments
    will be more involved here since the sum over~$r$ has no natural lower
    boundary. Still, we start in the same manner. We apply the command
    $$
    \text{\tt CreativeTelescoping[Annihilator[$V(k,s,r,\ell)$,}
          \text{\{$S_k$,$S_r$,$S_\ell$\}],$S_\ell-1$,\{$S_k$,$S_r$\}]
    }
    $$
    to find operators $E_j(k,s,r,S_k,S_r)$ and $F_j(k,s,r,\ell,S_k,S_r,S_\ell)$, $j=1,2$,
    where $E_j$ is a polynomial in the given terms and $F_j$ is in
    $\mathbb Q(k,s,r,\ell)[S_k,S_r,S_\ell]$, such that
    \begin{multline*}
      E_j(k,s,r,S_k,S_r)V(k,s,r,\ell)+(S_\ell-1)F_j(k,s,r,\ell,S_k,S_r,S_\ell)V(k,s,i,\ell)=0,\\
    \text{for }j=1,2.
    \end{multline*}
    Subsequently, we sum this relation over~$\ell$. The summation over~$\ell$ does have natural
    boundaries due to the terms~$\ell!$ and $(k-2r-2s+2-\ell)!$ in the
    denominator of~$V(k,s,r,\ell)$. Hence, we obtain
    \begin{equation*} 
    E_j(k,s,r,S_k,S_r)\sum_{\ell=0}^{k-2r-2s+2}V(k,s,r,\ell)=0,\quad \text{for }j=1,2.
    \end{equation*}
    Next we use the command
    $$
    \text{\tt CreativeTelescoping[$E(k,s,r,S_k,S_r$),$S_r-1$,$\{S_k\}$]}
    $$
    to find operators $G(k,s,S_k)$ and $H(k,s,r,S_k,S_r)$,
    where $G$ is a polynomial in the given terms and $H$ is in
    $\mathbb Q(k,s,r)[S_k,S_r]$, such that
    $$
    G(k,s,S_k)+(S_r-1)H(k,s,r,S_k,S_r)
    $$
    lies in the left ideal generated by $E_1(k,s,r,S_k,S_r)$ and
    $E_2(k,s,i,S_k,S_i)$ in\break $\mathbb Q(k,s,r)[S_k,S_r]$.
    In particular, we have
    \begin{equation} \label{eq:S11o} 
    G(k,s,S_k)\sum_{\ell=0}^{k-2r-2s+2}V(k,s,r,\ell)+(S_r-1)H(k,s,r,S_k,S_r)
    \sum_{\ell=0}^{k-2r-2s+2}V(k,s,r,\ell)=0.
    \end{equation}
    If we sum this relation over all~$r$ in the range $1\le r\le s$, then only the
    upper boundary is natural due to the term $(s-r)!$ in the denominator
    of~$V(k,s,r,\ell)$. On the other hand, the sum over~$r$ of the second term
    in~\eqref{eq:S11o} is still a telescoping sum. Even if it does not vanish
    entirely, it is only a term for $r=1$ (with a negative sign)
    that survives. More precisely, we obtain
    \begin{equation} \label{eq:Rek-V} 
    G(k,s,S_k)\sum_{r=1}^{s}\sum_{\ell=0}^{k-2r-2s+2}V(k,s,r,\ell)
    -\sum_{\ell=0}^{k-2s}H(k,s,1,S_k,S_r)V(k,s,1,\ell)
    =0.
    \end{equation}
    One simplifies $H(k,s,1,S_k,S_r)V(k,s,1,\ell)$ to
    \begin{multline} \label{eq:S12o}
    H(k,s,1,S_k,S_r)V(k,s,1,\ell)=
    \frac {(-1)^{k} 2^{
        2 - \ell + k - 
        2 s} (k - 2 s+3)^2 (4 s-1)^2 \text{Pol}(k,s,\ell)
    }
       {(2 k - 4 s+1) (2 k - 4 s+3) (k - 2 s+1)^2 }	
    \\
    \times
    \frac {(k - 2 s-1)!\, (k - 2 s+1)!\, (\frac {3} {2})_{s-1}}
                {\ell!^2 \,(k - 2 s-\ell)!\, (k - 2 s-\ell+4)! (s-1)!},
    \end{multline}
    where $\text{Pol}(k,s,\ell)$ is a huge polynomial in $k,s,\ell$.
    In principle, the sum over~$\ell$ of this expression can be evaluated by
    means of Chu--Vandermonde summation. However, it is more effective to
    use again Koutschan's {\tt HolonomicFunctions} in the form
    $$
    \text{\tt CreativeTelescoping[Annihilator[$W(k,s,\ell)$,$\{S_k$,$S_\ell\}$],$S_\ell-1$,$\{S_k\}$]}
    $$
    where $W(k,s,\ell)$ is the right-hand side in~\eqref{eq:S12o}, to find
    operators $I(k,s,S_k)$ and\break $J(k,s,\ell,S_k,S_\ell)$ such that
    \begin{equation*} 
    I(k,s,S_k)W(k,s,\ell)+(S_\ell-1)J(k,s,\ell,S_k,S_\ell)W(k,s,\ell)=0.
    \end{equation*}
    Summing this over~$\ell$, we obtain
    \begin{equation*}
    I(k,s,S_k)\sum_{\ell=0}^{k-2s}W(k,s,\ell)=0.
    \end{equation*}
    Consequently, if we apply the operator $I(k,s,S_k)$ on both sides
    of~\eqref{eq:Rek-V}, we arrive at the desired recurrence for the second
    double sum in~\eqref{eq:S8o},
    \begin{equation} \label{eq:Rek-V2}
    I(k,s,S_k)G(k,s,S_k)\sum_{r=1}^{s}\sum_{\ell=0}^{k-2r-2s+2}V(k,s,r,\ell)=0.
    \end{equation}
    Using the command {\tt OreTimes} of the package {\tt
      HolonomicFunctions}, one finally sees that
    \begin{multline*}
    I(k,s,S_k)G(k,s,S_k)=
    (18237 + 21392 k + 9328 k^2 + 1792 k^3 + 128 k^4 - 45944 s\\
    - 
       39104 k s - 11008 k^2 s - 1024 k^3 s + 34352 s^2 + 18944 k s^2 + 
       2560 k^2 s^2 \\
       + 42752 s^3 
    + 26624 k s^3 + 4096 k^2 s^3 - 
    105728 s^4 - 45056 k s^4 \\
    - 4096 k^2 s^4 + 67584 s^5 + 
    16384 k s^5 - 12288 s^6)\\
    \times
    (50877 + 45936 k + 15472 k^2 + 2304 k^3 + 
    128 k^4 - 97080 s - 64192 k s \\
    - 14080 k^2 s - 1024 k^3 s + 
    55856 s^2 + 24064 k s^2 + 2560 k^2 s^2 \\
    + 73472 s^3 + 34816 k s^3 + 
    4096 k^2 s^3 - 154880 s^4 - 53248 k s^4\\
    - 4096 k^2 s^4 + 
       83968 s^5 + 16384 k s^5 - 12288 s^6)
    \times C(k,s,S_k),
    \end{multline*}
    so that $I(k,s,S_k)G(k,s,S_k)$ is a multiple of the recurrence
    operator for the first double sum in~\eqref{eq:S8o}. However, the
    multiplicative factor is visibly an odd integer and can thus be
    cancelled in~\eqref{eq:Rek-V2}. We have therefore proved that both
    double sums in~\eqref{eq:S8o} satisfy the same recurrence relation
    in~$k$ (of order 6) for $k>4s-1$, namely 
    \begin{equation*} 
    \sum_{l=0}^6\gamma_l(k,s)F(k+l,s)=0,
    \end{equation*}
    with the coefficients $\gamma_l(k,s)$ being defined in~\eqref{eq:S6e}.
    
    \medskip
    For verifying \eqref{eq:id1} for $k=4s-1,4s,4s+1,4s+2,4s+3,4s+4$, one
    proceeds in exactly the same way. The only difference is that, now,
    one searches for recurrences in~$s$. Hence, in the corresponding
    computations that closely follow those of the previous
    paragraphs, instead of $S_k$ one has to put $S_s$ everywhere, together
    with a few other small modifications. It turns out that, for each of
    the above specializations of~$k$, both double sums in~\eqref{eq:id1}
    satisfy a linear recurrence of order~5 (obviously, a different one in
    each case).

\end{document}